%% file: template.tex
\documentclass{article}

\usepackage{arxiv}

\usepackage[utf8]{inputenc} 
\usepackage[T1]{fontenc}    
\usepackage{hyperref}       
\usepackage{url}            
\usepackage{booktabs}       
\usepackage{amsfonts}       
\usepackage{nicefrac}       
\usepackage{microtype}      
\usepackage{lipsum}
\usepackage{graphicx}

\usepackage{natbib}
\usepackage{algorithmic}
\usepackage{algorithm}
\usepackage{amsmath,amssymb}
\usepackage{amsthm}
\usepackage{bm}
\usepackage{mathrsfs}
\usepackage{subcaption}

\newtheorem{dfn}{Definition}
\newtheorem{prop}{Proposition}
\theoremstyle{remark}
\newtheorem{rem}{Remark}
\newtheorem{coro}{Corollary}
\newtheorem{ass}{Assumption}

\newcommand{\argmin}{\operatornamewithlimits{argmin}}
\newcommand{\argmax}{\operatornamewithlimits{argmax}}

\title{Flexible and Reliable Network Design \\ for Emerging Transportation Services: \\ Multi-stage Stochastic Programming Approach}

\author{
 Riki Kawase \thanks{Corresponding author.} \\
  Institute of Science Tokyo \\
  \texttt{kawase.r.ac@m.titech.ac.jp} \\\\
 Koki Satsukawa \\
Tohoku University \\
\texttt{satsukawa@tohoku.ac.jp} \\\\
 Toru Seo \\
Institute of Science Tokyo \\
\texttt{seo.t.aa@m.titech.ac.jp} \\  
}

\begin{document}
\maketitle
\begin{abstract}
This paper proposes a general framework for flexible and reliable network design problems (FR-NDPs).
The framework enables planners to change infrastructure investments in response to realized uncertainties, while ensuring desired levels of reliability.
Motivated by emerging transportation services such as shared autonomous vehicle (SAV) systems, where historical data are scarce and technological developments uncertain, FR-NDPs integrate strategic investment decisions with operational control.
We formulate the FR-NDPs as risk-averse multi-stage stochastic problems to be solvable by stochastic dual dynamic programming (SDDP) and establish sufficient conditions under which strategic and operational subproblems converge to the global optimum.
We illustrate applications to SAV capacity expansion and integrated SAV-BRT (Bus Rapid Transit) route design, and numerical experiments on a Midtown Manhattan network highlight three key findings: (i) flexibility and reliability act complementarily to hedge against severe scenarios while mitigating the loss of expected performance; (ii) flexibility in investment planning allows dynamic risk hedging, with risk-averse planners reducing early-stage investments to preserve adaptability; and (iii) differences in operational flexibility between SAV and BRT systems are reflected in strategic decisions, with risk-averse planners tending to refrain investment in transport modes with lower operational flexibility.
\end{abstract}


\input{ch1.tex}
\input{ch2.tex}
\input{ch2-1.tex}

\input{ch2-4.tex}

\input{ch2-5.tex}

\input{ch3-1.tex}
\input{ch3-2.tex}
\input{ch3-3.tex}
\input{ch3-4.tex}

\input{ch4-1.tex}

\input{ch4-3.tex}
\input{ch5.tex}

\input{ch5-1.tex}
\input{ch5-2.tex}

\input{ch5-3.tex}

\input{ch6.tex}

\section*{Acknowledgements}
This study was supported by JSPS Grant-in-Aid (KAKENHI) \#24K01002 and \#24K00999.

\appendix
\input{chx.tex}

\input{chxx.tex}

\input{chxxx.tex}

\input{chxxxx.tex}
\input{chxxxxx.tex}

\input{chxxxxxx.tex}

\bibliographystyle{abbrvnat0}
\bibliography{reference}

\end{document}

%% file: ch1.tex
\section{Introduction}

Network design problems (NDPs) are central to addressing inevitable uncertainties in long-term transportation infrastructure investment planning.
Such uncertainties include strategic-level uncertainties over years, such as travel demand growth and infrastructure deterioration, as well as operational-level uncertainties over shorter time units (e.g., within-day), such as fluctuations in daily demand and energy costs.
Most previous NDPs consider only operational-level uncertainties (particularly daily demand) to find optimal network improvements in a single period \citep{govindan2017supply}.
Although the resulting network design can alleviate risks such as excessive operational costs (e.g., total travel time) under uncertain daily demand spikes, 
society's increasing dependence on mobility requires sequential network improvements to satisfy demand at a desired service level.
Since such improvements usually involve significant and irreversible investments, it is important to find the long-term investment plan that optimally addresses estimated future demand growth.
\citet{goetz1997revisiting} and \citet{flyvbjerg2005accurate} reported that the classical practice of using point estimation based on historical data leads to significant economic losses.
These highlight the necessity of incorporating both strategic- and operational-level uncertainties into NDPs to capture the comprehensive risks inherent in long-term transportation planning.

Recent growth of emerging transportation services such as shared autonomous vehicle (SAV) systems (\citeauthor{narayanan2020shared}, \citeyear{narayanan2020shared}) motivates planners to re-envision network design.
This is because emerging services possess the capability of within-day dynamic vehicle control that exploits real-time observations of operational uncertainties.
Such dynamic operations alleviate the impacts of operational uncertainties, resulting in a reduction in required infrastructure investments to sustain performance even under daily demand spikes \citep{hua2019joint}.
Meanwhile, within-day operations alone are not enough to resolve the impacts of strategic uncertainties.
This can be supported by evidence that the entry of Transportation Network Companies led to more severe congestion in the US (cf. \citeauthor{diao2021impacts}, \citeyear{diao2021impacts}).
Moreover, emerging services usually require starting from scratch with little or no reliable data, and the future development of enabling technologies (e.g., vehicle electrification and automation) also remains highly uncertain, further amplifying strategic risks.
Therefore, toward emerging transportation services, it is urgent to revisit NDPs that incorporate within-day dynamic operations in order to address strategic and operational uncertainties.


Mitigating the impacts of uncertainties requires NDPs to complementarily incorporate reliability and flexibility \citep{vishnu2021development}.
\textit{Reliability}\footnote{Reliability of transportation networks must be distinguished into two dimensions: connectivity and time (cost) reliability \citep{asakura1991road}.
The former is generally used to evaluate reliability in abnormal situations (e.g., \citeauthor{nakayama2024topological}, \citeyear{nakayama2024topological});
the latter is appropriate for evaluating usual traffic conditions.
This study focuses on the latter, i.e., normal operating conditions where network connectivity remains intact.}
is the probability that a system achieves a specified performance \citep{chen2011transport}.
Reliability-based NDPs either maximize a reliability measure or maximize performance while introducing reliability as chance constraints.
The resulting network designs have risk-averse capabilities to hedge against severe scenarios at the sacrifice of expected performance.
\textit{Flexibility} is the ability of a system to adapt to external changes, while maintaining satisfactory system performance \citep{morlok2004measuring}.
Emerging transportation services can accommodate within-day uncertainty through dynamic vehicle control,
improving operations managers' flexibility in the sense of the taxonomy seen in \citet{ukkusuri2009multi}.
Similarly, providing strategic planners with flexibility requires a sequential decision-making process to change, delay, or abandon infrastructure investments from period to period (e.g., every year) based on realized uncertain strategic parameters \citep{ukkusuri2009multi}.
Traditional single-stage NDPs, which optimize network design at the beginning of all periods, cannot ensure such strategic flexibility; instead, this study addresses multi-stage NDPs that aim at optimizing sequencing network investments over time.
As \citet{vishnu2021development} pointed out, flexibility serves as a reactive capability to complement proactive counterparts like reliability.
In other words, incorporating both dimensions into an NDP enables planners to mitigate a reduction in expected performance while hedging against severe scenarios.

Unfortunately, flexible and reliable network design problems (FR-NDPs) face the following three methodological challenges: (1) \textit{curse-of-horizon}, (2) \textit{optimality}, and (3) \textit{time consistency}.
\begin{enumerate}
    \setlength{\itemsep}{0pt} \setlength{\parskip}{0pt}
    \item Evaluating reliability requires a probability distribution of {the optimal solution's performance}.
    Scenario-based stochastic programming is a typical optimization technique to find the optimal solution.
    However, the number of scenarios grows exponentially with the number of decision-making stages.
    This exponential growth renders the memory required by scenario-based stochastic programming computationally impracticable---known as the curse-of-horizon  \citep{fullner2025stochastic}---thereby forcing FR-NDPs to assume only a small number of possible uncertain levels and stages.
    \item 
    Recent studies have focused on scalable approaches, such as Approximate Dynamic Programming and Markov Decision Processes, for finding approximate solutions to multi-stage NDPs (e.g., \citeauthor{yoon2024sequential}, \citeyear{yoon2024sequential}; \citeauthor{rath2024deep}, \citeyear{rath2024deep}).
    Instead of enumerating scenarios, these approaches exploit Bellman's principle and recursively approximate the impact of future uncertainties in the form of state value functions, ensuring that the memory requirements grow only linearly with the number of stages.
    However, in general, the optimality of solutions is not theoretically guaranteed.
    This means that the approximate solutions may involve a risk of falling below the reliability desired by planners.
    \item FR-NDPs can be classified into a class of risk-averse sequential decision problems.
    Several authors have highlighted that time consistency is a desirable property for such a class of problems (e.g., \citeauthor{homem2016risk}, \citeyear{homem2016risk}).
    In the context of optimization\footnote{The definitions of time consistency in the literature differ mainly by their focus.
   	We focus on time consistency as defined in the optimization-oriented papers (e.g., \citeauthor{homem2016risk}, \citeyear{homem2016risk}; \citeauthor{dowson2025incorporating}, \citeyear{dowson2025incorporating}). },
    time consistency means that the optimal policy should exclude incentives to deviate from the ex-ante optimal policy at later dates \citep{schur2019time}.
    Informally, time-consistent planning ensures that decisions made today should agree with past plans for the scenario that actually occurred. 
    This notion can be interpreted as a kind of stability property, which is critical when formulating strategy-proof pricing schemes to recover the investment costs of infrastructure improvements \citep{lo2009time}.
    However, time consistency is not generally satisfied in reliable NDPs.
    For example, $\alpha$-reliable NDPs \citep{chen2007alpha} adopt the $(1-\alpha)$ percentile of performance as a risk measure, which has been pointed out as time-inconsistent \citep{kovacevic2009time}.
\end{enumerate}


To overcome these challenges simultaneously, we leverage a multi-stage stochastic programming (MSSP) approach.
An MSSP models sequential decision-making under uncertainty in a decision-observation-decision loop.
At each stage, a planner observes the current system state and the uncertainties that have just been realized (e.g., demand growth, daily demand, energy prices), makes decisions (e.g., capacity expansions, transit route extensions, fleet sizing/control) with the information currently available, and carries the system forward.
This structure is suitable for FR-NDPs, where long-term strategic decisions interleave with operational responses because (a) investments are irreversible and must be staged over a long horizon, so plans need the option to redirect later; (b) key components, such as demand, technology, and prices, unfold gradually, so investments must adapt as evidence accumulates; (c) reliability depends on the distribution of outcomes over time.
The MSSP approach also provides a tractable basis for stating global optimality and for embedding time-consistent risk measures, so that plans made earlier remain optimal once uncertainties resolve.

To solve the MSSP efficiently while meeting these requirements, we utilize stochastic dual dynamic programming (SDDP) \citep{pereira1991multi,shapiro2011analysis}.
SDDP applies Bellman’s principle to decompose the horizon into stagewise subproblems and iteratively constructs piecewise linear approximations of future cost (value) functions from dual information.
This algorithm avoids explicit scenario enumeration---addressing the curse-of-horizon---and, under mathematically tractable conditions such as convexity, provides the global optimum of the MSSP \citep{girardeau2015convergence,guigues2016convergence,guigues2023duality}. Furthermore, risk-averse extensions of SDDP accommodate time-consistent risk measures, preserving these guarantees and enabling reliable network design \citep{khatami2024risk,dowson2025incorporating}.


Building on the MSSP formulation, this study establishes a general framework to address FR-NDPs for emerging transportation services with both theoretical and computational tractability.
The proposed framework is designed to enable the incorporation of flexibility and reliability into conventional single-stage NDPs with only slight customizations.
Specifically, it alternately optimizes within-day dynamic vehicle control and period-to-period investment, while exploiting observations of realized operational and strategic uncertainties.
Each strategic and operational problem corresponds to decomposed subproblems of risk-averse MSSPs amenable to an SDDP-type algorithm, which guarantees the global optimality of the resulting network design.

The proposed framework is applied to two problems: (i) capacity expansion planning for SAV systems and (ii) sequential route design in integrated SAV systems with Bus Rapid Transit (BRT). 
In the former application, the mathematical tractability of the MSSP approach enables the derivation of theoretical relations between optimal infrastructure costs and toll strategies.
Finally, numerical experiments with the Midtown Manhattan case reveal the impacts of flexibility and reliability on strategic investments and their performance in comparison with risk-neutral single-stage NDPs.

Our contributions are listed as follows:
\begin{itemize}
    \setlength{\itemsep}{0pt} \setlength{\parskip}{0pt}
    \item We propose a general formulation of FR-NDPs solvable using SDDP with computationally practical memory requirements, even under a combinatorial explosion of scenarios.
    The proposed FR-NDPs can alleviate the impacts of significant and inevitable strategic uncertainties in the infrastructure investment planning process for emerging transportation services.
    \item We provide two types of sufficient conditions for strategic and operational subproblems to ensure convergence properties of FR-NDPs:
    (i) conditions ensuring that the deterministic lower bound for the optimal value almost surely converges to the global optimum and
    (ii) conditions enabling the computation of statistical and deterministic upper bounds for the optimal value.
    {The key contribution of these conditions is to dispense with a stagewise feasibility assumption that is standard in SDDP convergence theory but easily violated in NDPs.}
    \item Theoretical results on the self-financing principle for FR-NDPs are proved in the risk-neutral case of SAV systems {via the strong duality of multi-stage stochastic linear problems.}
    \item Our numerical experiments with a real-world case study demonstrate that
    {multi-stage planning and risk-averse evaluation can
    act complementarily in hedging severe scenarios while mitigating the loss of expected performance.}
\end{itemize}

The remainder of this paper is organized as follows.
Section \ref{sec:2} reviews related work on network design problems and stochastic dual dynamic programming.
Section \ref{sec:3} introduces the general formulation of the proposed FR-NDP framework, describes its theoretical properties, and provides sufficient conditions for convergence and optimality.
Section \ref{sec:4} presents two applications---to shared autonomous vehicle (SAV) capacity expansion and to sequential route design in integrated SAV-BRT systems---illustrating how to embed existing strategic decision models within the proposed framework.
Section \ref{sec:5} reports the results of numerical experiments on the Midtown Manhattan network, highlighting the impacts of flexibility and reliability on investment patterns, operational performance, and robustness under uncertainty.
Finally, Section 6 concludes the paper and discusses key directions.

%% file: ch2.tex
\section{Literature review}
\label{sec:2}

%% file: ch2-1.tex
\subsection{Network design problem}
\label{sec:review1}

NDPs aim to determine the optimal locations of new facilities or the capacity expansion of existing facilities within a transportation network \citep{friesz1985transportation}.
Over the years, NDPs have been classified along several dimensions (cf. \citeauthor{ukkusuri2009multi}, \citeyear{ukkusuri2009multi}; \citeauthor{farahani2013review}, \citeyear{farahani2013review}). 
Based on the nature of design decisions, NDPs can be classified into three groups: (1) discrete NDP which deals with discrete design decisions such as constructing new facilities, (2) continuous NDP is concerned with continuous design decisions such as capacity expansion of existing facilities, and (3) mixed NDP which contains a combination of discrete and continuous decisions.
Based on demand assumptions, NDPs are grouped into: (1) fixed demand NDP, (2) elastic demand NDP, (3) stochastic demand NDP, and (4) stochastic-elastic demand NDP. 
Based on the planning horizon, NDPs can be classified into: (1) single-stage NDP (one-shot investment), and (2) multi-stage NDP (multi-period investment planning).
The problem in this paper focuses on multi-stage NDPs under a broad class of uncertainties, including stochastic demand.

While numerous studies have addressed either multi-stage NDPs \citep{lo2004planning,szeto2006transportation,szeto2008time,lo2009time,hosseininasab2015integration,chen2016optimal,kumar2018simplified} or 
stochastic NDPs \citep{waller2001stochastic,chen2007alpha,lou2009robust,yin2009robust,chung2011robust,szeto2016reliable,chen2025sustainable}, few studies consider both aspects.
\citet{ahmed2003multi} developed a multi-stage investment model for capacity expansion under uncertainty.
A scenario-based stochastic integer programming problem is formulated to capture the evolution of uncertain demand and cost parameters.
\citet{ukkusuri2009multi} proposed the flexible network design problem that simultaneously handles multi-period investment sequencing and stochastic-elastic demand.
They express capacity expansion decisions in a scenario tree that allows planners to delay, modify, or abandon future investments in response to demand realizations, thereby capturing temporal flexibility.
However, the scenario-based approach leads to an exponential increase in problem size with respect to the length of the planning horizon.
Numerical experiments conducted in \citet{ukkusuri2009multi} evaluated a three-stage model with five types of branches at each stage.
Recently, scalable approximation methods such as reinforcement learning have been developed as alternatives to the scenario-based approach.
\citet{yoon2020contextual} proposed a reinforcement learning-based method to support sequential transit network design planning, where routes are deployed progressively in response to observed uncertain passenger demand.
A contextual multi-armed bandit algorithm is embedded within the route selection model to sequentially select the optimal route from a pre-enumerated set.
\citet{yoon2024sequential} addressed a similar problem by developing a knowledge gradient algorithm with correlated beliefs to capture the correlation between flows.
\citet{rath2024deep} proposed a machine learning-based framework for planning service area expansions in mobility-on-demand systems.
A case study in Brooklyn demonstrated that leveraging a recurrent neural network–based classifier to remove combinatorial complexity yielded a marked reduction in computation time.

{Distinct from the research stream aimed at computing optimal network designs, a recurring theoretical concept in the network design literature is the self-financing principle.
The principle characterizes the relationship between infrastructure capital costs and the revenue collected through congestion pricing.
Since the seminal work of \citet{mohring1962highway}, the principle has been extended to stochastic settings \citep{lindsey2009cost,lindsey2014cost} and to multi-stage network design problems \citep{lo2009time}.
More recently, the principle has been applied to emerging transportation services.
\citet{seo2024dynamic} proved that, in shared autonomous vehicle (SAV) systems, the revenue collected through optimal congestion tolls charged to SAVs coincides with the optimal infrastructure capital cost.
However, it remains unclear whether an analogous relation continues to hold when infrastructure investment and operational decisions are made sequentially under multi-stage uncertainty.}

%% file: ch2-4.tex
\subsection{Stochastic dual dynamic programming}
\label{sec:review2}

Stochastic Dual Dynamic Programming (SDDP) is a dynamic programming-inspired algorithm for multi-stage stochastic problems (MSSPs) introduced by \citet{pereira1991multi}.
The idea can be traced back to Benders' decomposition and its extension to the multi-stage case provided by \citet{birge1985decomposition}.
SDDP decomposes the whole MSSP into stagewise subproblems.
Each subproblem is an optimization problem that chooses the current decisions to minimize the current cost plus the expected future cost (also called the value function).
This value function is iteratively approximated with a set of piecewise linear functions called cuts, which are successively tightened.
By replacing explicit scenario enumeration with this cut-based representation, SDDP addresses the curse-of-horizon, the exponential increase in the number of scenarios with respect to the length of the planning horizon.
Under standard assumptions (e.g., linearity), SDDP yields the lower bounds of the value function with convergence to the global optimum in a finite number of iterations.
\citet{philpott2008convergence} proved that, given mild conditions, the SDDP procedure will converge with probability one to the optimal policy.
Subsequent analyses have extended the convergence proof to broader classes of programs: convex programs \citep{girardeau2015convergence}, and mixed-integer programs \citep{zou2019stochastic}.
For computational enhancement, recent research developed meta-learning approaches called Neural SDDP \citep{dai2021neural}, which can be integrated seamlessly with the existing SDDP-type algorithm.

Subsequent developments have extended SDDP along risk aversion \citep{shapiro2013risk}.
Dynamic formulations that propagate risk measures allow the planner to hedge tail events while preserving tractability.
Time-consistency, i.e., the absence of incentives to deviate ex post from an ex-ante optimal plan, can be ensured by carefully embedding these risk measures in the multi-stage recursion \citep{homem2016risk,dowson2025incorporating}.
\citet{guigues2016convergence} and \citet{dowson2025incorporating} proved the convergence properties of the risk-averse SDDP variants.

{
The convergence guarantees above are typically established under the assumption of relatively complete recourse (RCR).
The RCR assumption requires every stagewise subproblem to remain feasible for any state inherited from earlier decisions.
Because RCR is frequently violated in practice, particularly when strategic decisions can render operational subproblems infeasible, a line of research has sought to relax it. 
The relaxation approaches are divided into the following two remedies.
The first augments the standard optimality cuts with feasibility cuts that eliminate infeasible states during the iterative process.
\citet{guigues2016convergence} develops such cuts for risk-averse multi-stage stochastic linear problems (MSSLPs) and proves almost sure convergence without the RCR assumption.
Because these cuts are typically derived from Farkas certificates of linear infeasibility, extending this approach to mixed-integer or nonlinear subproblems is not straightforward.
The second relaxes feasibility through slack variables and penalization, replacing the problematic constraints by penalty terms.
Under exact-penalty conditions the penalized problem recovers the original optimum for a finite penalty (cf. \citealp{lefebvre2024exact}).
For risk-neutral MSSLPs, \citet{guigues2023duality} analyze a penalized Dual SDDP variant when the dual recursion violates the RCR assumption.
For risk-neutral multi-stage stochastic mixed-integer nonlinear problems, \citet{zhang2022stochastic} develop an SDDP-type regularization framework that can attain global optimality without the RCR assumption.
However, to the best of our knowledge, existing convergence analyses without the RCR assumption are mainly limited to risk-neutral stochastic programs and risk-averse MSSLPs.
Sufficient conditions guaranteeing lower-bound convergence by the SDDP algorithm in risk-averse and mixed-integer convex settings remain largely unexplored.

To assess solution quality and implement gap-based stopping criteria, an upper bound on the optimal value is needed in addition to the lower bound produced by SDDP. 
In the risk-neutral setting, the standard device is a statistical upper bound obtained by Monte Carlo evaluation of the policy induced by the current lower approximations of the value functions (cf. \citealp{shapiro2011analysis}).
This standard approach has two limitations.
First, although it is well established for risk-neutral objectives, its extension to risk-averse settings is more delicate because the risk-adjusted objective is no longer a simple expectation of simulated total costs.
It requires a construction tailored to the chosen risk measure and its dynamic representation. 
\citet{kozmik2015evaluating} studied policy evaluation for risk-averse MSSLPs with the Conditional Value-at-risk and \citet{guigues2023risk} constructed a statistical upper bound for risk-averse convex optimal control with convex and monotone risk measures.
Second, this upper bound is statistical rather than deterministic.
Because the statistical bound relies on asymptotic normal approximations, a reliable bound may require many Monte Carlo simulations.
Deterministic upper bounds therefore require different constructions.
\citet{guigues2023duality} developed Dual SDDP for dual problems of risk-neutral MSSLPs, and 
\citet{da2023dual} extended Dual SDDP to risk-averse MSSLPs with polyhedral risk measures.}

Although SDDP has seen extensive application in the energy sector (e.g., \citeauthor{pereira1991multi}, \citeyear{pereira1991multi}) and has been adopted in finance (e.g., \citeauthor{lee2023large}, \citeyear{lee2023large}), production management (e.g., \citeauthor{schmiedel2025strategic}, \citeyear{schmiedel2025strategic}), and agriculture (e.g., \citeauthor{dowson2019multi}, \citeyear{dowson2019multi}), the applications to strategic decisions on transportation network design remain limited.
\citet{hua2019joint} proposed an integrated optimization framework for joint infrastructure planning and fleet management in one-way electric car-sharing systems.
\citet{kawase2024multi} presented a multi-stage stochastic linear problem to jointly determine network capacity planning and traveler-vehicle assignment in SAV systems.
Their optimization problems consist of the first stage problem to determine infrastructure planning (e.g., capacity expansion and charging station location), followed by the subsequent stages of sequentially determining the vehicle dispatch plan.
\cite{yu2021value,yu2024value} addressed multi-stage facility location and capacity expansion planning with risk aversion.
A risk-averse multi-stage stochastic problem was formulated with expected conditional risk measures to capture risk in a time-consistent manner. 
Operational performance induced by strategic decisions was evaluated via a static network flow problem.

%% file: ch2-5.tex
\subsection{Originality of this study}

There has been little effort to address both multi-stage decision-making and uncertainties for transportation network design.
Conventional multi-stage stochastic NDPs rely on scenario trees, which suffer from exponential growth as the length of the planning horizon increases. 
Although learning-based surrogates mitigate this computational burden, their efficiency generally comes at the expense of optimality guarantees, failing to be time-consistent and meet the reliability desired by planners.
Furthermore, to our knowledge, a risk-averse, multi-stage NDP has not been developed, leaving a research gap at the intersection of flexible investment planning and reliable decision-making.

From a transportation modeling perspective, the originality of this study lies in filling this gap by formulating a general class of flexible and reliable network design problems (FR-NDPs) as risk-averse MSSPs solvable with SDDP.
{
By combining the network-design modeling reviewed in Section~\ref{sec:review1} with the convergence analysis established in the SDDP literature reviewed in Section~\ref{sec:review2}, the proposed framework simultaneously retains multi-stage investment flexibility, time-consistent reliability evaluation, and optimality guarantees.
In particular, the self-financing result derived under multi-stage uncertainty in Section~4.1 is a theoretical outcome obtained by applying the strong duality of stochastic linear programs.
This constitutes a distinctive theoretical contribution made possible by integrating insights from both stochastic programming and transportation theory, through a methodologically grounded framework with optimality guarantees.
}

Despite its success in other domains, SDDP has seen only limited application in transportation network design. Notably, existing SDDP applications to strategic decisions for emerging transportation services (e.g., \citeauthor{hua2019joint}, \citeyear{hua2019joint}; \citeauthor{kawase2024multi}, \citeyear{kawase2024multi}) address single-stage NDP, and thus do not engage with sequential investment planning. Risk-averse multi-stage models for facility location/capacity expansion, proposed by \citet{yu2021value,yu2024value}, can be viewed as special cases of our FR-NDP framework; however, they abstract away within-day operational dynamics, potentially underestimating the risk-mitigation effect due to operational flexibility. Our FR-NDP is designed so that existing risk-neutral single-stage NDPs can be embedded with only minor customization, thereby incorporating flexibility and reliability into the network design.

{
From a stochastic programming perspective, our main methodological contribution is to establish sufficient conditions, presented in Section~3.3.1.1, under which the lower bound produced by SDDP converges almost surely to the global optimum in risk-averse, mixed-integer convex settings without the relatively complete recourse (RCR) assumption.
In transportation network design, RCR is not merely a technical condition. 
A strategic decision, such as withholding capacity expansion or not installing a transit route, can render later operational subproblems infeasible under some demand realizations.
Furthermore discrete decisions of whether or not to build are typical in network design.
As discussed in Section~\ref{sec:review2}, existing convergence results that dispense with RCR are confined either to risk-averse MSSLPs or to risk-neutral MSSPs, leaving the risk-averse, mixed-integer convex case unexplored. 
This study fills this gap by restricting the class of risk measures to those satisfying conditional consistency, which admits a deterministic-equivalent reformulation with an exact-penalty argument for risk-averse, mixed-integer convex MSSPs.
Conditional consistency further yields a complementary benefit.
Section~3.3.1.2 provides sufficient conditions for computing a statistical upper bound and for obtaining a deterministic upper bound through the Dual SDDP algorithm.
This contribution is positioned as an extension of established results regarding upper bounds to a class of risk measures that has not previously been studied.}

%% file: ch3-1.tex
\section{Flexible and reliable network design problems}
\label{sec:3}

This section introduces a general formulation of FR-NDPs and explains sufficient conditions for their convergence with the SDDP algorithm.
We first illustrate a typical problem statement related to vehicle operations and strategic investments for emerging transportation services.
Our formulation is expressed in an abstract functional form to accommodate a broad class of FR-NDPs.
While such abstraction is instrumental in developing a general theory, it can hinder intuitive understanding and practical interpretation of the proposed framework.
Section \ref{sec:3.1} aims to present a specific illustration bridging to the subsequent abstract description.
It should be noted that the proposed framework is not confined to the specific setting illustrated here.

\subsection{Typical problem statement}
\label{sec:3.1}

We consider a class of transportation systems characterized by on-demand pickup and drop-off operations tailored to travelers' requests.
Travelers specify their demand profiles, including their origin, destination, and desired arrival time, and then receive transportation services accordingly.
Such demand-responsive systems, which include highly flexible services (e.g., SAVs) and high-capacity, low-flexibility services (e.g., fixed-route buses), are increasingly recognized as central components of emerging transportation services.
The infrastructure design, including roads and parking lots, for such systems has attracted significant attention due to their potential to enhance the responsiveness of emerging transportation services.

The system involves strategic and operational costs to satisfy travel demand at a desired service level.
The former includes infrastructure investment and maintenance costs, while the latter includes travelers' travel and schedule costs, vehicle mileage, and, in applications such as electric vehicles, energy costs.
Just as energy costs are typically categorized as operational costs, the combination of strategic and operational costs comprehensively represents the costs associated with emerging transportation services.

We consider a single decision-maker (planner) who aims to reduce the total social costs and makes strategic and operational decisions on emerging transportation services.
Strategic decisions involve infrastructure investments for facility construction and capacity expansion.
During the planning horizon, infrastructures could deteriorate, whereas mobility dependence on emerging transportation services could increase.
To address the incongruence due to such long-term changes in supply and demand, the planner has the flexibility to invest now or delay investments, from period to period (e.g., every year).
Operational decisions, including dynamic vehicle relocation, routing, and ride-sharing matching, are updated more frequently
to satisfy travel demand under capacity constraints resulting from past investments.
The planner makes optimal decisions to reduce the total social cost, while considering the essential interaction between strategic and operational levels.

Uncertainties inherent in the emerging transportation system complicate the planner's strategic and operational decision-making process.
The system involves strategic uncertainties, such as demand growth and infrastructure deterioration, and operational uncertainties, such as fluctuations in daily demand and energy costs.
Strategic uncertainties, such as demand growth, can often be inferred from travel surveys, whereas daily demand can be monitored through travelers' requests.

To mitigate the impacts of uncertainties, the planner's decision timeline is nested as follows: annually or every few years, the planner revisits long-term infrastructure investments in response to the realized strategic uncertainties.
Between two consecutive strategic reviews, the planner repeatedly observes the realized operational uncertainties and implements corresponding operational decisions.
After this sequence of operational decisions, the planner again observes the updated system state and revisits the strategic decisions.
The planner repeats this cycle during the specified planning horizon to minimize the total expected social cost.
Alternatively, the planner may adopt a risk measure that explicitly penalizes severe scenarios instead of the expectation-based criterion, if significant losses are of particular concern.

%% file: ch3-2.tex
\subsection{General problem statement and formulation}

This subsection presents a general formulation of FR-NDPs for emerging transportation services.
We consider a discrete planning horizon of length $\tau$ [period].
Each period comprises a strategic stage and an operational stage.
We denote the number of stages within the planning horizon as $S=2\tau$, 
set of all stages as $\mathcal{S}=\{1,...,S\}$, 
and sets of strategic and operational stages as $\mathcal{S}^{\rm F}=\{1,3,...,S-1\}$ and $\mathcal{S}^{\rm G}=\{2,4,..,S\}$, respectively.
Throughout this paper, to avoid notational complexity, the set of superscripts ${\rm F}$ and ${\rm G}$ is denoted as ${{\rm F}/{\rm G}}$.

At each stage, a single planner observes realizations of uncertainties and chooses decisions.
Strategic and operational uncertainties, $\bm{\xi}^{{\rm F}/{\rm G}}_s$,
are revealed at the beginning of the corresponding stage $s\in\mathcal{S}$.
Subsequently, strategic and operational decision variables, $\bm{u}^{{\rm F}/{\rm G}}_s$, are optimized to minimize current and future costs in response to the corresponding realized uncertainties.
Figure \ref{fig:process} illustrates the decision-making process in the proposed framework.

We assume that the system remains steady-state within each stage.
Specifically, for each stage $s\in\mathcal{S}$, the set of possible realizations of the uncertain parameter and the probability assigned to each realization are fixed.
They do not change over the course of the sub-time steps within the stage.
For example, daily demand at an operational stage $s$ follows a probability distribution, and any shift in this distribution, such as demand growth, is captured at the beginning of the next strategic stage $s+1$.
In addition, the infrastructure state also stays unchanged within each operational stage $s$.
Any changes in infrastructure state due to expansion or deterioration occur at the subsequent strategic stage $s+1$.

Under these assumptions, we define the system state at stage $s\in\mathcal{S}$ as $\bm{x}_s$.
According to \citet{powell2019unified} and \citet{dowson2020policy},
$\bm{x}_s$ captures all the information necessary to model the system after stage $s+1$.
In the context of our problem, typical components of $\bm{x}_s$ include network capacities and demand growth.
$\bm{x}_s$ evolves according to the decision variables $\bm{u}^{{\rm F}/{\rm G}}_{s}$ and the realized uncertainties $\bm{\xi}^{{\rm F}/{\rm G}}_s$.
{
Because the system state summarizes all the information needed to make optimal decisions in subsequent stages, we focus on decision rules that select $\bm{u}_s^{{\rm F}/{\rm G}}$ based on the incoming state $\bm{x}_{s-1}$ and the realized uncertainty $\bm{\xi}_s^{{\rm F}/{\rm G}}$ at stage $s$.
Following \citet{powell2019unified} and \citet{dowson2020policy}, we refer to the resulting sequence of state-contingent decision rules as a policy.
}

\begin{figure}[t]
	\centering
	\includegraphics[width=1.0\textwidth]{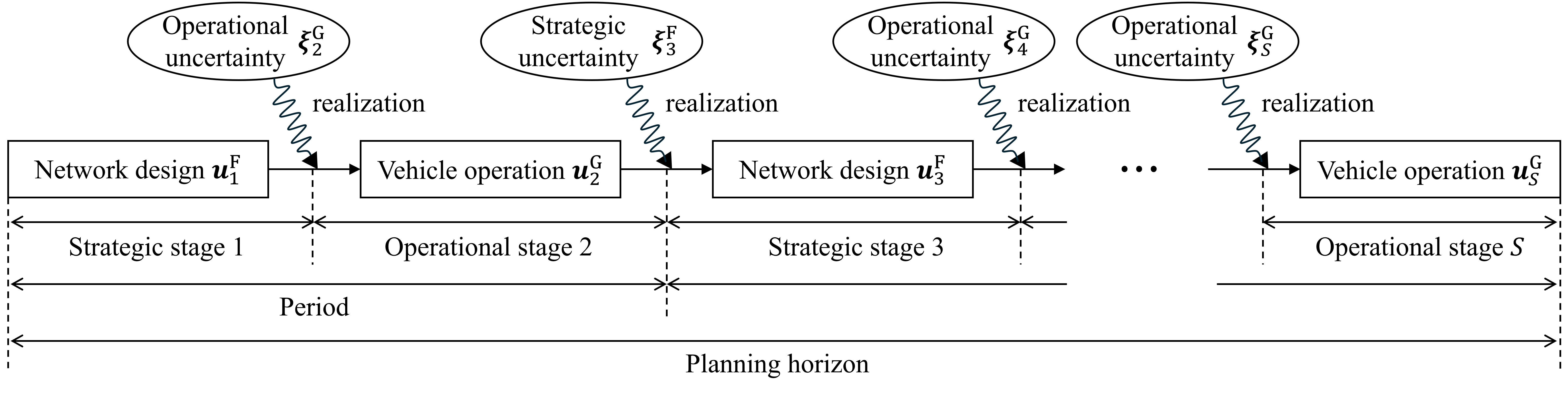}
	\caption{Schematic diagram of the decision-making process in the proposed framework.}
	\label{fig:process}  
\end{figure}

The general formulation of this FR-NDP is expressed as the following risk-averse MSSP:
\begin{flalign}
	\nonumber
	&\min_{\bm{x}_1,\bm{u}_1^{\rm F}}C^{\rm F}_1( \bm{x}_{1}, \bm{u}^{\rm F}_{1}, \bm{\xi}^{\rm F}_{1} )+\mathbb{F}^{}_{\bm{\xi}^{\rm F}_1}\Biggl[\min_{\bm{x}_2,\bm{u}_2^{\rm G}}C^{\rm G}_2( \bm{x}_{2}, \bm{u}^{\rm G}_{2}, \bm{\xi}^{\rm G}_{2} )+\mathbb{F}^{}_{\bm{\xi}^{\rm G}_2}\Biggl[\min_{\bm{x}_3,\bm{u}_3^{\rm F}}C^{\rm F}_3( \bm{x}_{3}, \bm{u}^{\rm F}_{3}, \bm{\xi}^{\rm F}_{3} )+\mathbb{F}^{}_{\bm{\xi}^{\rm F}_3}\Biggl[\cdots&\\
	\label{eq:Generic_Objective}
	&
	~~~~~~~~~~~~~~~~~~~~~~~~~~~~~~+\mathbb{F}^{}_{\bm{\xi}^{\rm G}_{S-2}}\Biggl[\min_{\bm{x}_{S-1},\bm{u}_{S-1}^{\rm F}}C^{\rm F}_{S-1}( \bm{x}_{S-1}, \bm{u}^{\rm F}_{S-1}, \bm{\xi}^{\rm F}_{S-1} )+\mathbb{F}^{}_{\bm{\xi}^{\rm F}_{S-1}}\left[\min_{\bm{x}_S,\bm{u}_S^{\rm G}}C^{\rm G}_S( \bm{x}_{S}, \bm{u}^{\rm G}_{S}, \bm{\xi}^{\rm G}_{S} )\right]\Biggr]\Biggr]\Biggr]\Biggr],&
\end{flalign}
\vspace{-2ex}
\begin{flalign}
        \label{eq:Generic_Transition}
	&{\rm s.t.~~} \bm{x}_s=
	\begin{cases}
		T_s^{\rm F}( \bm{x}_{s-1}, \bm{u}^{\rm F}_{s}, \bm{\xi}^{\rm F}_{s} ) & {\rm if~}s\in\mathcal{S}^{\rm F}\\
		T_s^{\rm G}( \bm{x}_{s-1}, \bm{u}^{\rm G}_{s}, \bm{\xi}^{\rm G}_{s} ) & {\rm if~}s\in\mathcal{S}^{\rm G}
	\end{cases}&\forall s\in\mathcal{S},\\
        \label{eq:Generic_Strategic_Action}
	& ~~~~~~~\bm{u}^{\rm F}_s \in U_s^{\rm F}(\bm{x}_{s}, \bm{\xi}^{\rm F}_{s}) & \forall s\in\mathcal{S}^{\rm F},\\
        \label{eq:Generic_Operational_Action}
	& ~~~~~~~\bm{u}^{\rm G}_s \in U_s^{\rm G}(\bm{x}_{s}, \bm{\xi}^{\rm G}_{s}) & \forall s\in\mathcal{S}^{\rm G},\\
	\label{eq:initial_condition}
	& ~~~~~~~\bm{x}_{0}=\hat{\bm{x}}_{0},
\end{flalign}
where $\mathbb{F}^{}_{\bm{\xi}}[\cdot]$ is defined as a conditional risk measure on $\bm{\xi}$; it can describe a planner's risk preference.
$C^{{\rm F}/{\rm G}}_s$ are cost functions incurred at stage $s$, mapping decisions $\bm{u}^{{\rm F}/{\rm G}}_{s}$, states $\bm{x}_{s}$, and realized uncertainties, $\bm{\xi}^{{\rm F}/{\rm G}}_s$, to scalars.
$T_s^{{\rm F}/{\rm G}}$ represent transition functions mapping the incoming state $\bm{x}_{s-1}$ at the beginning of stage $s$ to a random outgoing state $\bm{x}_{s}$ at the end of stage $s$.
$\hat{\bm{x}}_{0}$ is a given initial state.
$U_s^{{\rm F}/{\rm G}}$ represent action spaces for decision variables.
These functions can capture features in transportation systems.
Some examples are listed in Table \ref{tab:example}.

\begin{table}[t]
 	\centering
	\caption{Examples of features in transportation systems that can be expressed by $T$ and $U$.}
	\label{tab:example}  
        \begin{tabular}{lll}
        \hline
        stage             & function      & explainable feature                                                                                          \\ \hline
                          &               & network capacity expansion                                                                                   \\
                          & $T_s^{\rm F}$ & infrastructure deterioration                                                                   \\
        strategic stage   &               & travel demand growth                                                                                   \\ \cline{2-3} 
                          & $U_s^{\rm F}$ & \begin{tabular}[c]{@{}l@{}}budget constraint\\ connection constraint for public transit routes\end{tabular}   \\ \hline
        operational stage & $U_s^{\rm G}$ & \begin{tabular}[c]{@{}l@{}}network flow dynamics\\ capacity constraint\\ time window constraint\end{tabular} \\ \hline
        \end{tabular}
\end{table}

According to Bellman's principle, Eqs. (\ref{eq:Generic_Objective})--(\ref{eq:initial_condition}) can be decomposed into the following dynamic programming equations:
\begin{flalign}
	\nonumber
	&\underline{\rm[Strategic~subproblems]}&\\
	\label{eq:Strategic_Objective}
	&V_s(\bm{x}_{s-1}, \bm{\xi}^{\rm F}_{s}) = \min_{\bm{x}_s,\bm{u}_s^{\rm F}}C^{\rm F}_s( \bm{x}_{s}, \bm{u}^{\rm F}_{s}, \bm{\xi}^{\rm F}_{s} )+\mathbb{F}^{}_{\bm{\xi}^{\rm F}_s}[V_{s+1}(\bm{x}_{s}, \bm{\xi}^{\rm G}_{s+1})] & \forall s\in\mathcal{S}^{\rm F},\\
	\label{eq:Strategic_Transition}
	&{\rm s.t.~~} \bm{x}_s=T_s^{\rm F}( \bm{x}_{s-1}, \bm{u}^{\rm F}_{s}, \bm{\xi}^{\rm F}_{s} ), & \forall s\in\mathcal{S}^{\rm F},\\
	\label{eq:Strategic_Action}
	& ~~~~~~~\bm{u}^{\rm F}_s \in U_s^{\rm F}(\bm{x}_{s}, \bm{\xi}^{\rm F}_{s}) & \forall s\in\mathcal{S}^{\rm F},\\
	& ~~~~~~~\bm{x}_{0}=\hat{\bm{x}}_{0},\tag{\ref{eq:initial_condition}}&
\end{flalign}
\vspace{-4ex}
\begin{flalign}
	\nonumber
	&\underline{\rm[Operational~subproblems]}&\\
	\label{eq:Operational_Objective}
	&V_s(\bm{x}_{s-1}, \bm{\xi}^{\rm G}_{s}) = \min_{\bm{x}_s,\bm{u}_s^{\rm G}}C^{\rm G}_s( \bm{x}_{s}, \bm{u}^{\rm G}_{s}, \bm{\xi}^{\rm G}_{s} )+\mathbb{F}^{}_{\bm{\xi}^{\rm G}_s}[V_{s+1}(\bm{x}_{s}, \bm{\xi}^{\rm F}_{s+1})] & \forall s\in\mathcal{S}^{\rm G},\\
	\label{eq:Operational_Transition}
	&{\rm s.t.~~} \bm{x}_s=T_s^{\rm G}( \bm{x}_{s-1}, \bm{u}^{\rm G}_{s}, \bm{\xi}^{\rm G}_{s} ), & \forall s\in\mathcal{S}^{\rm G},\\
	\label{eq:Operational_Action}
	& ~~~~~~~\bm{u}^{\rm G}_s \in U_s^{\rm G}(\bm{x}_{s}, \bm{\xi}^{\rm G}_{s}) & \forall s\in\mathcal{S}^{\rm G},\\
	\label{eq:terminal_condition}
	& ~~~~~~~V_{S+1}(\bm{x}^{}_{S},\cdot)=0,&
\end{flalign}
where $V_s(\bm{x}_{s-1}, \bm{\xi}^{{\rm F}/{\rm G}}_{s})$ denotes the value function:
given the state $\bm{x}_{s-1}$ and the realized uncertainty $\bm{\xi}^{{\rm F}/{\rm G}}_{s}$,
$V_s(\bm{x}_{s-1}, \bm{\xi}^{{\rm F}/{\rm G}}_{s})$ is the minimal risk-adjusted cumulative cost from stage $s$ to the end of the horizon.
Obviously, in the special case of $S=2$ (i.e., one strategic subproblem only followed by one operational subproblem),
Eqs. (\ref{eq:initial_condition})--(\ref{eq:terminal_condition}) can be reduced to a single-stage NDP.
This observation indicates that the proposed framework can extend conventional NDPs
to incorporate flexibility and reliability simply by embedding single-stage NDP structures into
the functions $C^{{\rm F}/{\rm G}}_s$, $T^{{\rm F}/{\rm G}}_s$, and $U^{{\rm F}/{\rm G}}_s$.

%% file: ch3-3.tex
\subsection{Convergence property}
\label{sec: convergence}

A dynamic programming equation, as in Eqs. (\ref{eq:initial_condition})--(\ref{eq:terminal_condition}), is solved when the value functions $V_s(\cdot, \cdot)$ are identified.
{Given the value functions, the optimal policy is induced by selecting the optimal solution to the subproblem over each stage $s\in\mathcal{S}$.
}
A classical approach is to recursively compute the value function at a set of discretized points with respect to the state space.
However, this method is limited to low-dimensional problems, as the number of discretization points grows exponentially with the dimensionality of the state space.

Stochastic Dual Dynamic Programming (SDDP) \citep{pereira1991multi,shapiro2011analysis} provides a theoretically sound alternative for solving MSSPs with practical memory requirements.
Rather than evaluating the value functions at discretized points, SDDP forms a lower approximation of each value function using a set of piecewise linear functions, referred to as cuts.
These cuts are generated iteratively by solving dual
subproblems and refined over multiple iterations.
As the number of iterations increases, the set of cuts becomes richer, improving the approximation and tightening the lower bound for the value functions.
Furthermore, when the MSSP has a specific form (e.g., convexity), SDDP ensures that the lower bound almost surely converges to the global optimum \citep{guigues2016convergence}.
{{Complementing the lower bound, upper bounds are useful for assessing solution quality and for implementing stopping criteria.
A statistical upper bound can be computed by Monte Carlo evaluation of the policy induced by the current lower approximations of the value functions.
Under more restrictive assumptions (e.g., linearity), a deterministic upper bound with convergence guarantees can also be obtained via a variant of SDDP, referred to as Dual SDDP \citep{guigues2023duality}.
For completeness, Appendix \ref{app:sddp0} reviews established results on SDDP-type algorithms and their convergence properties.
}}

Building on the above discussion, this subsection establishes sufficient conditions under which SDDP guarantees convergence to the global optimum for the FR-NDPs (\ref{eq:initial_condition})--(\ref{eq:terminal_condition}).
{
Subsection \ref{subsec:convergence} first lists the technical notions used throughout this subsection, and then states the sufficient conditions relevant to the lower and upper bounds.} 
We further discuss the proposed sufficient conditions that cover a broad class of FR-NDPs towards emerging transportation services.

\subsubsection{Sufficient conditions for convergence}
\label{subsec:convergence}

{
This subsection provides sufficient conditions under which SDDP provides convergence guarantees for the FR-NDP (\ref{eq:initial_condition})--(\ref{eq:terminal_condition}).
We first present sufficient conditions guaranteeing almost sure convergence of the lower bounds generated by SDDP.
We then turn to upper bounds and state sufficient conditions under which
(i) statistical upper bounds can be estimated via Monte Carlo evaluation of the current policy and
(ii) deterministic upper bounds with convergence guarantees can be generated via Dual SDDP.
}

{
Before stating the sufficient conditions, we define the technical notions used throughout this subsection and briefly explain why they are needed.
The convergence analysis hinges on two structural ingredients.
First, SDDP constructs valid supporting cuts of the value functions from dual information, which relies on a convex structure in each subproblem and on strong duality guaranteed by an appropriate constraint qualification\footnote{Formally, we assume strict feasibility of the continuous relaxation; the precise statement is provided in Appendix \ref{app:prop1}}.
This motivates the requirements that the feasible sets be mixed integer convex and that the stage costs be convex.
Second, the dynamic programming equations incorporate risk aversion through stagewise conditional risk measures.
To ensure that the risk-adjusted value functions remain amenable to cut-based approximation, a convexity requirement for the risk measures is natural. 
Beyond convexity, we impose conditional consistency to prevent discrepancies between (i) the stagewise risk evaluation with conditional risk measures and (ii) the overall risk evaluation of the total accumulated cost over the planning horizon.
This property plays a central role both in establishing almost-sure convergence of deterministic bounds and in constructing statistical upper bounds via Monte Carlo evaluation. 
We introduce two representative risk specifications that satisfy these requirements: the entropic risk measure, a canonical conditionally consistent example, and expected conditional risk measures (ECRMs), which can restore conditional consistency for popular risk measures such as the conditional value-at-risk.
Note that we restrict attention to random variables for which the expectations and risk measures appearing below are well defined.
}

{
\begin{dfn}[Convex risk measure; \citealp{follmer2002convex}]
	\label{def:convex_risk}
	For random variables $Y$ and $Z$, a risk measure $\mathbb{F}$ is convex if it satisfies the following axioms:
	\begin{itemize}
		\setlength{\itemsep}{0pt} \setlength{\parskip}{0pt}
		\item Axiom 1. Monotonicity: if $Y\le Z$, then $\mathbb{F}[Y]\le\mathbb{F}[Z]$.
		\item Axiom 2. Translation invariance: $\mathbb{F}[Z+a]=\mathbb{F}[Z]+a$ for all $a\in\mathbb{R}$.
		\item Axiom 3. Convexity: $\mathbb{F}[aY+(1-a)Z]\le a\mathbb{F}[Y] + (1-a)\mathbb{F}[Z]$  for all $a\in[0,1]$.
	\end{itemize}
\end{dfn}
\begin{dfn}[Conditional consistency\footnote{Variations of conditional consistency have been defined and discussed in several papers (e.g., \citeauthor{kovacevic2009time}, \citeyear{kovacevic2009time}; \citeauthor{dowson2025incorporating}, \citeyear{dowson2025incorporating}).
For tractability, we choose the definition of conditional consistency following \citet{fullner2025stochastic}.}]
\label{def:consistency_risk}
	For a sequence of random variables $\{Z_s\}_{s=1}^S$, a risk measure $\mathbb{F}$ is conditionally consistent if it holds that
	\begin{flalign}
		\label{eq:CC}
		&\mathbb{F}[Z_1+Z_2+\cdots + Z_S]=\mathbb{F}[Z_1+\mathbb{F}_{Z_1}[Z_2+\cdots+\mathbb{F}_{Z_{S-1}}[Z_S]]].&
	\end{flalign}
	The left-hand side of Eq.~\eqref{eq:CC} is called the end-of-horizon risk, whereas the right-hand side of Eq. \eqref{eq:CC} is called the nested risk.
\end{dfn}
\begin{dfn}[Conditional value-at-risk; \citealp{rockafellar2002conditional}]
	Given $\alpha\in[0,1)$ and a random variable $Z$, the conditional value-at-risk ($\mathbb{CV@R}_{\alpha}$) is defined by 
	\begin{flalign}
		\label{eq:CVaR}
		&\mathbb{CV@R}_{\alpha}[Z] := \min_{\eta\in\mathbb{R}}\left\{\eta + \frac{1}{1-\alpha}\mathbb{E}[(Z-\eta)_{+}]\right\},&
	\end{flalign}
	where $(Z-\eta)_{+} = \max\{Z-\eta, 0\}$.
	In the limiting cases, $\mathbb{CV@R}_{0}[Z] = \mathbb{E}[Z]$ and $\lim_{\alpha\rightarrow1}\mathbb{CV@R}_{\alpha} = \operatorname{ess}\sup[Z]$, where $\operatorname{ess}\sup[Z]$ denotes the essential supremum of a random variable $Z$.
\end{dfn}
\begin{dfn}[Entropic risk measure; \citealp{kovacevic2009time}]
	\label{def:entropic_risk}
	Given $\gamma>0$ and a random variable $Z$, the entropic risk measure ($\mathbb{ENT}_{\gamma}$) is defined by 
	\begin{flalign}
		\label{eq:ENT}
		&\mathbb{ENT}_{\gamma}[Z] := \frac{1}{\gamma}\log\left(\mathbb{E}[\exp(\gamma Z)]\right),&
	\end{flalign}
	where $\lim_{\gamma\rightarrow 0}\mathbb{ENT}_{\gamma}[Z] = \mathbb{E}[Z]$ and $\lim_{\gamma\rightarrow\infty}\mathbb{ENT}_{\gamma}[Z] = \operatorname{ess}\sup[Z]$.
\end{dfn}
\begin{dfn}[Expected conditional risk measure; \citealp{homem2016risk}]
	\label{def:ECRM}
	Given a sequence of random variables $\{Z_s\}_{s=1}^S$ and stagewise conditional risk measures $\{\mathbb{F}_{Z_s}\}_{s=1}^{S-1}$, 
	the expected conditional risk measure ($\mathbb{ECRM}$) is a multi-period risk measure, mapping $\{Z_s\}_{s=1}^S$ to a real number, defined by 
	\begin{flalign}
		\label{eq:ECRM}
		&\mathbb{ECRM}[Z_1,...,Z_S]:=Z_1 + \mathbb{F}_{Z_1}\left[Z_2\right] + \mathbb{E}_{Z_1}\left[ \mathbb{F}_{Z_2}[Z_3] + \mathbb{E}_{Z_2}\left[ \mathbb{F}_{Z_3}[Z_4] + \cdots + \mathbb{E}_{Z_{S-2}}\left[  \mathbb{F}_{Z_{S-1}}\left[Z_S\right] \right]\right]\right].&
	\end{flalign}
	Equivalently, by repeated application of linearity and the tower property of conditional expectation,
	\begin{flalign}
		\label{eq:ECRM2}
		&\mathbb{ECRM}[Z_1,...,Z_S]=\mathbb{E}\left[Z_1 + \mathbb{F}_{Z_1}\left[Z_2\right] + \mathbb{F}_{Z_2}[Z_3] + \mathbb{F}_{Z_3}[Z_4] + \cdots + \mathbb{F}_{Z_{S-1}}\left[Z_S\right] \right].&
	\end{flalign}
\end{dfn}
In optimization settings, it is useful to consider convex combinations of the expectation and a risk measure, $(1-\beta)\mathbb{E}[Z] + \beta \mathbb{F}[Z]$, where $\beta\in[0,1]$.
For example, the convex combination of $\mathbb{CV@R}_{\alpha}[\cdot]$ and $\mathbb{E}[\cdot]$, referred to as mean-$\mathbb{CV@R}_{\alpha}$, is strictly monotone and can yield desirable optimization properties such as solution uniqueness in some settings.
}

{
With these definitions in place, we first state sufficient conditions for almost sure convergence of lower bounds.
We then provide sufficient conditions for computing statistical upper bounds via Monte Carlo evaluation and for obtaining convergent deterministic upper bounds.
Throughout this subsection, we assume that the constraint qualification stated in Appendix \ref{app:prop1} holds.
}

\paragraph{Sufficient conditions relevant to lower bounds}

Building on the literature on convergence proofs for risk-averse MSSPs (e.g., \citealp{guigues2016convergence}), sufficient conditions for solving the FR-NDP (\ref{eq:initial_condition})--(\ref{eq:terminal_condition}) with convergence guarantees using SDDP in a finite number of iterations are given by the following proposition:
\begin{prop}
    \label{prop:suff}
    The sufficient conditions for the lower bound of the value function {$V_1$} in Eqs. (\ref{eq:initial_condition})--(\ref{eq:terminal_condition}) to converge almost surely { and non-decreasingly} to the global optimum using the SDDP algorithm in a finite number of iterations are as follows:
    \begin{enumerate}
    	\setlength{\itemsep}{0pt} \setlength{\parskip}{0pt}
		\item Given fixed $\bm{\xi}_s$, {the cost function} $C_s$ is {lower semicontinuous, proper, and} convex, {the transition function} $T_s$ is {linear}, and {the action space} $U_s$ is a {non-empty compact mixed integer} convex set for all stages $s\in\mathcal{S}$,
    	\item {The conditional risk measure} $\mathbb{F}^{}_{\bm{\xi}_{s}}$ is conditionally consistent and convex  for all stages $s\in\mathcal{S}$,
        \item When integer variables are included, all the $\bm{x}_s$ are integer variables for all stages $s\in\mathcal{S}$,
        \item The length of the planning horizon, $\tau$, is finite and given,
    	\item The number of realizations of $\bm{\xi}_{s}$ is finite for all stages $s\in\mathcal{S}$, and
    	\item Realizations of $\bm{\xi}_{s}$ are independent of other stages for all stages $s\in\mathcal{S}$,
    \end{enumerate}
    where superscripts are dropped as the same conditions are required for strategic and operational subproblems.
\end{prop}
Proposition \ref{prop:suff} is proved in Appendix \ref{app:prop1}.
Condition 6 is commonly referred to as stagewise independence in the SDDP literature.

{
{A key technical issue in solving risk-averse MSSPs with SDDP is how to handle violations of the relatively complete recourse (RCR); see Assumption \ref{app:assumption_RCR} in Appendix \ref{app:sddp}.
Intuitively, RCR requires that the subproblem at stage $s\in\mathcal{S}$ remain feasible for any obtainable state at the previous stage $s-1$. 

This requirement is not trivial in transportation network design because strategic decisions can directly destroy operational feasibility.
An illustrative case is when planners choose not to expand network capacity to reduce investment costs.
Importantly, this case is not merely an extreme.
Because, in early iterations of SDDP, the lower bound approximation of the operational value function is typically loose, the positive value of infrastructure investment is underestimated in the strategic subproblem.
In such early iterations, some demand realizations may yield no feasible operational decision satisfying travel demand constraints, which is a typical violation of RCR.}

Proposition \ref{prop:suff} departs from the RCR assumption by restricting the conditional risk measure $\mathbb{F}_{\bm{\xi}_s}$ to the class of conditionally consistent measures. 
This restriction allows the nested risk representation to be collapsed into an equivalent end-of-horizon representation (cf. Eq. \eqref{eq:CC}), thereby enabling a reformulation as a deterministic equivalent mixed integer convex program. 
Furthermore, employing established results on exact penalty methods for deterministic optimization \citep{lefebvre2024exact}, we prove that the penalized formulation is exact: for sufficiently large (but finite) penalty coefficients, the optimal value of the penalized problem coincides with that of the original problem whenever the original problem is feasible.
The technical details are relegated to Appendix \ref{app:prop1}.

\begin{rem}
	\label{rem:integer}
	Since MSSPs with integer variables (permitted under Condition 3) are nonconvex, the value functions of their dynamic programming equations are generally nonconvex with respect to state variables.
	Fortunately, this difficulty can be overcome because any real-valued function of binary variables can be represented exactly by a convex piecewise linear function.
	This observation provides a theoretical basis for establishing convergence of SDDP-type algorithms for MSSPs with integer variables \citep{zou2019stochastic}.
	Under boundedness (implied by compactness in Condition 1), any bounded integer variable can be represented exactly using a finite set of binary variables.
    Therefore, arguments stated for binary states extend to general bounded integer states.
\end{rem}

\begin{rem}
	\label{rem:ECRM}
	For conditional risk measures, the class of convex measures that satisfy conditional consistency is limited to the expectation and worst-case (essential supremum) operators, as well as the entropic risk measure \citep{kovacevic2009time}.
	While popular convex risk measures such as $\mathbb{CV@R}_{\alpha}$ do not satisfy conditional consistency, embedding $\mathbb{CV@R}_{\alpha}$ into the conditional risk measure $\mathbb{F}_{\bm{\xi}}$ associated with $\mathbb{ECRM}$ can capture both conditional consistency and $\mathbb{CV@R}_{\alpha}$ preferences.
	\citet{khatami2024risk} derive explicit dynamic programming equations for MSSPs with $\mathbb{ECRM}$ as the objective function, rather than the nested risk evaluation shown in Eq. \eqref{eq:Generic_Objective}.
	This dynamic programming representation, however, requires additional continuous state variables that carry risk information across stages (see Appendix \ref{app:prop3}).
\end{rem}
}

\paragraph{Sufficient conditions relevant to upper bounds}

{Upper bounds are needed to support gap-based stopping criteria and to assess solution quality.}
Although SDDP guarantees convergence as stated in Proposition \ref{prop:suff}, the worst-case number of iterations required for convergence remains largely uncharacterized, particularly for problems involving general convex functions \citep{lan2022complexity}.
In practice, commonly used termination criteria include bound stalling, where the SDDP algorithm halts when the lower bound ceases to improve beyond a tolerance over several iterations.
This criterion is valid because the lower bound produced by SDDP is nondecreasing, because the set of cuts used to approximate the value function becomes richer with each iteration.
Another widely adopted criterion in standard optimization techniques is the relative gap between upper and lower bounds.
{To enable such gap-based stopping criteria, we next discuss sufficient conditions for constructing upper bounds.}

{
The first option is a statistical upper bound.
SDDP can provide a candidate policy by solving subproblems forward using the current approximations of the value functions.
The value of this policy can be estimated via Monte Carlo simulation by computing the risk-adjusted total cost over the planning horizon.
An upper confidence bound for this estimate serves as a statistical upper bound.
When the risk measure $\mathbb{F}^{}_{\bm{\xi}_{s}}[\cdot]$ satisfies conditional consistency (Condition 2), this policy evaluation is consistent with the nested objective in Eq. \eqref{eq:Generic_Objective}.
Building on this property, sufficient conditions for constructing the statistical upper bound are given by the following proposition:
\begin{prop}
	\label{prop:suff2}
	Let $\Psi$ denote the random total cost incurred over the planning horizon
	when the policy induced by the current approximation of the value function is implemented.
	Suppose that Conditions 1--6 in Proposition \ref{prop:suff} hold and that 
	the end-of-horizon risk measure admits the representation $\mathbb{F}[\Psi]=h(\mathbb{E}[g(\Psi)])$, where $h(\cdot)$ is nondecreasing and $W:=g(\Psi)$ has finite variance.
	Let $\{\Psi^m\}_{m=1}^M$ be i.i.d. samples of $\Psi$ obtained from $M$ independent Monte Carlo simulations, and define $W^m:=g(\Psi^m)$.
	Let $\mu_W:=\frac{1}{M}\sum_{m=1}^M W^m$ and $\sigma^2_W:=\frac{1}{M-1}\sum_{m=1}^M\left(W^m-\mu_W\right)^2$ be the respective sample mean and sample variance of $W$.
	Then, an asymptotic $100(1-\alpha)\%$ upper confidence bound for the optimal value is given by $h\left( \mu_W + z_{\alpha}\frac{\sigma_W}{\sqrt{M}}\right)$, 
	where $z_{\alpha}$ denotes the $(1-\alpha)$-quantile of the standard normal distribution for $\alpha\in(0,1)$. 
\end{prop}
}
Proposition \ref{prop:suff2} is proved in Appendix \ref{app:prop2}.

{
The entropic risk measure\footnote{Since the logarithm is concave, $\frac{1}{\gamma}\log( \mu_W )$ is generally an optimistically biased estimator of the entropic risk measure (cf. \citealp{sadana2024mitigating}), although $\mu_W$ and $\sigma_W$ are unbiased estimators of the parameters of $W$.} is a typical example that satisfies the representation $\mathbb{F}[Z]=h(\mathbb{E}[g(Z)])$ in Proposition \ref{prop:suff2}, with $g(Z)=\exp(\gamma Z)$ and $h(\cdot)=\frac{1}{\gamma}\log(\cdot)$.
Subsection 12.4 of \citet{fullner2025stochastic} also note that the entropic risk measure allows for upper bound computation via Monte Carlo simulation, however, it does not provide the explicit expression for such an upper bound.

It is important to note that the upper bound in Proposition \ref{prop:suff2} is statistical and asymptotic.
Since its validity relies on a normal approximation,
it may happen that the statistical upper bound becomes smaller than the valid lower bound stated in Proposition \ref{prop:suff}.
Moreover, the impact of $M$ is often more pronounced under risk measures than in risk-neutral (expectation) settings, because risk measures tend to emphasize severe outcomes and/or involve nonlinear transformations.
}

{A promising alternative is a deterministic upper bound.
One can obtain such a bound} by solving the dynamic programming equation of the dual of Eqs. (\ref{eq:Generic_Objective})--(\ref{eq:initial_condition}).
\citet{guigues2023duality} developed {the Dual} SDDP algorithm for dual problems of risk-neutral MSSPs and proved its convergence.
{
In the following proposition, we extend the convergence conditions in \citet{guigues2023duality} to include certain risk-averse MSSPs:}
\begin{prop}
    \label{prop:suff3}
    The sufficient conditions for the upper bound of the value function $V_1$ in Eqs. (\ref{eq:initial_condition})--(\ref{eq:terminal_condition}) to converge almost surely {and non-increasingly} to the global optimum using the Dual SDDP algorithm in a finite number of iterations are as follows:
    \begin{enumerate}
        \setcounter{enumi}{6}
    	\setlength{\itemsep}{0pt} \setlength{\parskip}{0pt}
		\item Given fixed $\bm{\xi}_s$, {the cost function} $C_s$ is linear, {the transition function} $T_s$ is linear, and {the action space} $U_s$ is a non-empty compact polyhedral set for all stages $s\in\mathcal{S}$,
    	\item {
    	The cost functions $\{C_s\}_{s=1}^S$ are mapped by $\mathbb{ECRM}$ with $\mathbb{CV@R}_{\alpha}$ serving as the conditional risk measures $\{\mathbb{F}^{}_{\bm{\xi}_{s}}\}_{s=1}^{S-1}$,
    	}
        \item All the $\bm{u}_s$ and $\bm{x}_s$ are continuous variables for all stages $s\in\mathcal{S}$, and 
    \end{enumerate}
    Conditions 4--6 in Proposition \ref{prop:suff}, where superscripts are dropped as the same conditions are required for strategic and operational subproblems.
\end{prop}
{
Proposition \ref{prop:suff3} is proved in Appendix \ref{app:prop3}.
}

{
Proposition \ref{prop:suff3} can be viewed as a risk-averse generalization of the convergence result for deterministic upper bounds with Dual SDDP established by \citet{guigues2023duality}. 
When $\alpha=0$, $\mathbb{ECRM}$ with $\mathbb{CV@R}_{\alpha}$ reduces to the conditional expectation operator.
In a complementary direction, \citet{da2023dual} pursue a different generalization by developing Dual SDDP for risk-averse MSSPs with specific conditional risk measures. 
Our focus differs in that we highlight \textit{conditionally consistent} risk measures and demonstrate that the deterministic upper bounds via Dual SDDP remain valid under this structure.

Beyond the above theoretical positioning, expressing risk  preferences through $\mathbb{ECRM}$ with $\mathbb{CV@R}_{\alpha}$ also offers concrete modeling benefits in transportation network design.
First, applying $\mathbb{CV@R}_{\alpha}$ to each stagewise cost enables different risk attitudes toward design expenditures and operational performance.
Second, conditional consistency implies that the end-of-horizon risk of total system costs can be evaluated when committing to the initial strategic decisions.
This contrasts with conditional risk formulations (such as \citealp{da2023dual}) that apply the risk mapping to the value functions.
}

{
\begin{rem}
    The convergence guarantees of this study are qualitative. 
    Propositions \ref{prop:suff}--\ref{prop:suff3}
    hold independently of the problem scale given sufficiently many iterations.
    None of them, however, characterizes the quantitative property of how many iterations convergence requires.
    This computational scalability is beyond the scope of our propositions.
    In Section 6, we discuss this limitation and a promising direction to address it.
\end{rem}
}

%% file: ch3-4.tex
\subsubsection{Discussion on application to emerging transportation services}

We will make some comments regarding sufficient conditions in Propositions \ref{prop:suff}--\ref{prop:suff3} {from a modeling perspective. 
This subsection delineates a tractable class of FR-NDP formulations in which flexibility and reliability can be incorporated while preserving the SDDP-type convergence guarantees established above.
We discuss
(i)~strategic settings in which the conditions are naturally satisfied,
(ii)~operational settings in which they hold only after approximation, together with the standard devices used to recover them, and
(iii)~strategic settings in which the conditions should be regarded as approximations or require a different theoretical treatment. 

Many strategic network design decisions are represented naturally by linear transition functions and mixed-integer polyhedral action sets.
Since capacity expansion across stages and discrete decisions such as constructing new facilities are linear in the state, the corresponding feasible regions are mixed-integer polyhedral and Condition~1 holds without modification. 
At the operational level, network design problems built on system optimal dynamic traffic assignment (SO-DTA) with a linear programming (LP) formulation \citep{ziliaskopoulos2000linear, waller2001stochastic} satisfy Condition~1 and the conditions of Proposition~\ref{prop:suff3}.
Both applications studied in this paper fall within this class.
The first application is capacity expansion planning in shared autonomous vehicle (SAV) systems,
while the second application is  a route design problem for bus rapid transit (BRT) systems integrated with SAV operation.
Both FR-NDPs are built on SO-DTA models to describe dynamic traffic flow of SAVs, BRTs and travelers.
The flow of SAVs and travelers as well as capacity expansion decisions are expressed by continuous decision variables with continuous capacity states.
The flow of BRTs and their lane design are expressed by binary decision variables with binary design states indicating whether the lanes have already been constructed.
The tractable class therefore covers network design planning for the representative emerging transportation services that motivate this study, which indicates that the restrictions are not so severe as to exclude the problems of primary interest.

The principal limitation arises at the operational level.
Firstly, nonlinear, microscopic, or equilibrium-based traffic models are not handled exactly by the convergence theory.
For example, closely adhering to kinematic wave theory, such as non-vehicle holding, often makes the operational model nonconvex \citep{long2018dynamic}.
The LP formulation of SO-DTA serves as a representative linear approximation for nonconvex traffic dynamics.
In practice, McCormick envelopes are used to relax bilinear terms in electric vehicle operations (e.g., \citeauthor{lv2019optimal}, \citeyear{lv2019optimal}).
These relaxation techniques restore the convexity required by Condition~1.
Secondly, the operational decisions can restrict the strategic decisions that convergence theory can exactly address.
Condition~3 requires that, when integer variables are present, the state variables be integer.
Discrete operational decisions are natural in operational problems where integer decisions dominate, such as vehicle routing with pickup and delivery.
Continuous NDPs built on such operational problems  fall outside Condition~3.
One of the standard devices bringing such models into scope is binary approximation.
Under boundedness (implied by compactness in Condition~1), any bounded integer variable admits an exact binary representation, which yields convex piecewise linear value functions and underpins SDDP convergence for integer programs (Remark~1).
Binary approximation of continuous variables extends this to certain nonconvex problems to a prescribed precision (Theorem~4 of \citealp{zou2019stochastic}).

These devices are not without cost; hence, we state the trade-off explicitly.
Since convex relaxation enlarges the feasible region, the convergence guarantees apply to the relaxed model rather than to the original nonconvex one.
The relaxation gap must be assessed in applications where the non-convexity is material.
Binary approximation buys precision at the price of additional binary variables and a correspondingly larger state space,
giving a trade-off between accuracy and tractability depending on the resolution of the approximation.

Limitations also arise at the strategic level.
Several conditions may fail and call for a separate theoretical treatment.}
\begin{itemize}
    \setlength{\itemsep}{0pt} \setlength{\parskip}{0pt}
    \item Since the performance of long-term planning could depend on the planning horizon $\tau$, it makes sense to evaluate performance through discounted infinite-horizon MSSPs, which violate Condition 4.
    While no theoretical convergence guarantee exists for infinite-horizon MSSPs, \citet{shapiro2020periodical} and \citet{hole2024capacity} have demonstrated numerical convergence in practice with SDDP-type algorithms.
    \item When uncertain parameters follow continuous distributions, Condition 5 does not hold. In the risk-neutral setting,  \citet{shapiro2011analysis} provides theoretical properties of the sufficient sample size for forming a Sample Average Approximation (SAA) problem to solve the true problem with desired precision.
    \item In our FR-NDP setting, the operational uncertainty (e.g., daily demand) may depend on the strategic uncertainties (e.g., demand growth), which potentially violates stagewise independence, Condition 6.
    In the following section, we demonstrate how to address this issue in FR-NDP applications to SAV and BRT systems.
\end{itemize}

%% file: ch4-1.tex
\section{Applications}
\label{sec:4}

This section provides two applications of the proposed FR-NDP framework to representative emerging transportation services.
The first application addresses a capacity expansion problem for shared autonomous vehicle (SAV) systems.
We apply the proposed framework to a deterministic linear programming model proposed by \citet{seo2022multi}, which jointly optimizes SAV operations and infrastructure design.
The second application considers a route design problem for bus rapid transit (BRT) systems integrated with SAV operations (hereafter referred to as SAV-BRT systems).
We modify the SAV-BRT model proposed by \citet{maruyama2023integrated}, which integrates BRT route design and vehicle assignment into \citet{seo2022multi}.
The SAV and SAV-BRT models are presented in Appendix \ref{app:SAV} and \ref{app:SAV_BRT}.
For each application, we present the corresponding formulation and specify sufficient conditions for our theoretical properties to hold.
For the capacity expansion problems for SAV systems, we leverage the linearity of the problem to derive analytical properties of multi-stage investment planning under uncertainty.
The optimal strategic and operational decisions obtained from these FR-NDPs will be illustrated through numerical experiments in the next section.

\subsection{Application to network design for SAV systems}

\subsubsection{Problem setting}

This subsection describes the key strategic and operational problem setting of the FR-NDP for capacity expansion planning in SAV systems.
The following problem setting is based on \citet{seo2022multi}, with modifications to accommodate demand uncertainties and multi-stage infrastructure investment decisions.

The operational assumptions employed by the proposed model are as follows.
We consider a system in which only SAVs provide transportation services to travelers.
An SAV is defined as a driverless vehicle that is shared by a society and may be used by any traveler in the society.
Each SAV has a passenger capacity (e.g., they can transport two travelers simultaneously).
SAVs' travels are constrained by link and node capacities.
Node capacity can be interpreted as queueing capacity or parking capacity of SAVs.
SAVs are deployed at the beginning of each operational stage and run or park in the network throughout the corresponding stage.
A traveler is defined as a person who has a specific origin, destination, and desired arrival time; to use the service he/she submits such demand profiles to a centralized planner.
Such a demand profile is aggregated as an origin-destination (OD) matrix by a desired arrival time.
A traveler can move only when they ride an SAV; otherwise he/she must wait at a node.
At the beginning of each operational stage, the OD matrix is realized, and all such demands must be served during the corresponding operational stage.
The number of possible realizations of daily demand at each operational stage is finite.
The planner makes operational decisions to minimize operational costs incurred in satisfying travelers' demand.
The operational decisions include SAV routing (including waiting at nodes), vehicle-traveler assignment, and fleet size of SAVs\footnote{Although fleet sizing may be considered as strategic, we regard it as part of operational decisions as our primary focus lies on the multi-stage network design.}, based on the realized daily demand.
The planner makes dynamic operational decisions at multiple time steps within each operational stage.
The operational costs comprise travel and schedule costs of travelers, travel costs of SAVs, and maintenance costs of SAVs.

The strategic assumptions employed by the proposed model are as follows.
The planner formulates capacity expansion planning of links and nodes to minimize the sum of the current strategic cost and future costs, which are assessed using a convex and conditionally consistent risk measure.
Added link and node capacities are available throughout the planning horizon.
The length of the planning horizon is finite and given.
At the beginning of each strategic stage, a random demand growth is realized in the form of the expectation of daily demand for the subsequent operational stage.
The set of possible realizations of the demand growth at each strategic stage is finite.
Strategic cost is the sum of capacity expansion and maintenance costs.
Note that under the problem setting, Conditions 2, 4, and 5 in Proposition \ref{prop:suff} are satisfied.

\begin{table}[t]
        \centering
	\caption{List of variable notation for the SAV system}
	\label{tab:variable}       
	\scalebox{0.75}{
		\begin{tabular}{llll}
			\hline\noalign{\smallskip}
			type& notation & definition \\
			\noalign{\smallskip}\hline\noalign{\smallskip}
            state variable $\bm{x}_s$ & $\kappa_{s,ij}$ & traffic capacity of link $ij$ at stage $s$\\ 
            & $\kappa_{s,ii}$ & storage capacity of node $i$ at stage $s$\\ 
             &  $Q_{s,od}^{l}$ & expected value of travel demand with origin $o$, destination $d$, \\
             & & and desired arrival time step $l$, at stage $s$\\ \hline
             strategic variable $\bm{u}_s^{\rm F}$ &  $\hat{\kappa}_{s,ij}$ & expansion of traffic capacity of link $ij$ at stage $s$\\
              &  $\hat{\kappa}_{s,ii}$ & expansion of storage capacity of node $i$ at stage $s$\\\hline
			operational variable $\bm{u}_s^{\rm G}$ & $z_{s,ij}^{t}$ & flow of SAVs that start traveling link $ij$ on time step $t$ at stage $s$\\
            & $\hat{z}_{s,i}$ & the number of SAVs deployed on node $i$ at stage $s$\\
			& $y_{s,d,ij}^{l,t}$ & flow of travelers who start traveling link $ij$ on time step $t$, with destination $d$\\
			 & &and desired arrival time step $l$, at stage $s$ \\
			& $q_{s,d}^{l,t}$& the number of traveler arrivals on time step $t$, with destination $d$\\
			& &and desired arrival time step $l$, at stage $s$\\
			\noalign{\smallskip}\hline
	\end{tabular}}
\end{table}

\begin{table}[t]
	\caption{List of parameter notation for the SAV system}
	\centering
	\label{tab:parameter}       
	\scalebox{0.75}{
		\begin{tabular}{ll}
			\hline\noalign{\smallskip}
			notation & definition \\
			\noalign{\smallskip}\hline\noalign{\smallskip}
			$t_{ij}$ & free-flow travel time of link $ij$ if $i\neq{j}$\\
			$t_{ii}$ & waiting time at node $i$\\
            $f_{s,ij}$ & unit travel cost of travelers on link $ij$ at stage $s$\\
			$d_{s,ij}$ & unit travel cost of SAVs on link $ij$ at stage $s$\\
            $d_s$ & unit maintenance cost of SAVs at stage $s$\\
            $\epsilon_s$ & unit early schedule cost of travelers at stage $s$\\
            $\lambda_s$ & unit late schedule cost of travelers at stage $s$\\
			$c_{s,ij}$ & unit cost of maintaining infrastructure of link $ij$ during a single period at stage $s$\\
            $c_{s,ii}$ & unit cost of maintaining infrastructure of node $i$ during a single period at stage $s$\\
            $\hat{c}_{s,ij}$ & unit cost of expanding traffic capacity of link $ij$ at stage $s$\\
            $\hat{c}_{s,ii}$ & unit cost of expanding storage capacity of node $i$ at stage $s$\\
			$\rho$ & passenger capacity of an SAV\\
            $\mu_{s,od}^{l}$ & the long-term mean of travel demand with origin $o$, destination $d$, and desired arrival time step $l$\\
            $\phi_{s,od}^{l}$ & the reverting speed to the long-term mean $\mu_{s,od}^{l}$\\
            $\delta_{i,d}$ & Kronecker’s delta (i.e., 1 if $i=d$, zero otherwise)\\
			$\xi_{s,od}^{l, \rm F}$ & uncertain growth of travel demand with origin $o$, destination $d$, and desired arrival time step $l$, at stage $s$\\
            $\xi_{s,od}^{l, \rm G}$ & fluctuation of daily travel demand with origin $o$, destination $d$, and desired arrival time step $l$, at stage $s$\\
			\noalign{\smallskip}\hline
	\end{tabular}}
\end{table}

\subsubsection{Formulation}

The FR-NDP for the SAV system based on the above problem setting is fully specified, by instantiating the subproblem functions $C^{{\rm F}/{\rm G}}_s$, $T^{{\rm F}/{\rm G}}_s$, and $U^{{\rm F}/{\rm G}}_s$ as follows:
\begin{subequations}
\label{eq:SAV-NR-NDP}
    \begin{flalign}
        \label{eq:SAV_strategic_cost}
    	&C^{\rm F}_s:=\hat{\bm{c}}_s^{\top}\hat{\bm{\kappa}}_{s}+\bm{c}_s^{\top}\bm{\kappa}_s&\forall s\in\mathcal{S}^{\rm F},\\
        \label{eq:SAV_strategic_transition}
    	&T_s^{\rm F}:=\{\bm{\kappa}_{s-1}+\hat{\bm{\kappa}}_{s},-\bm{\phi}^{\top}(\bm{Q}_{s-1} - \bm{\mu}) + \bm{\xi}_s^{\rm F}\}&\forall s\in\mathcal{S}^{\rm F},\\
            \label{eq:SAV_operational_cost}
    	&C^{\rm G}_s:=\sum_t\bm{f}_s^{\top}\bm{y}_{s}^t+\epsilon_s\sum_{t< l}\bm{1}^{\top}\bm{q}_s^t+\lambda_s\sum_{t>l}\bm{1}^{\top}\bm{q}_s^t+\sum_t\bm{d}_s^{\top}\bm{z}_{s}^t+d_s\bm{1}^{\top}\hat{\bm{z}}_{s}&\forall s\in\mathcal{S}^{\rm G},\\
            \label{eq:SAV_operational_transition}
            &T_s^{\rm G}:=\{\bm{\kappa}_{s-1},\bm{Q}_{s-1}\}&\forall s\in\mathcal{S}^{\rm G},\\
            \label{eq:SAV_vehicle_flow_dynamics}
    	&U_s^{\rm G}:=\Bigl\{\sum_j z_{s,ji}^{t-t_{ji}}=\sum_j z_{s,ij}^{t} - \delta_{t1} \hat{z}_{s,i}& \forall t,i,s\in\mathcal{S}^{\rm G}, \\
            \label{eq:SAV_passenger_flow_dynamics}
    	&~~~~~~~~~~~~~\sum_j y_{s,d,ji}^{l,t-t_{ji}}=\sum_j y_{s,d,ij}^{l,t} + \delta_{id} q_{s,d}^{l,t}& \forall t,d,l,i,s\in\mathcal{S}^{\rm G}, \\
            \label{eq:SAV_passenger_arivals}
    	&~~~~~~~~~~~~~\sum_t q_{s,d}^{l,t} = \sum_o (Q_{s,od}^{l} +\xi_{s,od}^{l,\rm G})&\forall d,l,s\in\mathcal{S}^{\rm G},\\
            \label{eq:SAV_vehicle_capacity}
    	&~~~~~~~~~~~~~\sum_{d,l}{y}_{s,d,ij}^{l,t}\le \rho{z}_{s,ij}^t &\forall t,ij,s\in\mathcal{S}^{\rm G},\\
            \label{eq:SAV_network_capacity}
    	&~~~~~~~~~~~~~{z}_{s,ij}^t\le {\kappa}_{s,ij} &\forall t,ij,s\in\mathcal{S}^{\rm G}\\
    	\nonumber
    	&\Bigr\},&
    \end{flalign}
\end{subequations}
where all the variables are nonnegative and continuous.
The notations in the scalar form are listed in Tables \ref{tab:variable} and \ref{tab:parameter}. 

The primary modifications from \citet{seo2022multi} to capture demand uncertainty and multi-stage network design decisions can be seen in Eq. (\ref{eq:SAV_strategic_transition}) and the right-hand side of Eq. (\ref{eq:SAV_passenger_arivals}).
Eq. (\ref{eq:SAV_strategic_transition}) represents capacity expansion and demand growth.
The demand growth is modeled by a mean‐reverting process with the long-term mean $\bm{\mu}$; note, however, that this serves only as a representative example of stagewise dependent demand growth.
The term $Q_{s,od}^l+\xi_{s,od}^{l,{\rm G}}$ on the right-hand side of Eq. (\ref{eq:SAV_passenger_arivals}) denotes the total daily demand in an operational stage: the first term is the expected daily demand realized at the preceding strategic stage, and the second term captures the random fluctuations in daily demand.

The meanings of the remaining functions are as follows.
Eqs. (\ref{eq:SAV_strategic_cost}) and (\ref{eq:SAV_operational_cost}) are strategic and operational costs.
The first and second terms of Eq. (\ref{eq:SAV_strategic_cost}) are capacity expansion and
maintenance costs.
Each term in Eq. (\ref{eq:SAV_operational_cost}) corresponds, in order, to 
the total travel cost of travelers, 
the total early schedule cost of travelers, 
the total late schedule cost of travelers, 
the total travel cost of SAVs, and
the total maintenance cost of SAVs.
Operational action space $U^{\rm G}_s$ is constructed by
Eqs. (\ref{eq:SAV_vehicle_flow_dynamics})--(\ref{eq:SAV_network_capacity}).
Eq. (\ref{eq:SAV_vehicle_flow_dynamics}) is the flow conservation of SAVs.
Eq. (\ref{eq:SAV_passenger_flow_dynamics}) is the
flow conservation of travelers.
Eq. (\ref{eq:SAV_passenger_arivals}) is the flow conservation of demand.
Eqs. (\ref{eq:SAV_vehicle_capacity}) and (\ref{eq:SAV_network_capacity}) are passenger capacity and traffic capacity constraints, respectively.
For detailed descriptions of the constraints, see \citet{seo2022multi}.

Eq. (\ref{eq:SAV-NR-NDP}) is linear and all the variables are continuous; thus, Conditions 1, 3, 7, and 9 in Propositions \ref{prop:suff} and \ref{prop:suff3} are satisfied.
Furthermore, because the aforementioned problem setting meets Conditions 2, 4, and 5, we obtain the following corollary:
\begin{coro}
\label{coro:SAV1}
If realizations of $\bm{\xi}^{\rm F}$ and $\bm{\xi}^{\rm G}$ are mutually independent (i.e., stagewise independent), 
the proposed FR-NDP for the SAV system in Eq. (\ref{eq:SAV-NR-NDP}) satisfies sufficient conditions stated in Proposition \ref{prop:suff}.
\end{coro}
Furthermore, we immediately obtain the following corollary:
\begin{coro}
\label{coro:SAV2}
If realizations of $\bm{\xi}^{\rm F}$ and $\bm{\xi}^{\rm G}$ are mutually independent (i.e., stagewise independent) and the objective function is given in the form of $\mathbb{ECRM}$ with $\mathbb{CV@R}$ as in Eq. (\ref{eq:ECRM}), the proposed FR-NDP for the SAV system in Eq. (\ref{eq:SAV-NR-NDP}) satisfies sufficient conditions stated in Proposition \ref{prop:suff3}.
\end{coro}

Leveraging the linearity of the proposed problem, we can derive theoretical results on the cost-recovery principles for infrastructure investments.
The proof is deferred to Appendix \ref{app:prop4}; letting $\bm{p}^*_s$ denote the optimal dual variables corresponding to traffic capacity constraints (\ref{eq:SAV_network_capacity}), we obtain the following proposition:
\begin{prop}
    \label{prop:cost_principle}
    If realizations of $\bm{\xi}^{\rm F}$ and $\bm{\xi}^{\rm G}$ are mutually independent (i.e., stagewise independent) and the conditional risk measure $\mathbb{F}_{\bm{\xi}}$ is $\mathbb{E}_{\bm{\xi}}$, we obtain
    \begin{flalign}
    	\nonumber
        &c^{}_{1,ij}\kappa^*_{1,ij}+\hat{c}^{}_{1,ij}\hat{\kappa}^*_{1,ij}+\mathbb{E}_{\xi_2^{\rm G}}\left[c^{}_{3,ij}\kappa^*_{3,ij}+\hat{c}_{3,ij}\hat{\kappa}^*_{3,ij}+\mathbb{E}_{\xi_4^{\rm G}}\left[\cdots+\mathbb{E}_{\xi_{S-2}^{\rm G}}\left[c^{}_{S-1,ij}\kappa^*_{S-1,ij}+\hat{c}^{}_{S-1,ij}\hat{\kappa}^*_{S-1,ij}\right]\right]\right]&
    \end{flalign}
    \vspace{-3.2ex}
    \begin{flalign}
    	\label{eq:self_financing}
    	&=\mathbb{E}_{\xi_1^{\rm F}}\left[\sum_t p^{t,*}_{2,ij}z^{t,*}_{2,ij}+\mathbb{E}_{\xi_3^{\rm F}}\left[\cdots+\mathbb{E}_{\xi_{S-1}^{\rm F}}\left[\sum_t p^{t,*}_{S,ij}z^{t,*}_{S,ij}\right]\right]\right]~~~~~~~~~~~~~~~~~~~~~~~~~~~~~~~~~~~~~~~~~~~~~~~~~~~~~~~~~~~\forall ij.&	
    \end{flalign}
\end{prop}
Given the interpretation that the optimal dual variables $\bm{p}^*_s$ can be interpreted as optimal congestion tolls, Proposition \ref{prop:cost_principle} represents the following cost-recovery principle: the expected total infrastructure expansion and maintenance costs (on the left-hand side of Eq. (\ref{eq:self_financing})) equal the expected total congestion tolls paid by SAVs (on the right-hand side of Eq. (\ref{eq:self_financing})).

While a proof of this interpretation is left for future work, it has been widely discussed in the literature on the duality of SO-DTA problems (e.g., \citeauthor{akamatsu2017tradable}, \citeyear{akamatsu2017tradable}).
The operational problem in our proposed framework builds upon the SO-DTA formulation introduced by \citet{seo2022multi}.
Our preliminary work \citep{seo2024dynamic} proved the aforementioned interpretation of $\bm{p}^*_s$ corresponding to traffic capacity constraints (\ref{eq:SAV_network_capacity}).
This result is derived from the duality of the SO-DTA for SAV systems.
Our analytical approach to the infrastructure cost-recovery principle, which leverages the mathematical tractability of the MSSP framework, could be extended beyond SAV systems to a broad class of network flow problems.

%% file: ch4-3.tex
\subsection{Application to sequential route design for SAV-BRT systems}

\begin{table}[t]
	\centering
	\caption{List of variable notation for the SAV-BRT system}
	\label{tab:variable2}       
	\scalebox{0.85}{
		\begin{tabular}{llll}
			\hline\noalign{\smallskip}
			type& notation & definition \\
			\noalign{\smallskip}\hline\noalign{\smallskip}
			state variable $\bm{x}_s$ & $g^m_{s,ij}$ & the value 1 if the link $ij$ is designated as a part of $m$-th BRT route;\\ 
			&  & 0 otherwise, at stage $s$\\ 
			& $a^m_{s,i}$ & the value 1 if the node $i$ is designated as a starting point of $m$-th BRT route;\\ 
			&  & 0 otherwise, at stage $s$\\ 
			& $b^m_{s,i}$ & the value 1 if the node $i$ is designated as an end point of $m$-th BRT route;\\ 
			&  &  0 otherwise, at stage $s$\\ 
			&  $Q_{s,od}^{l}$ & expected value of travel demand with origin $o$, destination $d$, \\
			& & and desired arrival time step $l$, at stage $s$\\ \hline
			strategic variable $\bm{u}_s^{\rm F}$ &  $\hat{g}^m_{s,ij}$ & the value 1 if a $m$-th BRT lane is installed at the link $ij$;\\
			&   & 0 otherwise, at stage $s$\\
			&  $\hat{a}^m_{s,i}$ & the value 1 if a starting point of $m$-th BRT route is installed at the node $i$; \\
			&   & 0 otherwise, at stage $s$\\
			&  $\hat{b}^m_{s,i}$ & the value 1 if an end point of $m$-th BRT route is installed at the node $i$;\\
			&   & 0 otherwise, at stage $s$\\\hline
			operational variable $\bm{u}_s^{\rm G}$ & 
			$w_{s,ij}^{m,t}$ & flow of BRTs that start traveling link $ij$ on $m$-th BRT route on time step $t$ at stage $s$\\
			& $\hat{w}_{s,0i}^{m,t}$ & flow of BRTs that depart at $m$-th BRT route on time step $t$ at stage $s$\\
			& $\hat{w}_{s,i0}^{m,t}$ & flow of BRTs that arrive at $m$-th BRT route on time step $t$ at stage $s$\\
			& $z_{s,ij}^{t}$ & flow of SAVs that start traveling link $ij$ on time step $t$ at stage $s$\\
			& $\hat{z}_{s,i}$ & the number of SAVs deployed on node $i$ at stage $s$\\
			& $y_{s,d,ij}^{l,t}$ & flow of travelers who start traveling link $ij$ by SAV on time step $t$, with destination $d$\\
			& &and desired arrival time step $l$, at stage $s$ \\
			& $\tilde{y}_{s,d,ij}^{l,t}$ & flow of travelers who start traveling link $ij$ by BRT on time step $t$, with destination $d$\\
			& &and desired arrival time step $l$, at stage $s$ \\
			& $v_{s,d,i}^{l,t}$ & flow of travelers who start transferring from BRT to SAV at node $i$
			on time step $t$,\\
			& &with destination $d$ and desired arrival time step $l$, at stage $s$ \\
			& $\hat{v}_{s,d,i}^{l,t}$ & flow of travelers who start transferring from SAV to BRT at node $i$
			on time step $t$, \\
			& &with destination $d$ and desired arrival time step $l$, at stage $s$ \\
			& $q_{s,d}^{l,t}$& the number of traveler arrivals by SAV on time step $t$, with destination $d$\\
			& &and desired arrival time step $l$, at stage $s$\\
			& $\tilde{q}_{s,d}^{l,t}$& the number of traveler arrivals by BRT on time step $t$, with destination $d$\\
			& &and desired arrival time step $l$, at stage $s$\\
			\noalign{\smallskip}\hline
	\end{tabular}}
\end{table}

\begin{table}[tbhp]
	\caption{List of parameter notation for the SAV-BRT system}
	\centering
	\label{tab:parameter2}  
	\scalebox{0.85}{
		\begin{tabular}{ll}
			\hline\noalign{\smallskip}
			notation & definition \\
			\noalign{\smallskip}\hline\noalign{\smallskip}
			$\tilde{t}_{ij}$ & BRT's free-flow travel time of link $ij$ if $i\neq{j}$\\
			$\tilde{t}_{ii}$ & transferring time at node $i$\\
			$\tilde{f}_{s,ij,d}^{l}$ & unit travel cost of travelers on BRT, with destination $d$ and desired arrival time step $l$, on link $ij$ at stage $s$\\
			$\tilde{d}_{s,ij}$ & unit cost of operating BRTs on link $ij$ at stage $s$\\
			$\tilde{d}_s$ & unit maintenance cost of BRTs at stage $s$\\
			$c^{\rm BRT}_{s,ij}$ & unit cost of maintaining infrastructure of link $ij$ equipped with a BRT lane during a single period at stage $s$\\
			$\hat{c}^{\rm BRT}_{s,ij}$ & unit cost of installing a BRT lane to link $ij$ at stage $s$\\
			$\tilde{\rho}$ & passenger capacity of a BRT\\
			$\tilde{\kappa}_{ij}$ & number of SAVs which will be occupied by installing a BRT lane at the link $ij$\\
			\noalign{\smallskip}\hline
	\end{tabular}}
\end{table}

\subsubsection{Problem statement}

This subsection describes the key strategic and operational problem settings for the sequential route design problem for SAV-BRT systems.
As in the first application, we modify the original setting stated in \citet{maruyama2023integrated} to accommodate demand uncertainty and multi-stage network design.
Unless noted otherwise, all definitions, notation, and assumptions carry over from the first SAV application and are not repeated here.

The operational assumptions employed by the proposed model are as follows.
BRTs travel only on links equipped with designated BRT lanes.
BRTs are free from traffic congestion, while their average link travel time is slower than that of SAVs because BRTs stop at nodes on designated routes for travelers' transfers.
The passenger capacity of BRTs is generally larger than that of SAVs.
A traveler can move to their destination only by SAV or BRT.
Transfers between SAV and BRT are allowed at nodes on designated BRT routes with a given transfer time.
The operational decisions newly include the BRTs' schedule (such as frequency and running interval) and routing.
The operational cost components are extended to include travel costs and maintenance costs of BRTs.

The strategic assumptions employed by the proposed model are as follows.
The strategic decision is the BRT route design.
The maximum number of BRT routes on the network is given.
Unlike in the first application, link and node capacities are considered as exogenous parameters.
All links assigned to BRT routes are equipped with dedicated BRT lanes.
BRT lanes occupy part of the link, thereby reducing the link capacity available to SAVs.
Once a BRT lane is equipped on a link, it remains in service throughout the planning horizon.
The strategic costs consist of the investment costs for constructing BRT lanes.

\subsubsection{Formulation}

The FR-NDP for the SAV-BRT system based on the above problem setting is fully specified, by instantiating the subproblem functions $C^{{\rm F}/{\rm G}}_s$, $T^{{\rm F}/{\rm G}}_s$, and $U^{{\rm F}/{\rm G}}_s$ as follows:
\begin{subequations}
    \label{eq:BRT-FR-NDP}
    \begin{flalign}
        \label{eq:BRT_strategic_cost}
    	&C^{\rm F}_s:=(\bm{c}_s^{\rm BRT})^{\top}\sum_m\bm{g}^m_s+(\hat{\bm{c}}_s^{\rm BRT})^{\top}\sum_m\hat{\bm{g}}^m_{s}&\forall s\in\mathcal{S}^{\rm F},\\
        \label{eq:BRT_strategic_transition}
    	&T_s^{\rm F}:=\Bigl\{\hat{\bm{a}}_{s},\hat{\bm{b}}_{s},\bm{g}_{s-1}+\hat{\bm{g}}_{s},\bm{Q}_{s-1}+\bm{\xi}_s^{\rm F}|&\forall s\in\mathcal{S}^{\rm F},\\
        \label{eq:route}
    	&~~~~~~~~~~~~~{a}_{s,i}^m+\sum_j {g}_{s,ji}^{m}={b}_{s,i}^m+\sum_j {g}_{s,ij}^{m}&\forall m,i,s\in\mathcal{S}^{\rm F},\\
        \label{eq:BRT_binary1}
    	&~~~~~~~~~~~~~\bm{a}_s\in\{\bm{0},\bm{1}\},\bm{b}_s\in\{\bm{0},\bm{1}\},\bm{g}_s\in\{\bm{0},\bm{1}\}&\forall s\in\mathcal{S}^{\rm F},\\
    	\nonumber
    	&\Bigr\},\\
    	&U_s^{\rm F}:=\Bigl\{ \sum_i\hat{a}_{s,i}^{m}= 1& \forall m,s\in\mathcal{S}^{\rm F}, \\
    	&~~~~~~~~~~~~~\sum_i\hat{b}_{s,i}^{m}= 1& \forall m,s\in\mathcal{S}^{\rm F}, \\
        \label{eq:BRT_binary2}
    	&~~~~~~~~~~~~~\bm{\hat{a}}_s\in\{\bm{0},\bm{1}\},\bm{\hat{b}}_s\in\{\bm{0},\bm{1}\},\bm{\hat{g}}_s\in\{\bm{0},\bm{1}\} &\forall s\in\mathcal{S}^{\rm F}, &\\
    	\nonumber
    	&\Bigr\},&
    \end{flalign}
    \vspace{-4ex}
    \begin{flalign}
    	\nonumber
    	&C^{\rm G}_s:=\left(\sum_t\bm{f}_s^{\top}\bm{y}_{s}^t+\sum_t\tilde{\bm{f}}_s^{\top}\tilde{\bm{y}}_{s}^t\right)
    	+\left(\epsilon_s\sum_{t< l}\bm{1}^{\top}\bm{q}_s^t+\lambda_s\sum_{t>l}\bm{1}^{\top}\bm{q}_s^t+\epsilon_s\sum_{t< l}\bm{1}^{\top}\tilde{\bm{q}}_s^t+\lambda_s\sum_{t>l}\bm{1}^{\top}\tilde{\bm{q}}_s^t\right)&
    \end{flalign}
    \vspace{-4ex}
    \begin{flalign}
    	\label{eq:BRT_operational_cost}
    	&~~~~~~~~~~+\left(\sum_t\bm{d}_s^{\top}\bm{z}_{s}^t+\sum_t\tilde{\bm{d}}^{\top}\sum_m\bm{w}_{s}^{m,t}\right)
    	+\left({d}_s\bm{1}^{\top}\hat{\bm{z}}_{s}+\tilde{d}_s\bm{1}^{\top}\sum_{t,m}\hat{\bm{w}}_{s}^{m,t}\right)&\forall s\in\mathcal{S}^{\rm G},\\
    	\label{eq:BRT_operational_transition}
    	&T_s^{\rm G}:=\{{\bm{a}}_{s-1},{\bm{b}}_{s-1},\bm{g}_{s-1},\bm{Q}_{s-1}\}&\forall s\in\mathcal{S}^{\rm G},\\
    	\label{eq:BRT_SAV_vehicle_flow_dynamics}
    	&U_s^{\rm G}:=\Bigl\{\sum_j z_{s,ji}^{t-t_{ji}}=\sum_j z_{s,ij}^{t} - \delta_{t1} \hat{z}_{s,i}& \forall t,i,s\in\mathcal{S}^{\rm G}, \\
    	\label{eq:BRT_BRT_vehicle_flow_dynamics}
    	&~~~~~~~~~~~~~\sum_j w_{s,ji}^{m,t-\tilde{t}_{ji}}=\sum_j w_{s,ij}^{m,t} - \hat{w}_{s,0i}^{m,t} + \hat{w}_{s,i0}^{m,t}& \forall m,t,i,s\in\mathcal{S}^{\rm G}, \\
    	\label{eq:BRT_SAV_passenger_flow_dynamics}
    	&~~~~~~~~~~~~~v_{s,d,i}^{l,t-\tilde{t}_{ii}} + \sum_j y_{s,d,ji}^{l,t-t_{ji}}=\hat{v}_{s,d,i}^{l,t} + \sum_j y_{s,d,ij}^{l,t} + \delta_{id} q_{s,d}^{l,t}& \forall t,d,l,i,s\in\mathcal{S}^{\rm G}, \\
    	\label{eq:BRT_BRT_passenger_flow_dynamics}
    	&~~~~~~~~~~~~~\hat{v}_{s,d,i}^{l,t-\tilde{t}_{ii}} + \sum_j \tilde{y}_{s,d,ji}^{l,t-\tilde{t}_{ji}}={v}_{s,d,i}^{l,t} + \sum_j \tilde{y}_{s,d,ij}^{l,t} + \delta_{id} \tilde{q}_{s,d}^{l,t}& \forall t,d,l,i,s\in\mathcal{S}^{\rm G}, \\
    	\label{eq:BRT_SAV_passenger_arivals}
    	&~~~~~~~~~~~~~\sum_t q_{s,d}^{l,t} + \sum_t \tilde{q}_{s,d}^{l,t} = \sum_o (Q_{s,od}^{l} +\xi_{s,od}^{l,\rm G})&\forall d,l,s\in\mathcal{S}^{\rm G},\\
    	\label{eq:BRT_SAV_vehicle_capacity}
    	&~~~~~~~~~~~~~\sum_{d,l}{y}_{s,d,ij}^{l,t}\le \rho{z}_{s,ij}^t &\forall t,ij,s\in\mathcal{S}^{\rm G},\\
    	\label{eq:BRT_BRT_vehicle_capacity}
    	&~~~~~~~~~~~~~\sum_{d,l}{\tilde{y}}_{s,d,ij}^{l,t}\le \sum_m\tilde{\rho}{w}_{s,ij}^{t,m} &\forall t,ij,s\in\mathcal{S}^{\rm G},\\
    	\label{eq:BRT_SAV_network_capacity}
    	&~~~~~~~~~~~~~{z}_{s,ij}^t\le {\kappa}_{s,ij} - \tilde{\kappa}_{ij}g_{s,ij}^m &\forall m,t,ij,s\in\mathcal{S}^{\rm G},
 	    \end{flalign}
		\vspace{-4ex}
		\begin{flalign}
    	\label{eq:BRT_BRT_route}
    	&~~~~~~~~~~~~~0\le{w}_{s,ij}^{m,t}\le {g}_{s,ij}^m &\forall m,t,ij,s\in\mathcal{S}^{\rm G},\\
    	\label{eq:BRT_BRT_start}
    	&~~~~~~~~~~~~~0\le{w}_{s,0i}^{m,t}\le {a}_{s,i}^m &\forall m,t,i,s\in\mathcal{S}^{\rm G},\\
    	\label{eq:BRT_BRT_end}
    	&~~~~~~~~~~~~~0\le{w}_{s,i0}^{m,t}\le {b}_{s,i}^m &\forall m,t,i,s\in\mathcal{S}^{\rm G},\\
    	\label{eq:BRT_subtour}
    	&~~~~~~~~~~~~~{w}_{s,ij}^{m,1}=0 &\forall m,ij,s\in\mathcal{S}^{\rm G},\\
    	\label{eq:BRT_binary3}
    	&~~~~~~~~~~~~~\bm{w}_s\in\{\bm{0},\bm{1}\},\bm{\hat{w}}_s\in\{\bm{0},\bm{1}\} &\forall s\in\mathcal{S}^{\rm G}\\
    	\nonumber
    	&\Bigr\},&
    \end{flalign}
\end{subequations}
where, except for the binary variables defined in Eqs. (\ref{eq:BRT_binary1}), (\ref{eq:BRT_binary2}), and (\ref{eq:BRT_binary3}), all the variables are nonnegative and continuous.
The notations in scalar form are listed in Tables \ref{tab:variable2} and \ref{tab:parameter2}.
For the latter, parameter notations common to the first application are omitted.

The multi-stage BRT route design at stage $s$ is represented by $g_{s-1}$ and $\hat{g}_{s}$. $g_{s-1}$ denotes the BRT routes already installed up to the previous stage $s-1$;  $\hat{g}_{s}$ indicates the newly installed BRT lanes.
Since the start and end points of BRT routes change as the routes are extended, the route design for the start and end points is expressed solely by $\hat{a}_{s}$ and $\hat{b}_{s}$, without using state variables ${a}_{s-1}$ and ${b}_{s-1}$ at the previous stage $s-1$.

Eqs. (\ref{eq:BRT_strategic_cost}) and (\ref{eq:BRT_operational_cost}) are strategic and operational costs.
The first and second terms of Eq. (\ref{eq:BRT_strategic_cost}) are maintenance and installation costs for BRT lanes.
Each term in Eq. (\ref{eq:BRT_operational_cost}) corresponds, in order, to 
the total travel cost of travelers, 
the total schedule cost of travelers, 
the total travel cost of SAVs and BRTs, and
the total maintenance of SAVs and BRTs.
Operational action space $U^{\rm G}_s$ is constructed by Eqs. (\ref{eq:BRT_SAV_vehicle_flow_dynamics})--(\ref{eq:BRT_binary3}).
Eqs. (\ref{eq:BRT_SAV_vehicle_flow_dynamics}) and (\ref{eq:BRT_BRT_vehicle_flow_dynamics}) are the flow conservation of SAVs and BRTs.
Eqs. (\ref{eq:BRT_SAV_passenger_flow_dynamics}) and (\ref{eq:BRT_BRT_passenger_flow_dynamics}) are the flow conservation of travelers riding on SAVs and BRTs.
Eqs. (\ref{eq:BRT_SAV_vehicle_capacity}) and (\ref{eq:BRT_BRT_vehicle_capacity}) are passenger capacity constraints of SAVs and BRTs.
Eq. (\ref{eq:BRT_SAV_network_capacity}) represents traffic capacity constraints of SAVs.
Eq. (\ref{eq:BRT_BRT_route}) means that BRTs can travel on the links equipped with a BRT lane.
Eqs. (\ref{eq:BRT_BRT_start}) and (\ref{eq:BRT_BRT_end}) mean that BRTs depart from the starting point and terminate at the end point.
Eq. (\ref{eq:BRT_subtour}) is needed for the removal of subtour.
For detailed descriptions of the constraints, see \citet{maruyama2023integrated}.

Eq. (\ref{eq:BRT-FR-NDP}) is linear; thus, Condition 1 in Proposition \ref{prop:suff} is satisfied.
Furthermore, because the aforementioned problem setting meets Conditions 2, 4, and 5, we obtain the following corollary:
\begin{coro}
    \label{coro:BRT}
    Under the two following assumptions, the proposed FR-NDP for the SAV-BRT system in Eq. (\ref{eq:BRT-FR-NDP}) satisfies sufficient conditions stated in Proposition \ref{prop:suff}:
    \begin{itemize}
    \setlength{\itemsep}{0pt} \setlength{\parskip}{0pt}
        \item [$\bullet$] Realizations of $\bm{\xi}^{\rm F}$ and $\bm{\xi}^{\rm G}$ are mutually independent (i.e., stagewise independent), and
        \item [$\bullet$] $\bm{Q}_{s}$ and $\bm{\xi}_s^{\rm F}$ are integer.
    \end{itemize}
\end{coro}

\subsubsection{Optional extensions of the SAV-BRT formulation}

{The SAV-BRT formulation in Eq. \eqref{eq:BRT-FR-NDP} is presented in a minimal form to demonstrate how the proposed FR-NDP framework embeds an existing integrated SAV-BRT model \citep{maruyama2023integrated}.
The following examples briefly illustrate three extensions a planner might wish to impose: decommissioning BRT lanes, precluding overlap between routes, and confining SAVs to a first/last-mile feeder role.

Firstly, to let the design contract as demand declines or shifts, we introduce a decommissioning decision $\check{g}^{m}_{s,ij}\in\{0,1\}$, equal to one if the $m$-th BRT lane on link $ij$ is removed at stage $s$.
The route state transition then becomes
\begin{subequations}
    \label{eq:decom}
    \begin{flalign}
    &g^{m}_{s} = g^{m}_{s-1} + \hat{g}^{m}_{s} - \check{g}^{m}_{s}
       & \forall m,s\in\mathcal{S}^{F}, \label{eq:decom-trans}\\
    &\check{g}^{m}_{s,ij} \le g^{m}_{s-1,ij}
       & \forall m,ij,s\in\mathcal{S}^{F}, \label{eq:decom-exist}\\
    &\hat{g}^{m}_{s,ij} + \check{g}^{m}_{s,ij}\le 1
       & \forall m,ij,s\in\mathcal{S}^{F}, \label{eq:decom-excl}\\
    &\check{g}_{s}\in \{0,1\}
       & \forall s\in\mathcal{S}^{F}. \label{eq:decom-bin}
    \end{flalign}
\end{subequations}
Eq.~(\ref{eq:decom-exist}) permits removal only of lanes already in service, 
and Eq.~(\ref{eq:decom-excl}) forbids installing and removing the same lane within a stage.
Eq.~(\ref{eq:route}), which is imposed on $g^{m}_{s}$ at every stage, ensures that the surviving lanes still form admissible routes after a removal. 

Secondly, 
if a planner instead wishes to enforce spatial route diversity, the following linear constraint can be added:
\begin{flalign}
    \label{eq:nonoverlap_brt}
    &\sum_{m} g^m_{s,ij} \leq 1& \forall ij,s\in\mathcal{S}^F .
\end{flalign}
This constraint requires each link to belong to at most one BRT route at each strategic stage. 

Thirdly, 
to describe a canonical multimodal system where SAVs serve as first/last-mile feeders and BRT handles line-haul services,
we replace the single SAV traveler flow $y^{l,t}_{s,d,ij}$ by two types of flows,
an access flow $y^{\mathrm{ac},l,t}_{s,d,ij}$ and an egress flow $y^{\mathrm{eg},l,t}_{s,d,ij}$.
The access flow enters at the origin as the SAV traveler flow, and its only outflow is a transfer to BRT.
This movement can be expressed by removing the arrival term $\delta_{id}q^{l,t}_{s,d}$ as follows:
\begin{flalign}
    &\sum_{j} y^{\mathrm{ac},\,l,\,t-t_{ji}}_{s,d,ji}
    \;=\;
    \hat{v}^{\,l,t}_{s,d,i} + \sum_{j} y^{\mathrm{ac},\,l,t}_{s,d,ij}
    \qquad &\forall t,d,l,i,\, s\in\mathcal{S}^{G}.
    \label{eq:ext-feeder-access}
\end{flalign}
The egress flow enters via a transfer from BRT and terminates only at the destination.
This movement can be expressed by removing the transfer term $\hat{v}^{\,l,t}_{s,d,i}$ as follows:
\begin{flalign}
    &v^{\,l,\,t-\tilde{t}_{ii}}_{s,d,i} + \sum_{j} y^{\mathrm{eg},\,l,\,t-t_{ji}}_{s,d,ji}
    \;=\;
    \sum_{j} y^{\mathrm{eg},\,l,t}_{s,d,ij} + \delta_{id}\, q^{\,l,t}_{s,d}
    \qquad &\forall t,d,l,i,\, s\in\mathcal{S}^{G}.
    \label{eq:ext-feeder-egress}
\end{flalign}
Since the transfer flow from SAV to BRT, $\hat{v}$, appears only in the flow conservation of access travelers and
the transfer flow from BRT to SAV, ${v}$, appears only in the flow conservation of egress travelers, every traveler must traverse a BRT lane.
A pure SAV door-to-door trip is therefore infeasible by construction.

The above extensions only append linear constraints or binary variables to the original SAV-BRT formulation.
The transition functions $T^{F/G}_{s}$ remain linear and the action spaces $U^{F/G}_{s}$ remain non-empty compact mixed-integer polyhedral sets.
Consequently, Conditions~1--6 of Proposition~1 continue to hold, and Corollary~\ref{coro:BRT} guarantees the same almost-sure convergence as in our numerical experiments.}

%% file: ch5.tex
\section{Numerical experiments}
\label{sec:5}

This section presents numerical experiments to demonstrate the effectiveness of our framework.
{
Our numerical experiments have two objectives.
{First, we illustrate the convergence behavior of the lower and upper bounds predicted by Propositions 1–3.}
Second, we examine how flexibility (multi-stage strategic decisions) and reliability (risk-averse evaluation) affect strategic and operational decisions and the resulting system performance.

We emphasize that the following numerical experiments were not intended to evaluate computational scalability.
The primary contribution of the proposed framework is to obtain the policies accompanied by bounds that provably almost surely converge to the optimal value, without explicitly enumerating the full scenario tree.
Although computational scalability {with respect to the network size and the horizon length} is sacrificed, 
our numerical results can provide valuable insights into the necessity of jointly incorporating flexibility and reliability in transportation network design under optimality guarantees.
{We should note that the quantitative aspects of convergence such as the number of iterations or the computation time required to converge are left to future work. }
}

\subsection{Numerical settings}

We considered the Midtown Manhattan case shown in Figure \ref{fig:network}.
The initial value of travel demand was generated from the New York City (NYC) taxi data
between 8:00 and 9:00 on 2019-04-01 (Monday) in the area.
The total number of travelers was 5,357.
The desired arrival time for all travelers was assumed as 8:30.
Figure \ref{fig:midtown} shows the spatial distribution of passenger demand.
The network was constructed according to \citet{seo2022multi}.
{Figure \ref{fig:m_net} shows the network considered in our numerical experiments.
The number of nodes and links was 19 and 64, respectively.}
The free-flow link travel time of SAVs and BRTs was given by 5 and 10 min, respectively.
$\bm{c}_s$ and $\hat{\bm{c}}_s$ were constant during the planning horizon and were determined as proportional to land values (see Figure \ref{fig:land}).
For computational tractability, the feasible regions of capacity expansion were given by $\bm{\kappa}_{s,ii}\in[4, 120]$ and $\bm{\kappa}_{s,ij}\in[4, 60]$.
The capacity of the nodes and links in the SAV-BRT problem was given by the optimal solution of capacity expansion in the SAV system.
{The number of BRT routes was set as $m=2$.}
The other parameters were listed in Table \ref{tab:parameter_settings}.

The FR-NDPs for the SAV and SAV-BRT systems were replaced by SAA problems, with 30 samples generated from nominal distributions at each stage.
The length of the planning horizon was given by $\tau=5$.
Since the number of stages was $2\tau=10$, the total number of scenarios during the planning horizon was $30^9\simeq2\times10^{13}$.
{Although the exact number of variables differs slightly across stages, the strategic and operational subproblems for the SAV system contain approximately \(2.5\times 10^{2}\) and \(1.8\times 10^{4}\) decision variables, respectively, with approximately \(8.0\times 10^{1}\) state variables. 
For the SAV-BRT system, the corresponding numbers are approximately \(6.0\times 10^{2}\), \(4.0\times 10^{4}\), and \(2.0\times 10^{2}\), respectively. 
Thus, the largest stagewise subproblem solved by SDDP is on the order of \(10^{4}\) variables, and the total size of the overall problem with 10 stages is on the order of \(10^{5}\) variables}\footnote{
    Another approach to provide optimality guarantees is a classical scenario-based formulation.
    The typical procedure for exactly solving problems using the classical approach is to replicate stagewise decisions over the scenario tree.
    Since the total number of scenarios was approximately \(2\times 10^{13}\), the scenario-based formulation would contain on the order of \(10^{13}\times 10^{4}=10^{17}\) variable copies.
    This scale is far beyond what can be assembled or stored in memory.
    Future improvements in hardware do not remove the fundamental exponential growth of scenario-based formulations with respect to the number of stages. 
    In contrast, for the SDDP approach, the order of the number of variables grows only linearly in the number of stages.
}. 

\begin{figure}[t]
	\centering
	\begin{subfigure}{0.35\textwidth}
		\centering
		\includegraphics[width=1.0\textwidth]{m_network.pdf}
		\caption{Network.}
		\label{fig:m_net}  
	\end{subfigure} 
	\begin{subfigure}{0.3\textwidth}
		\centering
		\includegraphics[width=0.965\textwidth]{midtown_land.pdf}
		\caption{Land value distribution.}
		\label{fig:land}  
	\end{subfigure}  
	\\ \par \bigskip
	\begin{subfigure}{0.3\textwidth}
		\centering
		\includegraphics[width=1.0\textwidth]{midtown.pdf}
		\caption{Travel demand distribution.}
		\label{fig:midtown}  
	\end{subfigure} 
	\begin{subfigure}{0.4\textwidth}
		\centering
		\includegraphics[width=0.85\textwidth]{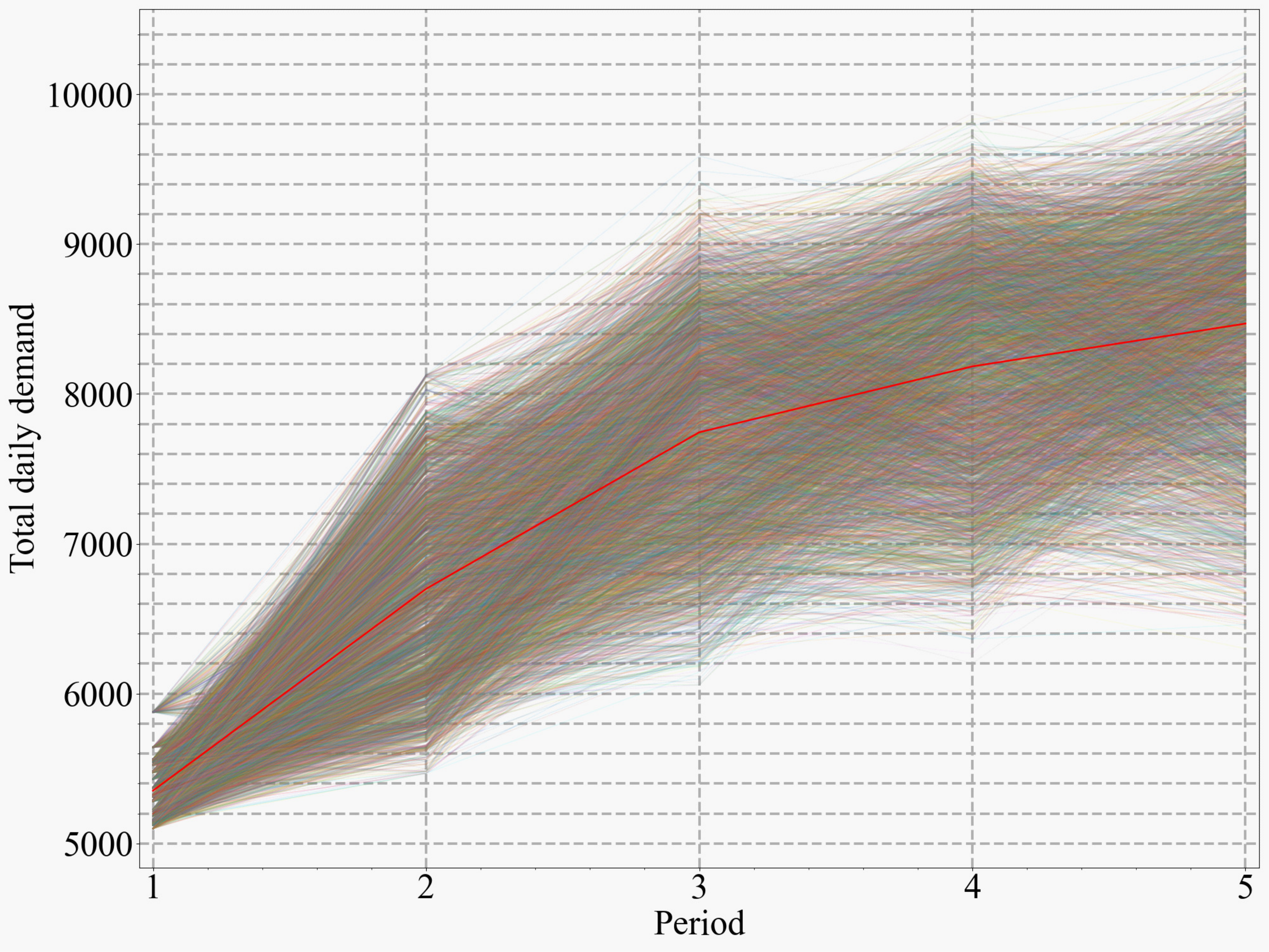}
		\caption{5,000 sample paths of total daily demand.}
		\label{fig:sample_path}  
	\end{subfigure} 
	\caption{Network and parameters in the Midtown Manhattan case.}
	\label{fig:network}  
\end{figure}

\begin{table}[t]
	\centering
	\caption{Parameter setting}
	\label{tab:parameter_settings}
	\scalebox{0.9}{
		\begin{tabular}{cccccccccccccccccccc}
			\cline{1-17}
			Parameter & $f_{s,ij}$ & $\tilde{f}_{s,ij}$ & $\epsilon_s$ & $\lambda_s$ & $d_{s,ij}$ & $\tilde{d}_{s,ij}$ & $d_{s}$ & $\tilde{d}_{s}$ & $c^{\rm BRT}_{s,ij}$ & $\hat{c}^{\rm BRT}_{s,ij}$ &  & $\rho$ & $\tilde{\rho}$ & $\tilde{\kappa}_{ij}$ & $\mu_{s,od}^{l}$             & $\phi_{s,od}^l$ &  &  &  \\ \cline{1-11} \cline{13-17}
			Value     & 30         & 30                 & 20           & 40          & 8          & 8                  & 15      & 150             & 3                    & 3                          &  & 4      & 60             & 10                    & $1.6\times\overline{Q}_{od}$ & 0.5             &  &  & \\\cline{1-17}
	\end{tabular}}
\end{table}

In the SAV case, $\xi_{s,od}^{\rm F}$ and $\xi_{s,od}^{\rm G}$ were generated from the nominal distributions given by uniform distributions $U[-\theta\overline{Q}_{od},\theta\overline{Q}_{od}]$ and $U[-\sigma \overline{Q}_{od},\sigma \overline{Q}_{od}]$, respectively.
The sample path of daily demand is shown in Figure \ref{fig:sample_path}.
In the SAV-BRT case, based on Corollary \ref{coro:BRT}, they were generated from a normal distribution and subsequently rounded to integers.
$\overline{Q}_{od}$ was given by travel demand generated from the NYC taxi data.
$\theta$ and $\sigma$ were set as the scale of fluctuations of demand growth and daily demand, respectively.
{To investigate the impact of uncertainties, the following multiple cases of $\theta$ and $\sigma$ were given: $(\theta, \sigma)\in\{{(0.2, 0.2), (0.2, 0.5), (0.2, 1.0), (0.5, 0.5), (0.7, 0.2), (0.8, 0.0)}\}$, where the base case was given by $(\theta, \sigma) = (0.2, 0.2)$.}

{
Following the propositions stated in Subsection \ref{sec: convergence}, the conditional risk measures in FR-NDPs for both the SAV and SAV-BRT systems were given by $\mathbb{ENT}_{\gamma}$.
For the SAV system, we also examined FR-NDPs employing $\mathbb{ECRM}$ with mean-$\mathbb{CV@R}_{0.75}$ as the conditional risk measures.
We set different levels of risk aversion by varying the value of $\gamma\in\{0.0, 10^{-6}, 10^{-5}, 10^{-4}\}$ and $\beta\in\{0.0, 0.4, 0.8, 1.0\}$.
Larger values of $\gamma$ and $\beta$ represent more risk-averse preferences, and both $\gamma=0$ and $\beta=0$ represent risk-neutral preferences.
While there is no clear general rule for choosing $\gamma$, there is an observation that choosing the inverse of the order of the total cost over the planning horizon could yield a policy to hedge tail risks (cf. \citealp{dowson2025incorporating}).
Based on this observation, we first computed the risk-neutral optimal policy and then set the values of $\gamma$.
The numerical results of strategic and operational decisions and the resulting system performance are reported in Subsections \ref{sec:strategic}--\ref{sec:performance} for the SAV and SAV-BRT models employing $\mathbb{ENT}_{\gamma}$ with $\gamma=10^{-5}$ as the base case.
}

{For the stopping criterion, }
we employed a stringent stopping rule to ensure that the obtained solutions are sufficiently close to the global optimum. 
{Since statistical upper bounds may fall below valid lower bounds even with a large number of Monte Carlo simulations,}
the SDDP algorithm was implemented to stop if the lower bound improves by at most 0.1 for 30 consecutive iterations.
After stopping the SDDP algorithm, we ran $5,000$ independent Monte Carlo simulations to evaluate the resulting solutions for the SAV {and SAV-BRT applications}.

The SDDP and Dual SDDP algorithms were implemented in Julia 1.9.2 using JuMP \citep{dunning2017jump} for modeling and Gurobi 11.0 as the optimizer.
For the SDDP algorithm, we employ {the version 1.13.1 of} SDDP.jl \citep{dowson2021sddp}, an open-source multi-stage stochastic programming solver.
As for the Dual SDDP algorithm, we use our own implementation developed in Julia.

%% file: ch5-1.tex
\subsection{Numerical results}

\subsubsection{Convergence results}
\label{sec:convergence}

\begin{figure}[t]
	\centering
	\begin{subfigure}{1.0\textwidth}
		\centering
		\includegraphics[width=0.85\textwidth]{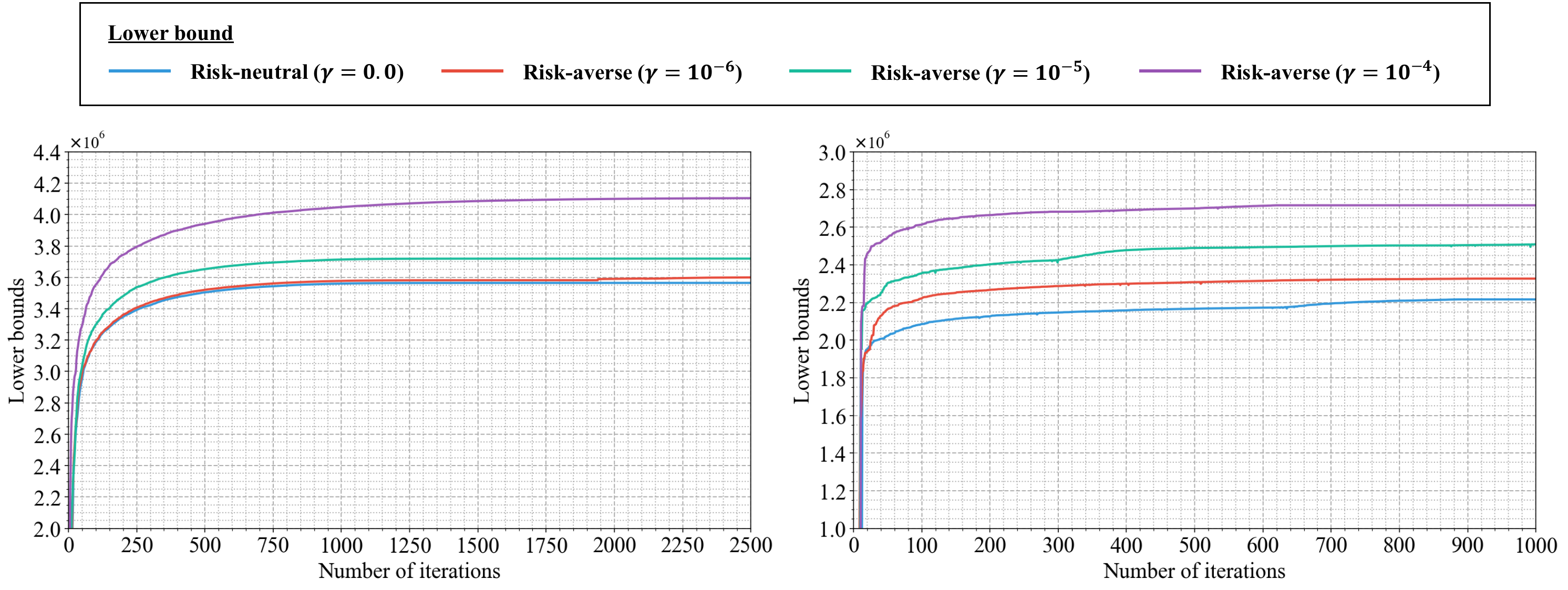}
		\caption{Convergence processes using SDDP where the risk measure is entropic (left: SAV, right: SAV-BRT).}
		\label{fig:convergence}  
	\end{subfigure}
	\\ \par \bigskip
	\begin{subfigure}{1.0\textwidth}
		\centering
		\includegraphics[width=0.85\textwidth]{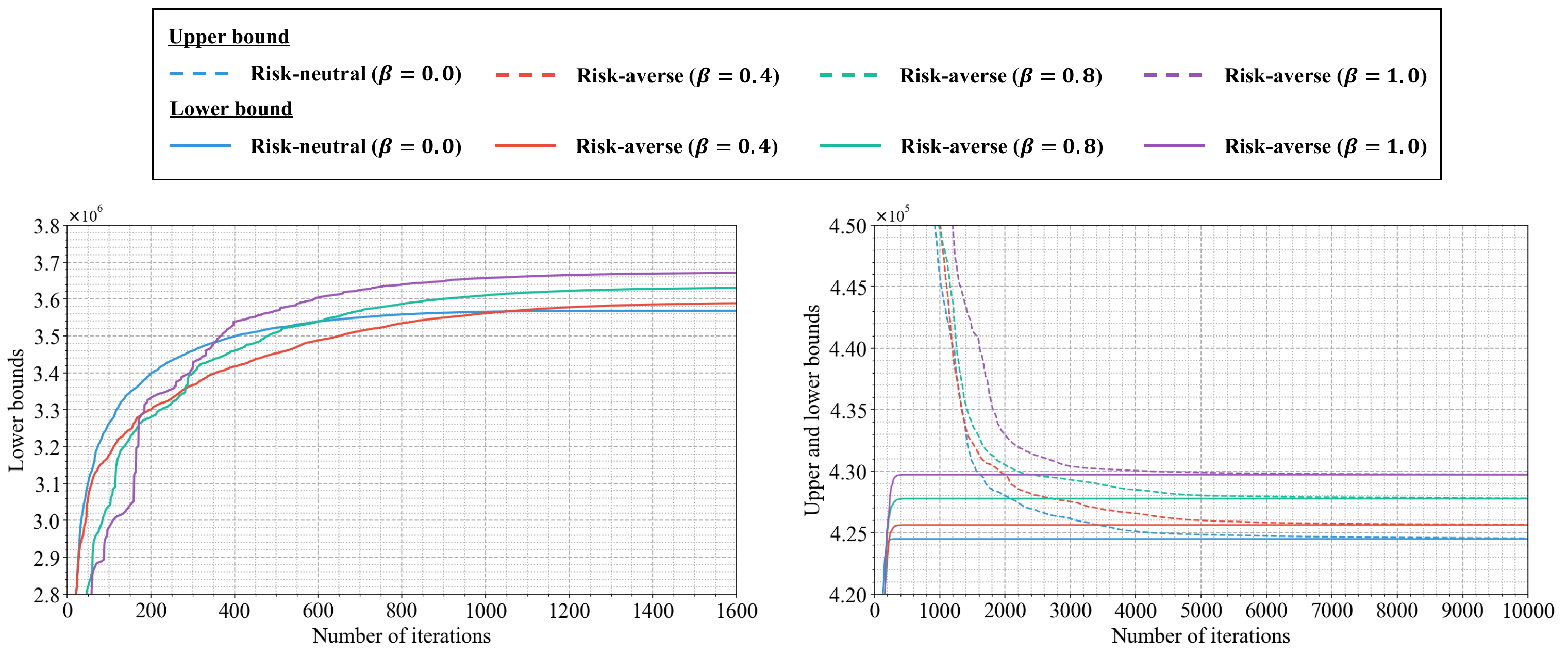}
		\caption{Convergence processes for the SAV systems where the risk measure is $\mathbb{ECRM}$ with mean-$\mathbb{CV@R}$ (left: SDDP, right: Dual SDDP).}
		\label{fig:convergence_ECRM}  
	\end{subfigure}
	\caption{Convergence processes of the deterministic lower and upper bounds.}
	\label{fig:convergence0}  
\end{figure}

\begin{table}[t]
	{
	\centering
	\caption{Performance of SDDP algorithm.}
	\label{tab:gap}
        \begin{subtable}[t]{1.0\textwidth}
        \centering
            \caption{SAV system}
            \label{tab:gap-sav}
            \centering
    		\begin{tabular}{clccclcc}
                \hline
    			&  & & \multicolumn{2}{c}{$M=100$}                           &  & \multicolumn{2}{c}{$M=5000$}                                          \\ \cline{4-5} \cline{7-8} 
    			$\mathbb{F}$                    &  & LB                  & UB                  & Gap          &  & UB                     & Gap                     \\ \cline{1-1} \cline{3-5} \cline{7-8} 
    			$\mathbb{E}$                    & & $3.564\times10^{6}$  & $3.571\times10^{6}$  & $~~~0.166\%$ &  & $3.568\times10^{6}$  & $0.082\%$            \\
    			$\mathbb{ENT}~(\gamma=10^{-6})$ & & $3.599\times10^{6}$  & $3.584\times10^{6}$  & $-0.366\%$ &  & $3.619\times10^{6}$  & $0.474\%$             \\
    			$\mathbb{ENT}~(\gamma=10^{-5})$ & & $3.719\times10^{6}$  & $3.697\times10^{6}$  & $-0.542\%$ &  & $3.728\times10^{6}$  & $0.210\%$             \\
    			$\mathbb{ENT}~(\gamma=10^{-4})$ & & $4.105\times10^{6}$  & $3.852\times10^{6}$  & $-6.147\%$ &  & $4.107\times10^{6}$  & $0.048\%$              \\
    			$\mathbb{ECRM}~(\beta=0.4)$     & & $3.590\times10^{6}$  & $3.624\times10^{6}$  & $~~~0.930\%$ &  & $3.610\times10^{6}$  & $0.538\%$ \\
    			$\mathbb{ECRM}~(\beta=0.8)$     & & $3.631\times10^{6}$  & $3.673\times10^{6}$  & $~~~1.145\%$ &  & $3.647\times10^{6}$  & $0.426\%$ \\
    			$\mathbb{ECRM}~(\beta=1.0)$     & & $3.671\times10^{6}$  & $3.715\times10^{6}$  & $~~~1.202\%$ &  & $3.703\times10^{6}$  & $0.880\%$ \\ \hline
    	   \end{tabular}
        \end{subtable} \\
        \vspace{2ex}
        \begin{subtable}[t]{1.0\textwidth}
        \centering
            \caption{SAV-BRT system}
            \label{tab:gap-sav-brt}
            \centering
    		\begin{tabular}{clccclcc}
                \hline
    			&  & & \multicolumn{2}{c}{$M=100$}                           &  & \multicolumn{2}{c}{$M=5000$}                                          \\ \cline{4-5} \cline{7-8} 
    			$\mathbb{F}$                    &  & LB                  & UB                  & Gap          &  & UB                     & Gap                     \\ \cline{1-1} \cline{3-5} \cline{7-8} 
    			$\mathbb{E}$                    & & $2.216\times10^{6}$  & $2.246\times10^{6}$  & $~~~1.349\%$ &  & {$2.252\times10^{6}$}     & {$1.622\%$}            \\
    			$\mathbb{ENT}~(\gamma=10^{-6})$ & & $2.327\times10^{6}$  & $2.321\times10^{6}$  & $-0.254\%$ &  & {$2.351\times10^{6}$}     & {$1.066\%$}              \\
    			$\mathbb{ENT}~(\gamma=10^{-5})$ & & $2.509\times10^{6}$  & $2.450\times10^{6}$  & $-2.338\%$ &  & {$2.529\times10^{6}$}     & {$0.800\%$}              \\
    			$\mathbb{ENT}~(\gamma=10^{-4})$ & & $2.716\times10^{6}$  & $2.634\times10^{6}$  & $-3.028\%$ &  & {$2.720\times10^{6}$}     & {$0.134\%$}              \\ \hline
    	   \end{tabular}
        \end{subtable}
    }
\end{table}

Figures \ref{fig:convergence} and \ref{fig:convergence_ECRM} illustrate the convergence process with the SDDP and Dual SDDP algorithms.
{
	Figure \ref{fig:convergence} plots the deterministic lower bounds obtained by SDDP for the FR-NDP with entropic risk measures (left: SAV; right: SAV-BRT).
}
Figure \ref{fig:convergence_ECRM} reports the results for the FR-NDP with ECRMs.
{
The left panel of Figure \ref{fig:convergence_ECRM} shows the lower bounds obtained by SDDP, 
whereas the right panel of Figure \ref{fig:convergence_ECRM} shows the deterministic lower and upper bounds obtained by Dual SDDP.
}
Because Dual SDDP is computationally intensive, it was tested on a toy network with 10 links and 5 nodes.
{
The horizontal axis indicates the number of iterations, and the vertical axis shows the upper and lower bounds at each iteration.
Solid lines indicate lower bounds, and dashed lines indicate upper bounds (when shown).
The colors of these lines correspond to different risk-aversion parameter settings.
}

From Figures \ref{fig:convergence} and \ref{fig:convergence_ECRM}, we observe that the lower bounds generated by SDDP monotonically increase,
{while the deterministic upper bounds generated by Dual SDDP monotonically decrease, as the number of iterations increases.
The right panel of Figure \ref{fig:convergence_ECRM} also illustrates that the deterministic lower and upper bounds converge to the same value.
These behaviors are consistent with Propositions \ref{prop:suff} and \ref{prop:suff3}.}
The number of iterations required for the deterministic upper bound to stabilize is substantially higher under Dual SDDP than under the standard counterpart, consistent with previous findings (cf. \citealp{guigues2023duality} and \citealp{da2023dual}).
{
Finally, the converged bounds increase with increasing risk-aversion parameters $\gamma$ and $\beta$, reflecting the more conservative objective induced by stronger risk aversion.}

{
Table \ref{tab:gap} reports the upper and lower bounds for the optimal value obtained by the SDDP algorithm.
{Tables \ref{tab:gap-sav} and \ref{tab:gap-sav-brt} show the results for the SAV and SAV-BRT systems, respectively.}
The first column lists the type of risk measure $\mathbb{F}$ used in the objective.
UB and LB denote the 95\% confidence statistical upper bound and deterministic lower bound, respectively, when the stopping criterion is satisfied.
The statistical upper bound is reported as the 95\% confidence upper bound, constructed using samples of total costs obtained from independent Monte Carlo simulations.
Gap is the relative gap between the bounds, ${\rm (UB - LB)/LB\times100}$.
{Columns 3–4 and Columns 5–6 correspond to the results with the number of Monte Carlo replications $M=100$ and $M=5,000$, respectively.}
{Distributional information on total costs obtained from Monte Carlo simulations of the SAV and SAV-BRT systems is provided in Appendix \ref{app:TotalCost}.}
Note that Table \ref{tab:gap} is not intended to compare the performance of the two systems, as they have different objective functions.
}

{
From Table~\ref{tab:gap}, we first focus on the results obtained with $M=5{,}000$.
For the SAV and SAV-BRT systems, the relative gap between the statistical UB and the deterministic LB remains below 1.7\% in all experiments with different choices of the risk-aversion parameters $\gamma$ and $\beta$.
These results indicate that, when the policies are evaluated with a sufficiently large number of Monte Carlo replications,
the lower and statistical upper bounds are close at termination for both applications.
}

{Comparing $M=100$ with $M=5{,}000$ further shows that the reliability of the statistical UB depends on both the risk measure and the number of Monte Carlo simulations. 
For the case with $M=100$, the UB falls below the deterministic LB in the cases with entropic risk measures for both systems, yielding negative apparent gaps. 
Moreover, the inversion deepens as $\gamma$ increases, reaching $-6.147\%$ for the SAV system and $-3.028\%$ for the SAV-BRT system at $\gamma=10^{-4}$. 
These observations are consistent with the asymptotic nature of Proposition~\ref{prop:suff2}: the statistical upper bound is guaranteed to be reliable only in the large samples, particularly in the case of strong risk aversion.}

\begin{figure}[p]
	\centering
	\begin{subfigure}{0.95\textwidth}
			\centering
			\includegraphics[width=1.0\textwidth]{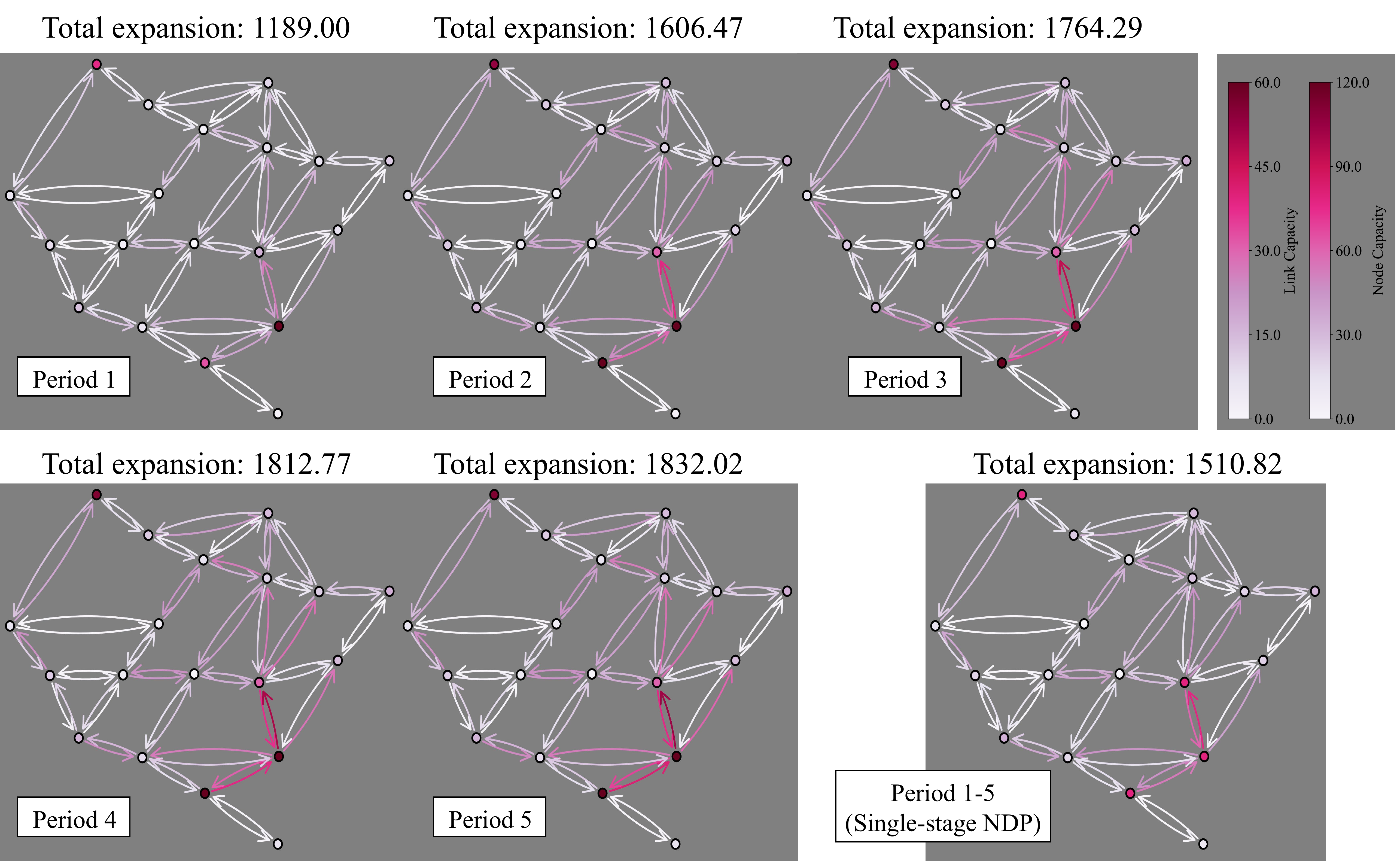}
			\caption{Risk-neutral case ($\gamma = 0.0$)}
			\label{fig:SAVinfra_neutral}  
		\end{subfigure}
	\\ \par \bigskip
	\begin{subfigure}{0.95\textwidth}
			\centering
			\includegraphics[width=1.0\textwidth]{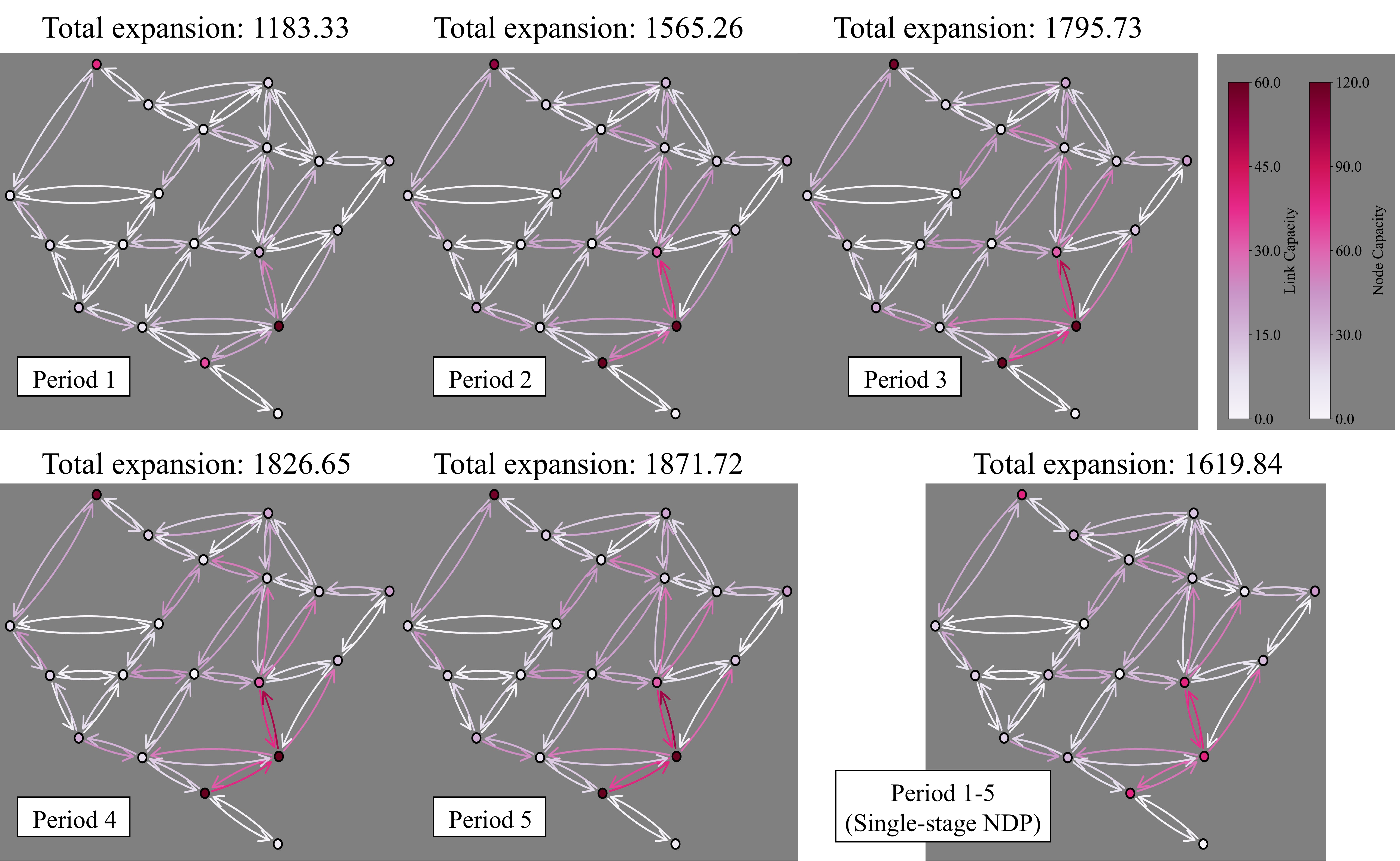}
			\caption{Risk-averse case ($\gamma = 10^{-5}$)}
			\label{fig:SAVinfra_averse}  
		\end{subfigure}
	\caption{Capacity expansion pattern and total expansion in the SAV system.}
	\label{fig:SAVinfra}  
\end{figure}

%% file: ch5-2.tex
\subsubsection{Strategic decisions}
\label{sec:strategic}

Figure \ref{fig:SAVinfra} describes the strategic decisions obtained from the FR-NDP for the SAV system.
Figures \ref{fig:SAVinfra_neutral} and \ref{fig:SAVinfra_averse} correspond to the risk-neutral and risk-averse cases, respectively.
In each figure, the colors of the links and nodes represent the optimal state variables of capacities in the scenario with the largest daily demand among the 5,000 scenarios; darker shades represent larger capacities.
The bottom-right panel shows the single-stage network design, whereas the remaining panels present the multi-stage counterparts.
The numbers above each panel are the sum of the optimal capacity state variables, indicating the total network capacity in the corresponding period.

\begin{figure}[t]
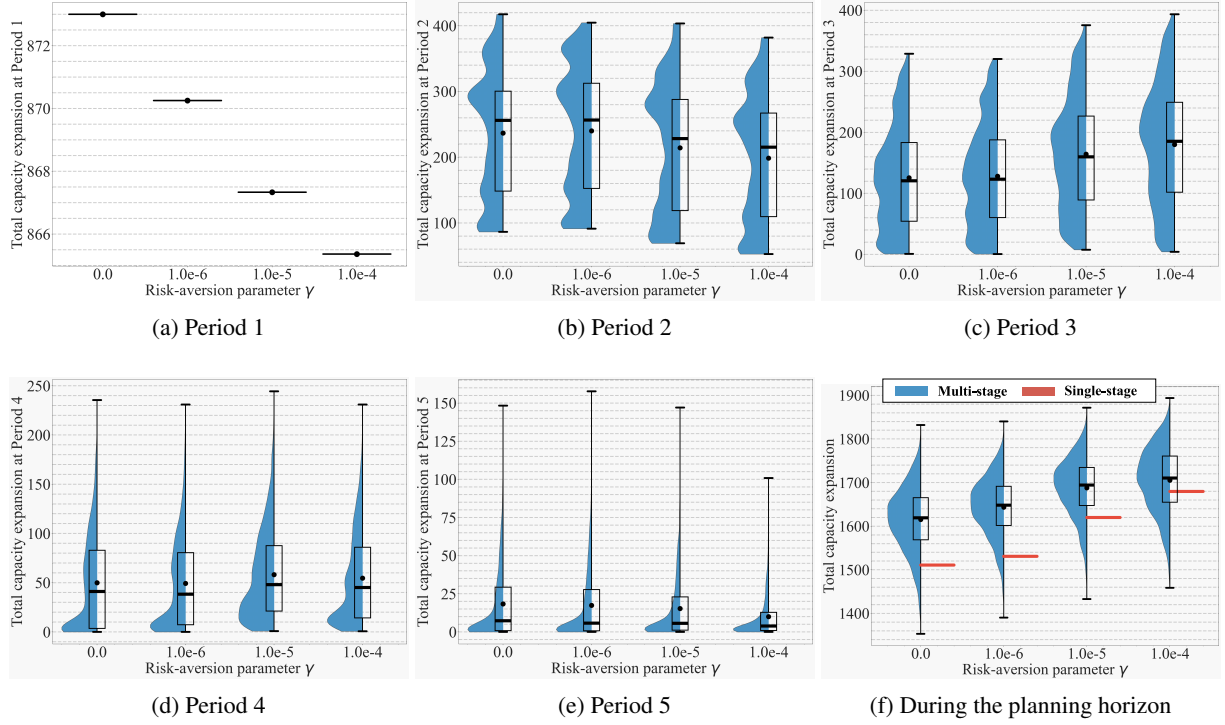

	\centering
	\begin{subfigure}{0.32\textwidth}
		\centering
		\includegraphics[width=1.0\textwidth]{TotalCapacityExpansionPeriod1.pdf}
		\caption{Period 1}
		\label{fig:TotalCapacityExpansion1}  
	\end{subfigure}
	\begin{subfigure}{0.32\textwidth}
		\centering
		\includegraphics[width=1.0\textwidth]{TotalCapacityExpansionPeriod2.pdf}
		\caption{Period 2}
		\label{fig:TotalCapacityExpansion2}  
	\end{subfigure}
	\begin{subfigure}{0.32\textwidth}
		\centering
		\includegraphics[width=1.0\textwidth]{TotalCapacityExpansionPeriod3.pdf}
		\caption{Period 3}
		\label{fig:TotalCapacityExpansion3}  
	\end{subfigure} \\ \par \bigskip
	\begin{subfigure}{0.32\textwidth}
		\centering
		\includegraphics[width=1.0\textwidth]{TotalCapacityExpansionPeriod4.pdf}
		\caption{Period 4}
		\label{fig:TotalCapacityExpansion4}  
	\end{subfigure} 
	\begin{subfigure}{0.32\textwidth}
		\centering
		\includegraphics[width=1.0\textwidth]{TotalCapacityExpansionPeriod5.pdf}
		\caption{Period 5}
		\label{fig:TotalCapacityExpansion5}  
	\end{subfigure}
	\begin{subfigure}{0.32\textwidth}
		\centering
		\includegraphics[width=1.0\textwidth]{TotalCapacityExpansion.pdf}
		\caption{During the planning horizon}
		\label{fig:TotalCapacityExpansion0}  
	\end{subfigure}
	\caption{Sensitivity analysis of total capacity expansion with respect to $\gamma$ in the SAV system.}
	\label{fig:TotalCapacityExpansion}  
\end{figure}

\begin{figure}[p]
	\centering
	\begin{subfigure}{0.95\textwidth}
		\centering
		\includegraphics[width=1.0\textwidth]{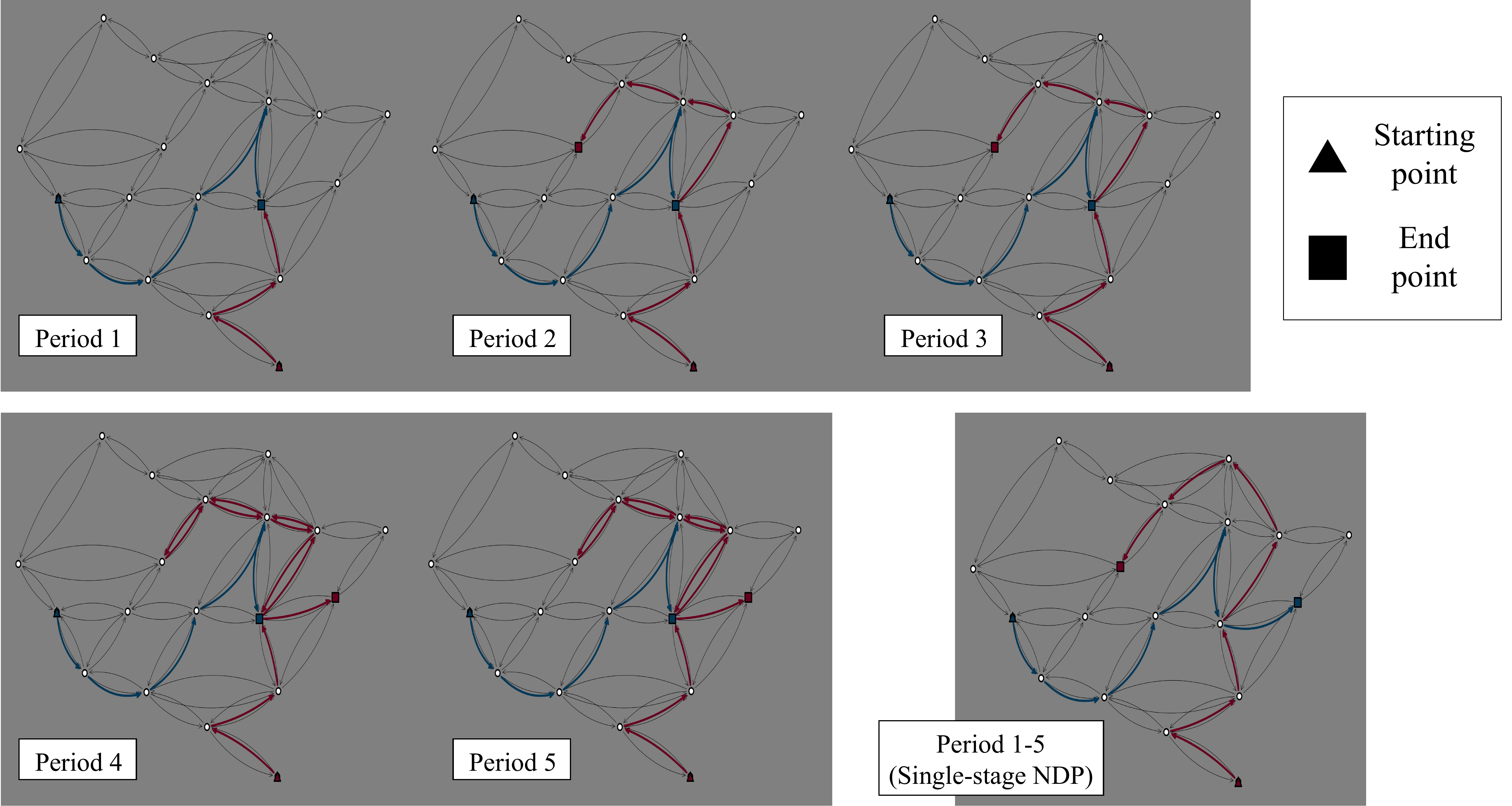}
		\caption{Risk-neutral case ($\gamma = 0.0$)}
		\label{fig:BRTinfra_neutral}  
	\end{subfigure}
	\\ \par \bigskip
	\begin{subfigure}{0.95\textwidth}
		\centering
		\includegraphics[width=1.0\textwidth]{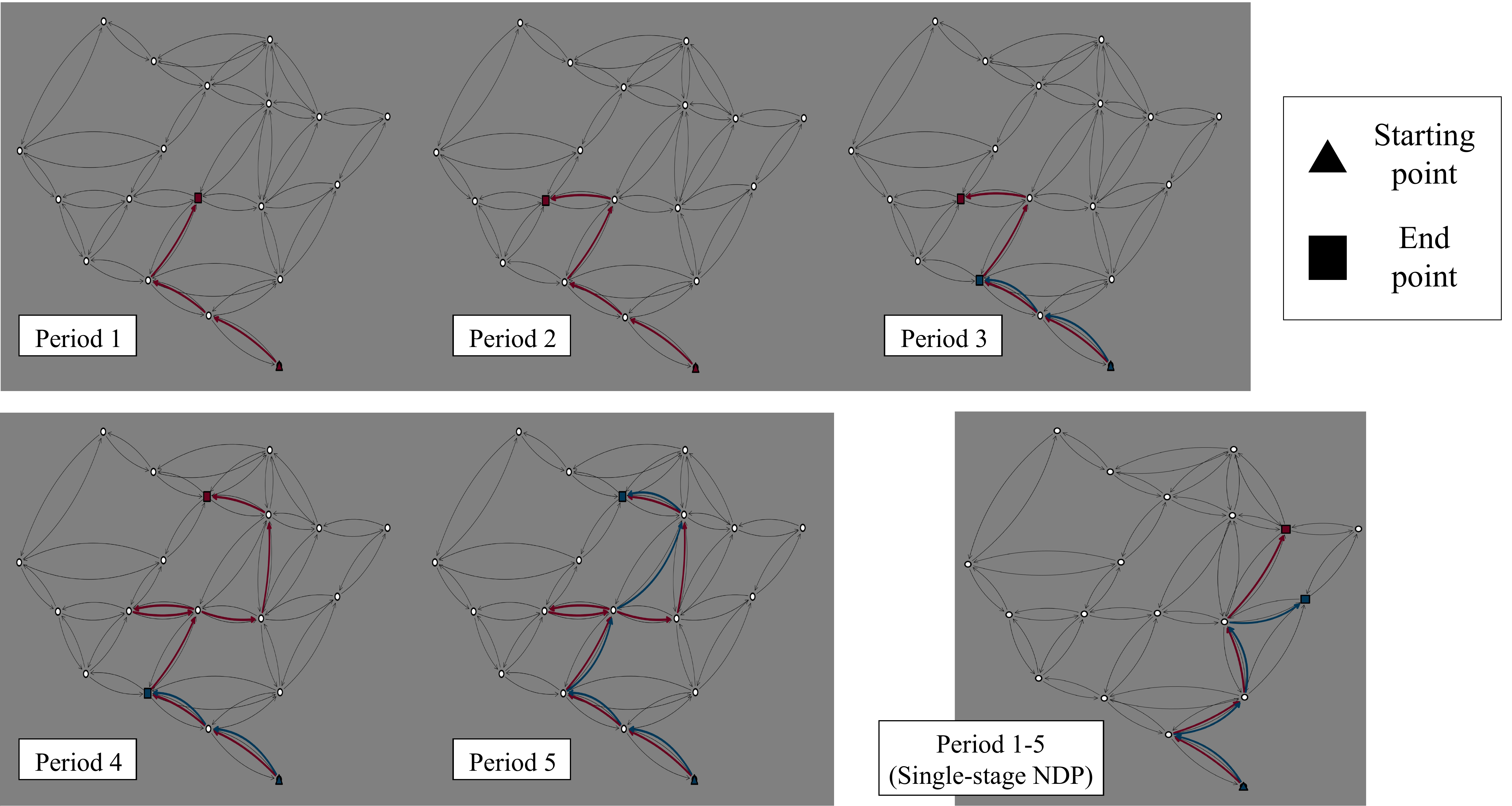}
		\caption{Risk-averse case ($\gamma = 10^{-5}$)}
		\label{fig:BRTinfra_averse}  
	\end{subfigure}
	\caption{BRT route deployment pattern.}
	\label{fig:BRTinfra}  
\end{figure}

Figure \ref{fig:SAVinfra} demonstrates that introducing multi-stage strategic decisions increases the end-of-horizon network capacity regardless of risk preference.
This increase can be attributed to the ability of multi-stage investments to reduce intermediate infrastructure maintenance costs through more flexible investment timing.
Comparing the single-stage case with its multi-stage counterpart reveals distinct spatial and temporal investment patterns.
In the single-stage case, investment concentrates at nodes 107, 137, and 170 in the bottom right area and at node 50 in the top left.
In the multi-stage case, node 137 receives substantial investment from period 1, whereas the other nodes expand gradually over subsequent periods.
Furthermore, comparing Figures \ref{fig:SAVinfra_neutral} and \ref{fig:SAVinfra_averse} reveals that, in both the single- and multi-stage cases, infrastructure investments are generally larger under the risk-averse setting.
This reflects a risk-averse strategy that aims to hedge against high operational costs in severe scenarios encountered during the subsequent operational stages.
Notably, the increase in infrastructure investment due to risk-averse preferences is more pronounced in the single-stage model than in its multi-stage counterpart.
This observation suggests that single-stage network designs bear greater risk, and that incorporating flexibility into strategic decisions can alleviate such risks.
This suggestion is further supported by examining the investment sequence under risk-neutral and risk-averse settings.
From a comparison between Figures \ref{fig:SAVinfra_neutral} and \ref{fig:SAVinfra_averse}, infrastructure investments in periods 1 and 2 are slightly higher under the risk-neutral case, which contrasts with the investment pattern observed in single-stage designs.
This result implies that the risk-averse planner avoids the risk of over-supplying infrastructure early in the planning horizon, while preserving room to adjust later.

{
Figure \ref{fig:TotalCapacityExpansion} presents a sensitivity analysis of total capacity expansion with respect to the risk-aversion parameter $\gamma$.
Figures \ref{fig:TotalCapacityExpansion1}--\ref{fig:TotalCapacityExpansion5} show the sum of traffic and storage capacity added at each stage,
while Figure \ref{fig:TotalCapacityExpansion0} shows the cumulative capacity level at the terminal stage.
The decisions of capacity expansion, dependent on uncertain demand, are visualized in Figures \ref{fig:TotalCapacityExpansion2}--\ref{fig:TotalCapacityExpansion0} in the form of probability distributions derived from the kernel density estimators, with overlaid box plots.
In Figure \ref{fig:TotalCapacityExpansion0}, the blue areas and red lines correspond to the multi- and single-stage network designs, respectively.}
The vertical axis denotes the optimal capacity variables,
{and the horizontal axis denotes the risk-aversion parameter $\gamma$.}

{
Figure \ref{fig:TotalCapacityExpansion0} shows that the total infrastructure investment increases as risk-averse preferences become stronger and as multi-stage strategic decisions are introduced, similar to the observations in Figure \ref{fig:SAVinfra}. 
Also,}
as shown in Figure \ref{fig:TotalCapacityExpansion1}--\ref{fig:TotalCapacityExpansion5}, the overall distributions reveal that under the risk-averse setting, infrastructure investments in periods 1 and 2 are suppressed to allow for greater flexibility in later periods.
This pattern highlights the planner's risk-hedging strategy to preserve adjustment capacity under uncertainty.
Interestingly, the optimal investment strategy does not always exhibit a simple increasing trend over time, even as risk is gradually resolved in later periods.
This non-monotonic behavior underscores the necessity of the proposed framework for modeling long-term infrastructure planning with flexibility and reliability.

\begin{figure}[t]
	\centering
	\includegraphics[width=0.95\textwidth]{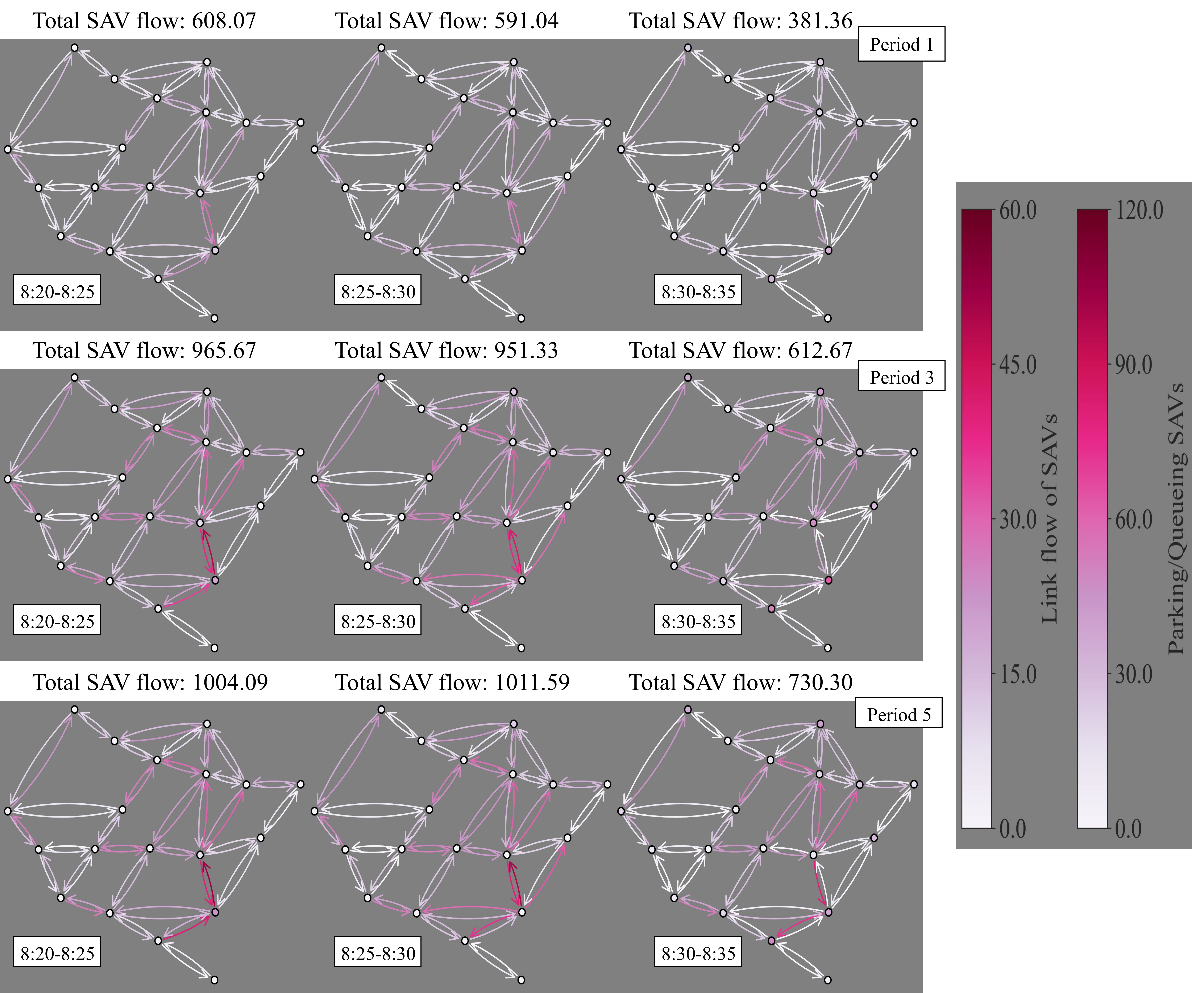}
	\caption{Dynamic flow of SAVs in the scenario with the largest travel demand under the multi-stage network design {with $\gamma=10^{-5}$} (left: 8:20-8:25, center: 8:25-8:30, right: 8:30-8:35).}
	\label{fig:SAVflow_multi}  
\end{figure}

Figure \ref{fig:BRTinfra} illustrates the optimal BRT routes obtained from the FR-NDP in the SAV-BRT system.
Each route begins at the triangle node and ends at the square node.
As in Figure \ref{fig:SAVinfra}, two colored links depict the optimal BRT routes for the scenario with the largest daily demand.
Figures \ref{fig:BRTinfra_neutral} and \ref{fig:BRTinfra_averse} correspond to the risk-neutral and risk-averse cases, respectively.
The bottom-right panel of each subfigure corresponds to the single-stage network design, while the remaining panels represent its multi-stage counterparts.
Figure \ref{fig:BRTinfra} shows that, as with capacity expansion for SAV systems, multi-stage planning yields longer BRT routes at the end of the horizon, regardless of risk preference.
This difference between single- and multi-stage designs is more pronounced for the BRT route planning, which may be attributed to the binary nature of the route decision variables.
Meanwhile, comparing Figures \ref{fig:BRTinfra_neutral} and \ref{fig:BRTinfra_averse} demonstrates that risk aversion leads to shorter BRT routes, unlike the capacity expansion shown in Figure \ref{fig:SAVinfra}.
{Furthermore, from the comparison, we observe links where two routes overlap in the risk-averse scenario.
These results reflect} the relatively limited spatial flexibility of BRT operations.
Once a fixed route is built, it offers limited ability to hedge against operational risks; thus, a risk-averse planner tends to invest less in BRT, {focusing instead on core route segments with higher demand certainty}, while relying more {heavily and extensively} on SAV capacity.
This observation will be further supported by the analysis of operational decisions presented in the following subsection.

\begin{figure}[t]
	\centering
	\includegraphics[width=0.95\textwidth]{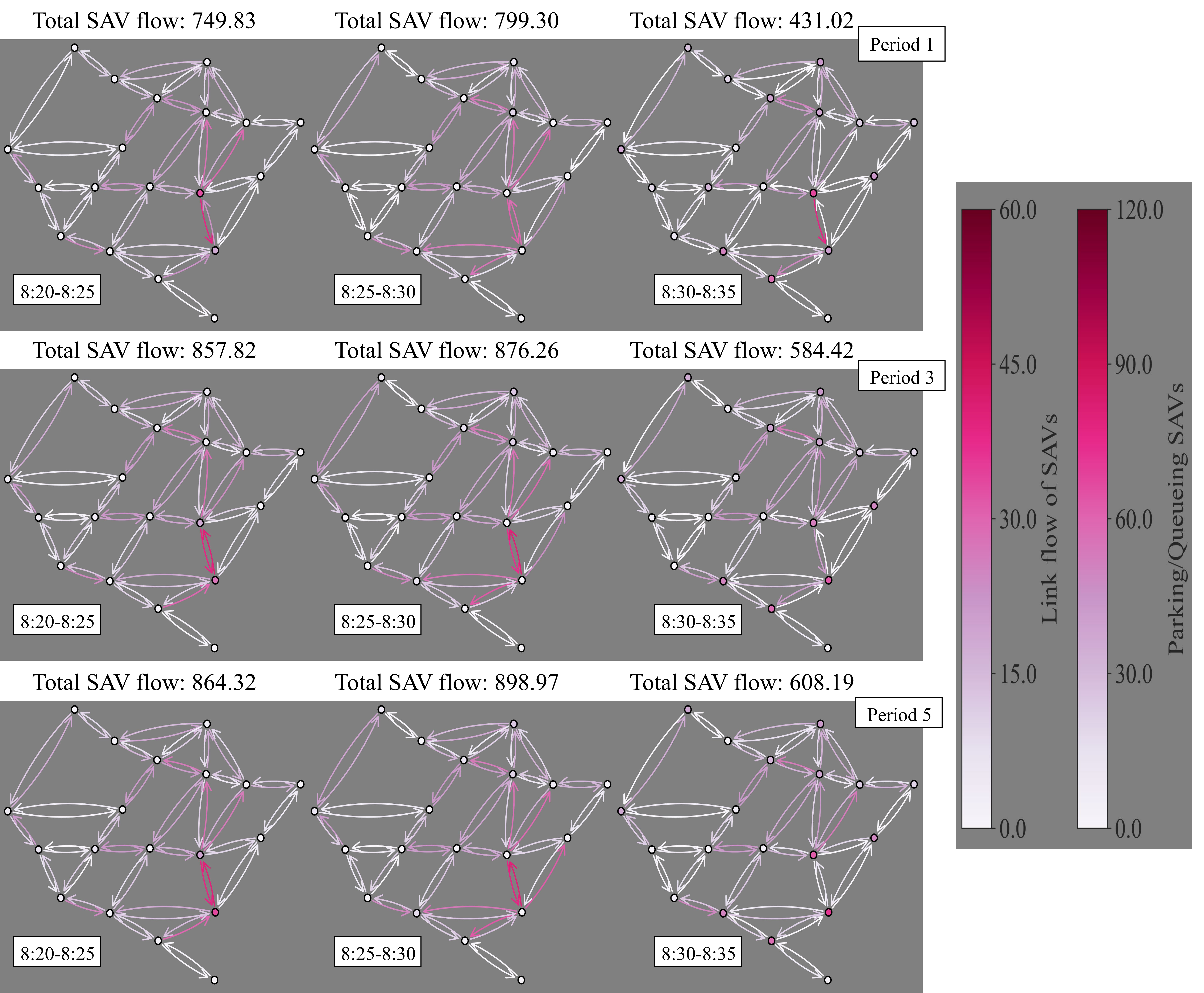}
	\caption{Dynamic flow of SAVs in the scenario with the largest travel demand under the single-stage network design {with $\gamma=10^{-5}$} (left: 8:20-8:25, center: 8:25-8:30, right: 8:30-8:35).}
	\label{fig:SAVflow_single}  
\end{figure}

\subsubsection{Operational decisions}

Figures \ref{fig:SAVflow_multi} and \ref{fig:SAVflow_single} present the operational decisions obtained from the FR-NDP for the SAV system.
Figures \ref{fig:SAVflow_multi} and \ref{fig:SAVflow_single} show the optimal SAV flows around the desired arrival time of travelers (8:30) under the multi- and single-stage network design with $\gamma=10^{-5}$.
Each subfigure presents the operational decisions at three different time steps: the top corresponds to period 1, the middle to period 3, and the bottom to period 5.
The colors of the links and nodes depict, respectively, link traffic volume and the number of parking or queueing vehicles at the node, under the same scenario used in {Figure} \ref{fig:SAVinfra}.
The number above each subfigure is the total SAV flow in the corresponding {time step}.

From Figures \ref{fig:SAVflow_multi} and \ref{fig:SAVflow_single} we observe that SAV link flows decrease after the desired travelers' arrival time (8:30), demonstrating reasonable responsiveness of the operational decisions to temporal demand patterns.
Comparing Figures \ref{fig:SAVflow_multi} and \ref{fig:SAVflow_single}, flows in both cases rise with demand growth over the planning horizon; however, the change is muted in the single-stage design and much more pronounced in the multi-stage design.
This implies that in the single-stage setting, SAV operations are constrained by the static network capacities determined at the beginning of the planning horizon, which limits the ability of vehicle operations to accommodate demand fluctuations and provide smooth travel services to travelers.
In contrast, the multi-stage setting allows for capacity expansions over time, as illustrated in Figure \ref{fig:TotalCapacityExpansion}, enabling the operations manager to provide more responsive transportation services with increasing demand.

\begin{figure}[t]
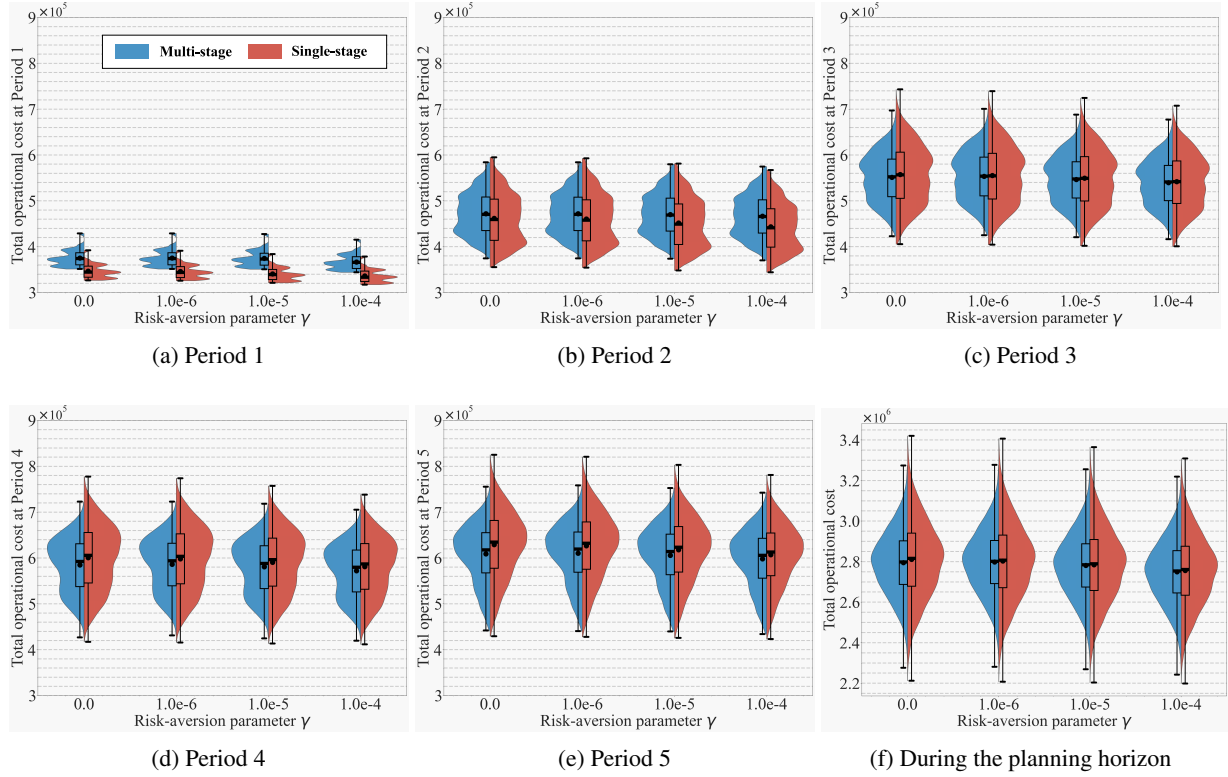

	\centering
	\begin{subfigure}{0.32\textwidth}
		\centering
		\includegraphics[width=1.0\textwidth]{OperationalCostPeriod1.pdf}
		\caption{Period 1}
		\label{fig:OperationalCostPeriod1}  
	\end{subfigure}
	\begin{subfigure}{0.32\textwidth}
		\centering
		\includegraphics[width=1.0\textwidth]{OperationalCostPeriod2.pdf}
		\caption{Period 2}
		\label{fig:OperationalCostPeriod2}  
	\end{subfigure} 
	\begin{subfigure}{0.32\textwidth}
		\centering
		\includegraphics[width=1.0\textwidth]{OperationalCostPeriod3.pdf}
		\caption{Period 3}
		\label{fig:OperationalCostPeriod3}  
	\end{subfigure} \\ \par \bigskip
	\begin{subfigure}{0.32\textwidth}
		\centering
		\includegraphics[width=1.0\textwidth]{OperationalCostPeriod4.pdf}
		\caption{Period 4}
		\label{fig:OperationalCostPeriod4}  
	\end{subfigure} 
	\begin{subfigure}{0.32\textwidth}
		\centering
		\includegraphics[width=1.0\textwidth]{OperationalCostPeriod5.pdf}
		\caption{Period 5}
		\label{fig:OperationalCostPeriod5}  
	\end{subfigure}
	\begin{subfigure}{0.32\textwidth}
		\centering
		\includegraphics[width=1.02\textwidth]{OperationalCost.pdf}
		\caption{During the planning horizon}
		\label{fig:OperationalCost0}  
	\end{subfigure} 
	\caption{Sensitivity analysis of operational cost at each period with respect to $\gamma$ in the SAV system.}
	\label{fig:OperationalCost}  
\end{figure}

\begin{figure}[t]
	\centering
	\begin{subfigure}{0.95\textwidth}
		\centering
		\includegraphics[width=0.95\textwidth]{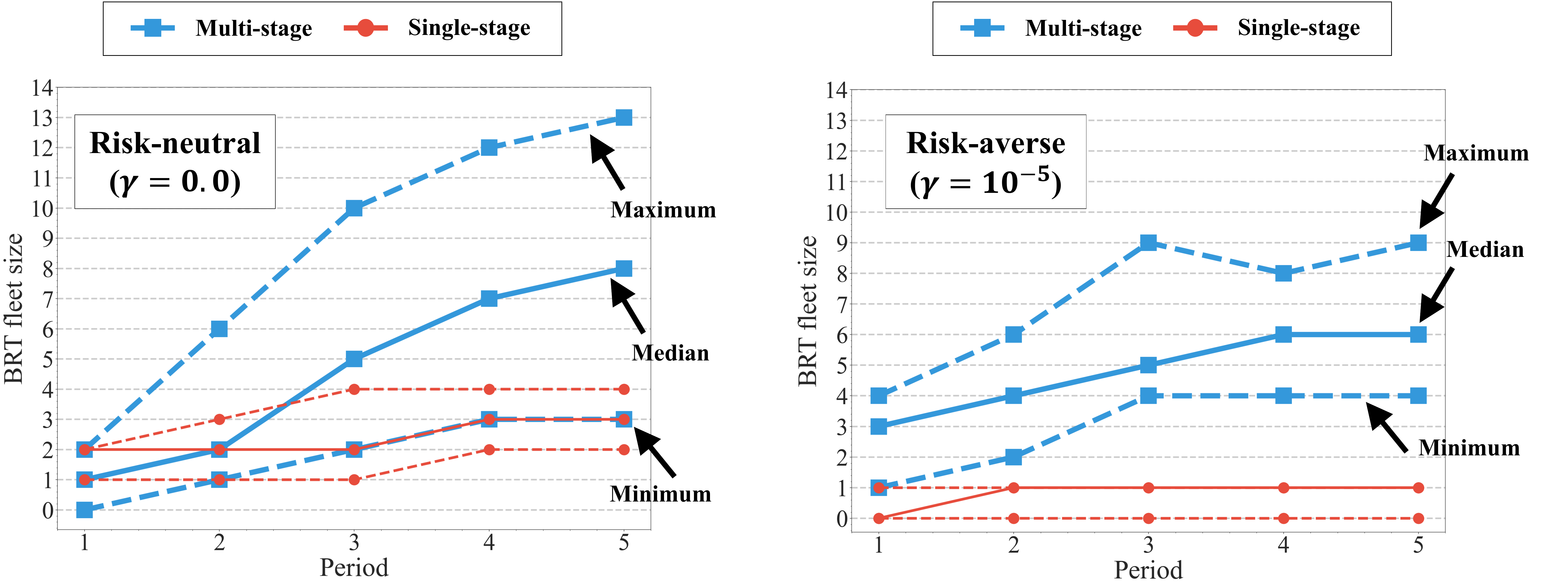}
		\caption{BRT fleet size (left: risk-neutral ($\gamma=0.0$), right: risk-averse ($\gamma=10^{-5}$))}
		\label{fig:BRT_BRT}  
	\end{subfigure} \\ \par \bigskip
	\begin{subfigure}{0.95\textwidth}
		\centering
		\includegraphics[width=0.95\textwidth]{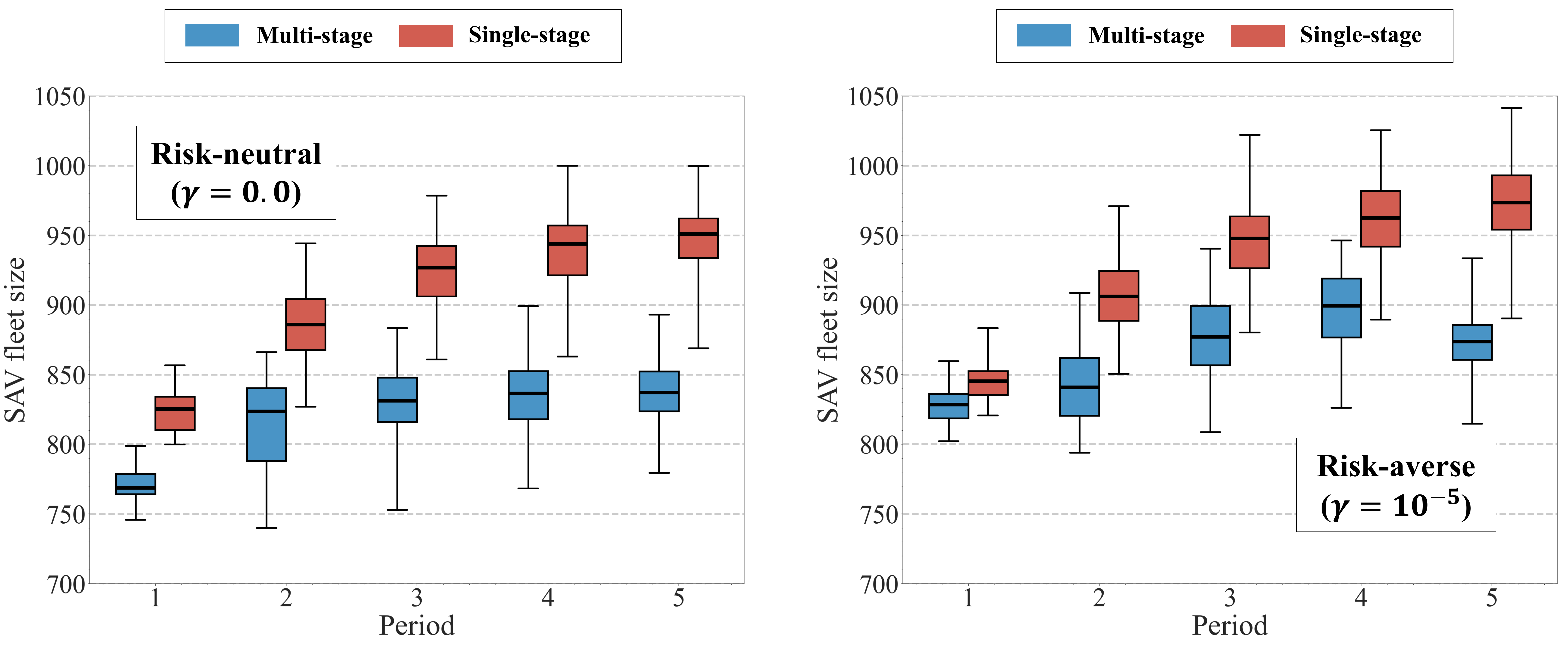}
		\caption{SAV fleet size (left: risk-neutral ($\gamma=0.0$), right: risk-averse ($\gamma=10^{-5}$))}
		\label{fig:SAV_BRT}  
	\end{subfigure}
	\caption{Fleet size at each period in the SAV-BRT system.}
	\label{fig:fleet_size}  
\end{figure}

Figure \ref{fig:OperationalCost} presents {a sensitivity analysis of} the probability distribution of operational costs {with respect to the risk-aversion parameter $\gamma$.
Figures \ref{fig:OperationalCostPeriod1}--\ref{fig:OperationalCostPeriod5} correspond to the stagewise operational costs, while Figure \ref{fig:OperationalCost0} corresponds to the total operational cost.
}
Within each panel, blue and red areas represent operational costs under the multi- and single-stage network designs, respectively.

Figures \ref{fig:OperationalCostPeriod1}--\ref{fig:OperationalCostPeriod5} show that operational costs tend to rise monotonically over the planning horizon, simply because demand grows over time.
Notably, under the single-stage network design, the tail of the distribution of operational costs becomes heavier, indicating the planner encounters excessive operational costs in severe scenarios.
We also observe that introducing risk aversion truncates the tail even in the single-stage setting, yet they remain significantly heavier than those under the multi-stage design.
This finding suggests that NDPs with reliability cannot fully mitigate the risk arising from uncertain demand growth; strategic flexibility in the form of the multi-stage planning framework is essential.

Finally, Figure \ref{fig:fleet_size} illustrates the operational decisions in the SAV-BRT system. 
Figures \ref{fig:BRT_BRT} and \ref{fig:SAV_BRT} show the BRT and SAV fleet sizes at each period, respectively.
{
The left and right panels illustrate respective risk-neutral ($\gamma = 0.0$) and risk-averse ($\gamma = 10^{-5}$) cases.
Within each panel, the blue areas and lines correspond to the multi-stage design, while the red areas and lines correspond to the single-stage design.
}

As shown in Figure \ref{fig:fleet_size}, the trends of the fleet size decisions under the single- and multi-stage designs are in contrast.
Under the single-stage design, the BRT network persists over time, limiting BRT's ability to accommodate operational uncertainties.
In contrast, the multi-stage strategic decisions allow the BRT network to be expanded in response to uncertain demand growth, enhancing its spatial flexibility over time.
As a result, the BRT fleet size increases in later periods, reducing dependence on SAVs.
{
In addition, comparing the left and right panels of Figure \ref{fig:fleet_size} reveals that the introduction of risk aversion shifts fleet sizing away from BRT and toward SAV.
Because BRT service is tied to fixed corridors and is less spatially adjustable than SAV operations,
a risk-averse planner tends to limit BRT deployment and instead relies more on the operational flexibility of SAVs to accommodate scenarios with high demand.
}
This observation highlights the interrelation between strategic and operational decisions from the perspective of 
{the complementary relations between flexibility and reliability
} in accommodating uncertainty, which argues for the necessity of the proposed framework for achieving both flexible and reliable network design planning.

%% file: ch5-3.tex
\subsubsection{System performance}
\label{sec:performance}

\begin{figure}[t]
	\centering
	\begin{subfigure}{1.0\textwidth}
		\centering
		\includegraphics[width=1.0\textwidth]{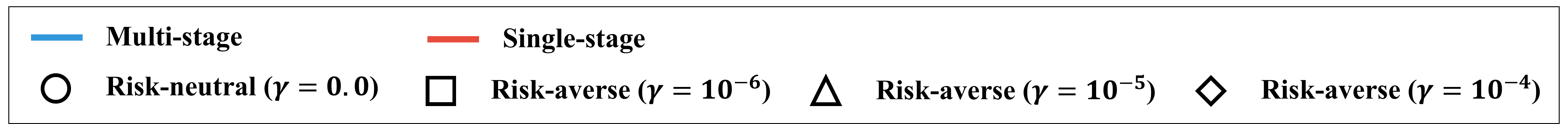}
	\end{subfigure} \\ \par \vspace{0.5ex}
	\begin{subfigure}{0.32\textwidth}
		\centering
		\includegraphics[width=1.0\textwidth]{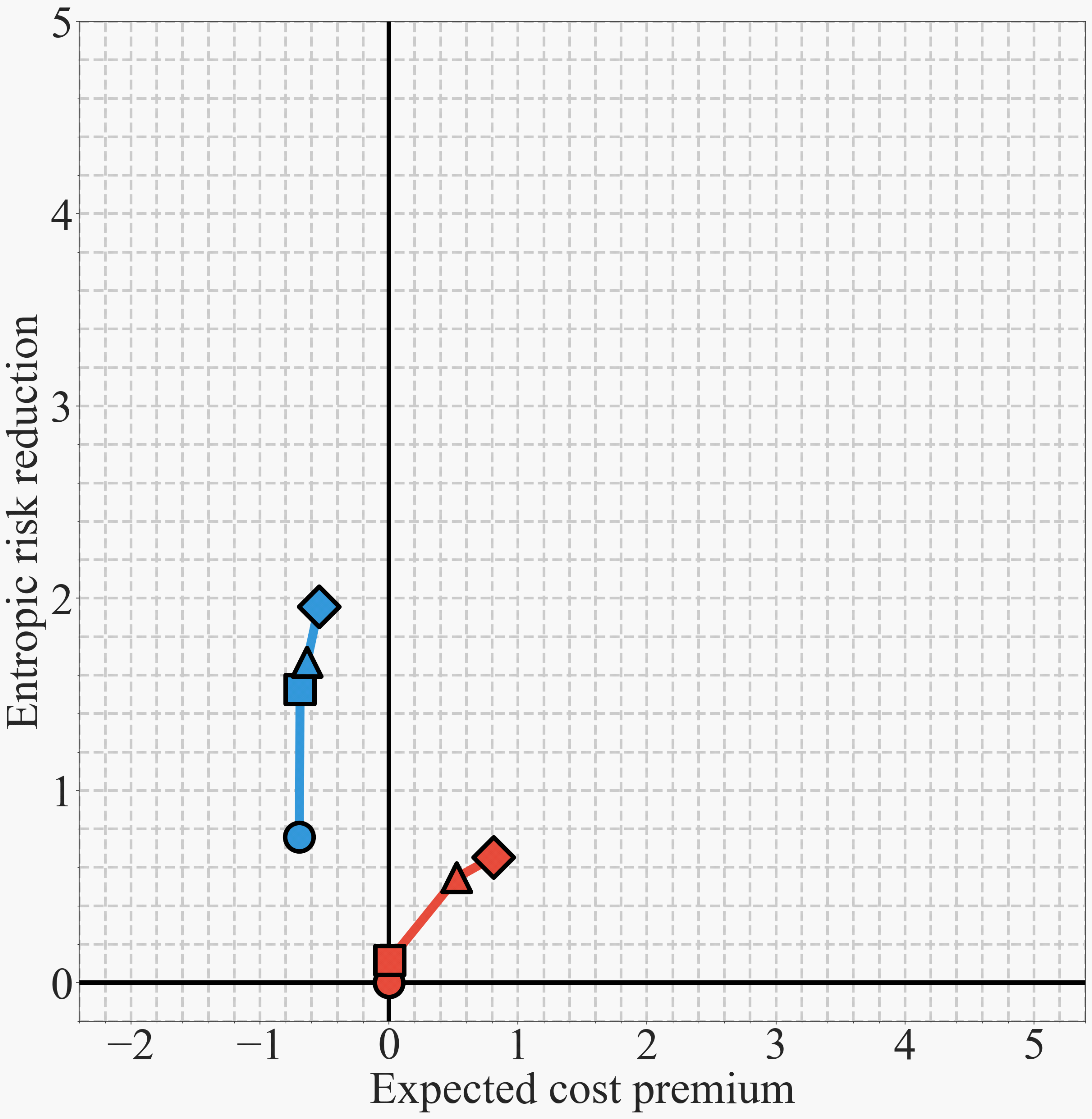}
		\caption{$(\theta,\sigma)=(0.2,0.2)$}
		\label{fig:Performance1}  
	\end{subfigure}
	\begin{subfigure}{0.32\textwidth}
		\centering
		\includegraphics[width=1.0\textwidth]{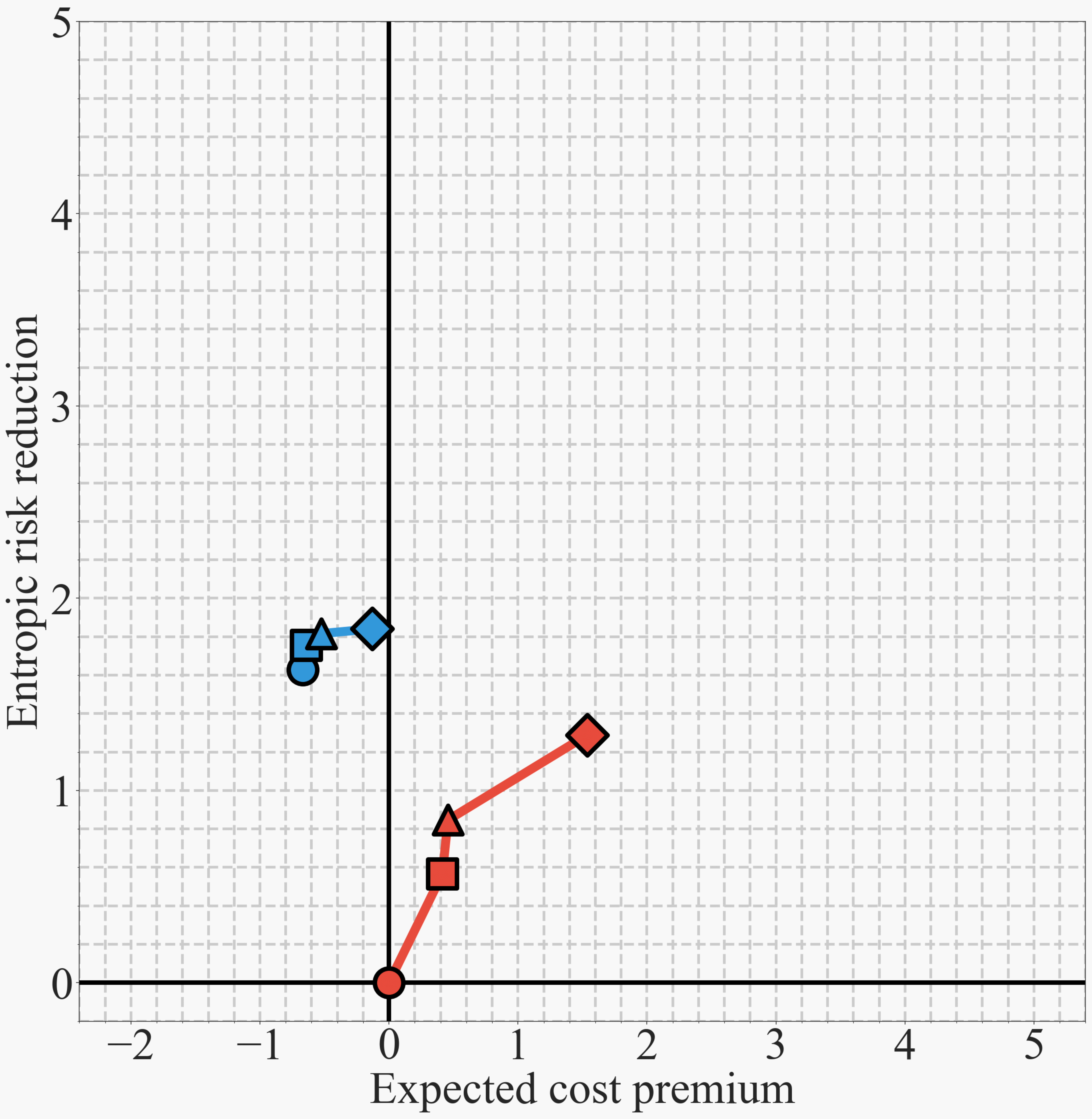}
		\caption{$(\theta,\sigma)=(0.2,0.5)$}
		\label{fig:Performance2}  
	\end{subfigure} 
	\begin{subfigure}{0.32\textwidth}
		\centering
		\includegraphics[width=1.0\textwidth]{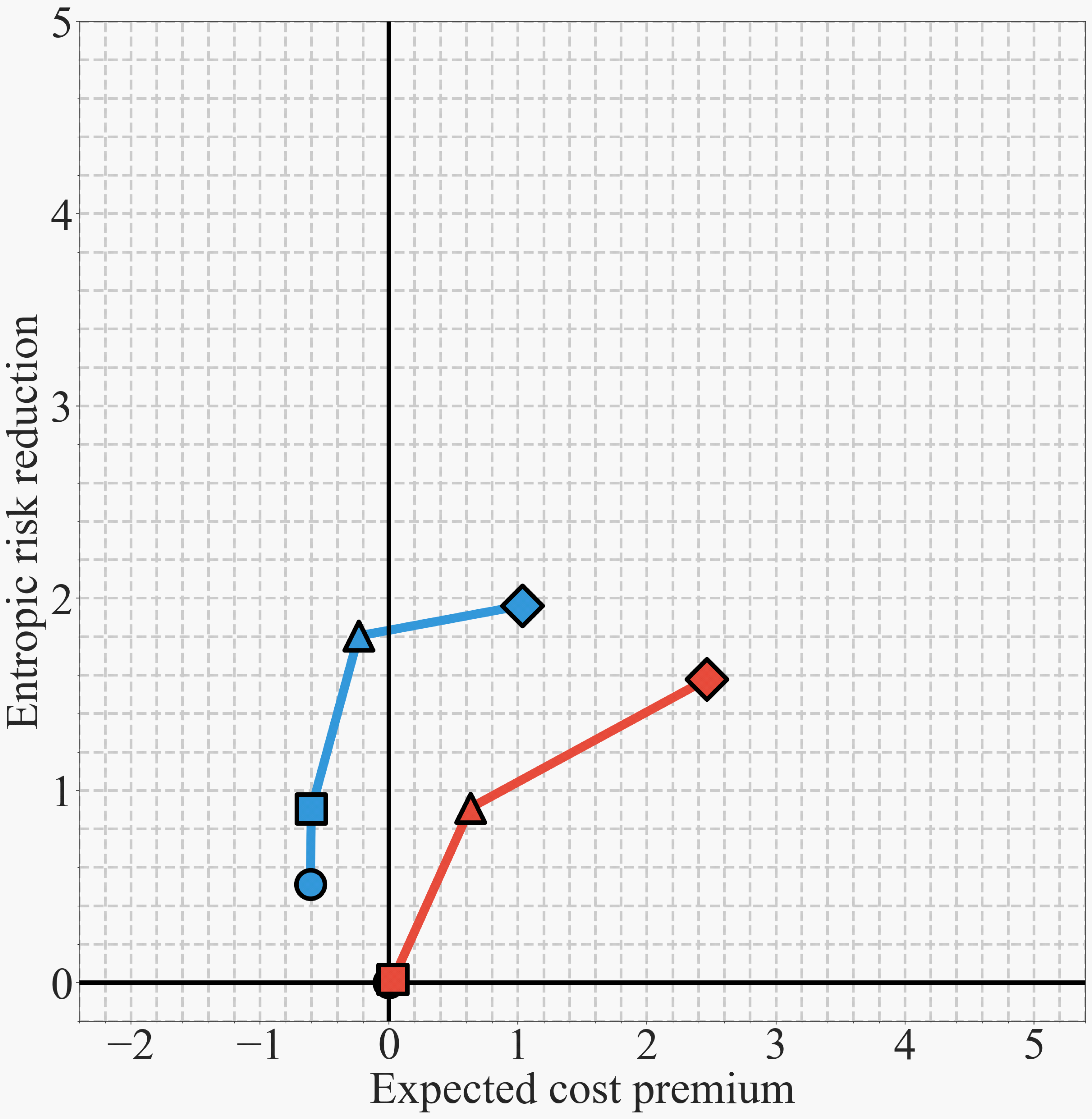}
		\caption{$(\theta,\sigma)=(0.2,1.0)$}
		\label{fig:Performance3}  
	\end{subfigure} \\ \par \bigskip
	\begin{subfigure}{0.32\textwidth}
		\centering
		\includegraphics[width=1.0\textwidth]{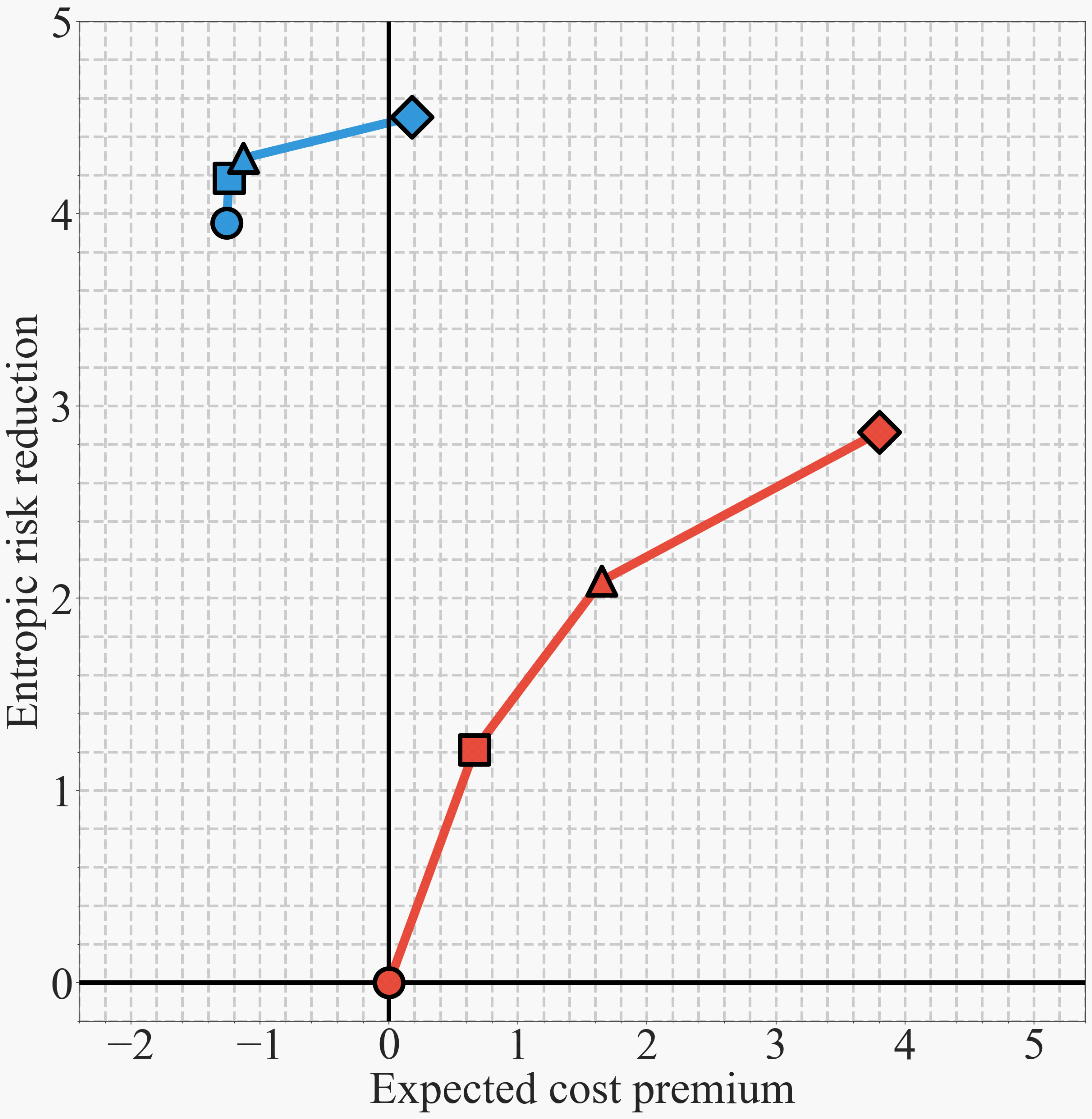}
		\caption{$(\theta,\sigma)=(0.5,0.5)$}
		\label{fig:Performance4}  
	\end{subfigure} 
	\begin{subfigure}{0.32\textwidth}
		\centering
		\includegraphics[width=1.0\textwidth]{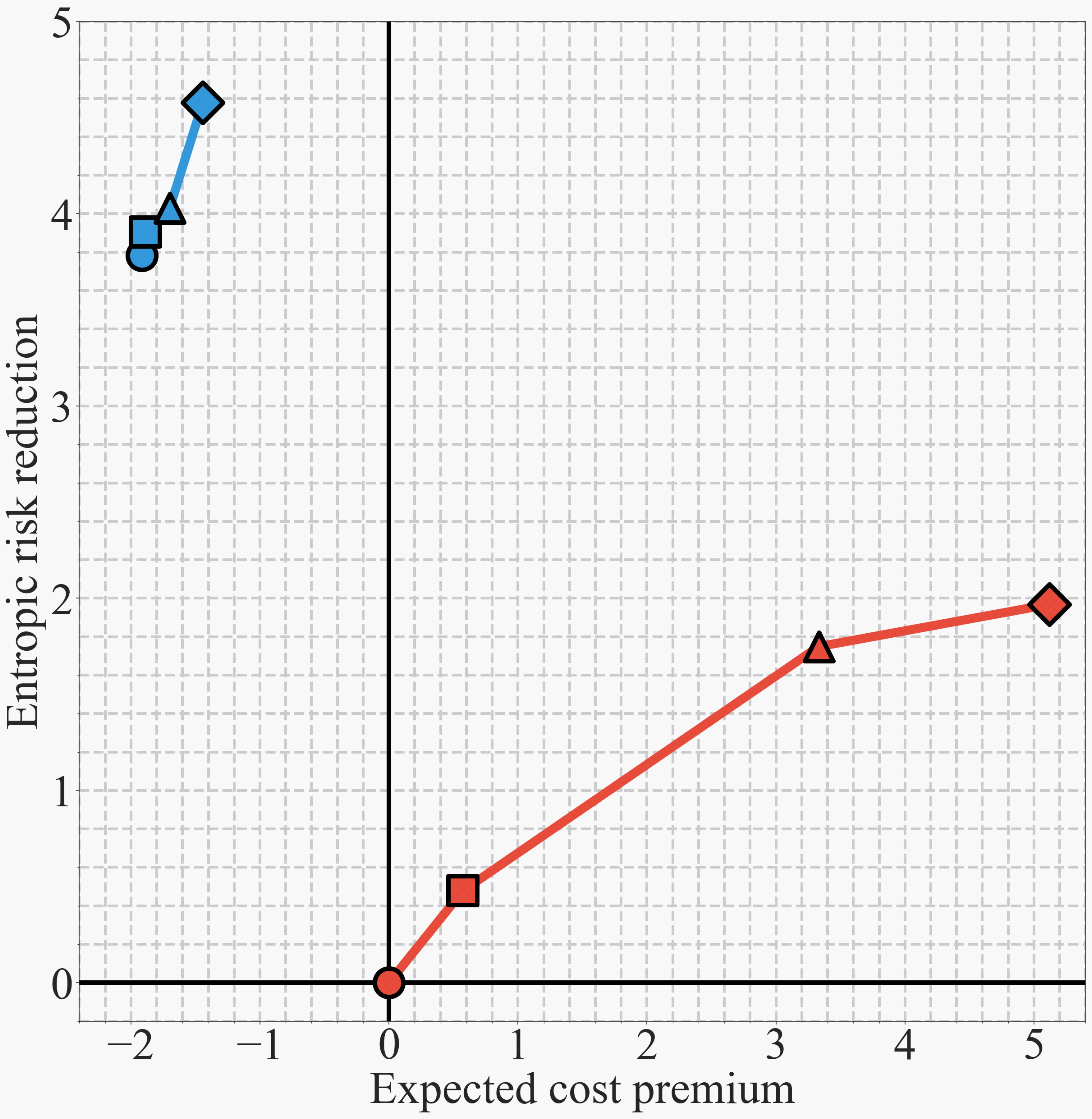}
		\caption{$(\theta,\sigma)=(0.7,0.2)$}
		\label{fig:Performance5}  
	\end{subfigure}
	\begin{subfigure}{0.32\textwidth}
		\centering
		\includegraphics[width=1.0\textwidth]{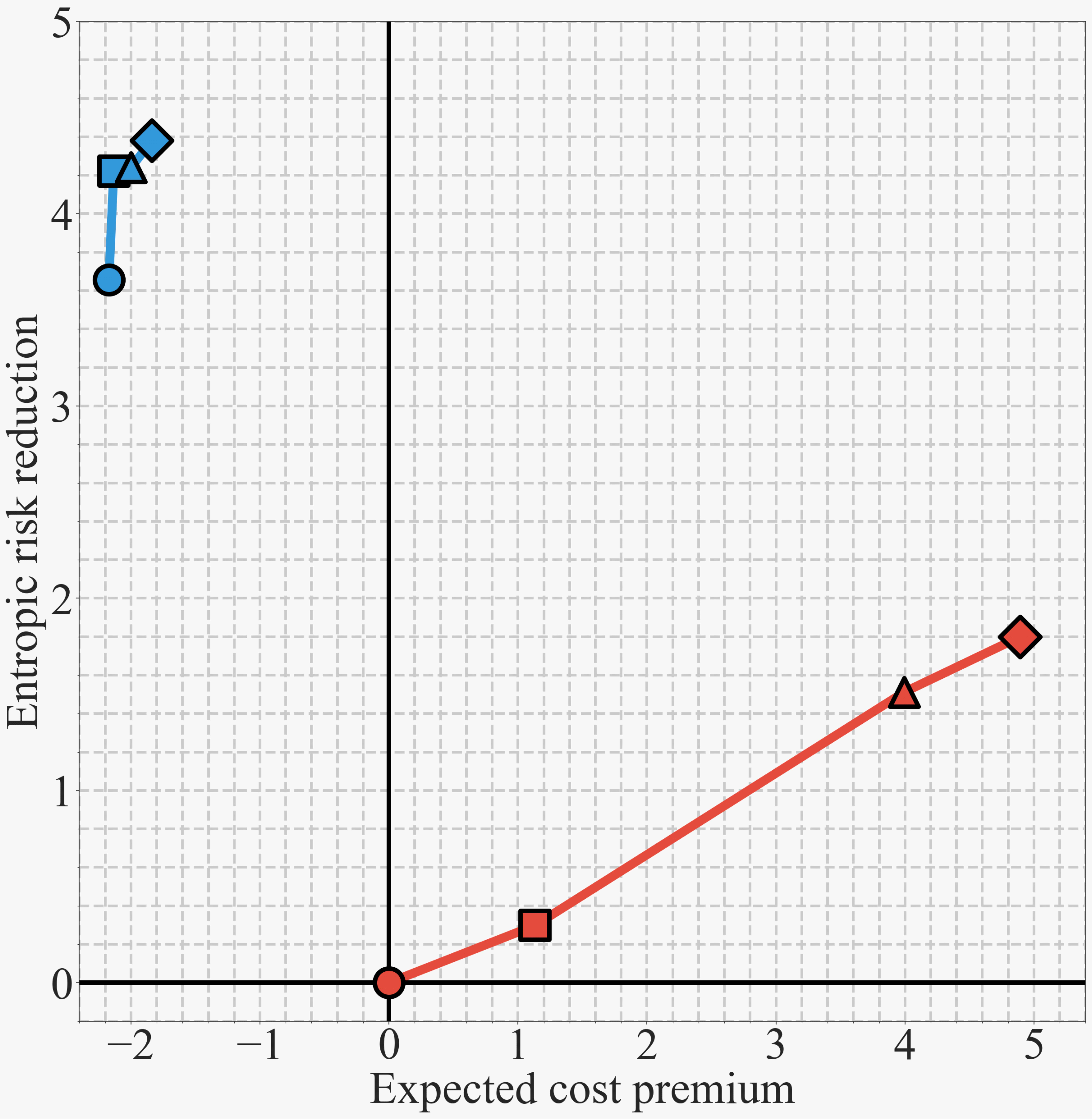}
		\caption{$(\theta,\sigma)=(0.8,0.0)$}
		\label{fig:Performance6}  
	\end{subfigure} 
	\caption{Expected cost premium versus  Entropic risk reduction in the SAV system.}
	\label{fig:Performance}  
\end{figure}

{
This subsection reports the impact of introducing flexibility and reliability into transportation network design on system performance.
Our validation employs the following two metrics to evaluate averaged capability and risk hedging capability:
\begin{flalign*}
	&{\rm (Expected~cost~premium)}:=\frac{{\rm M}[\Psi^{{\rm model}_\gamma}]-{\rm M}[\Psi^{{\rm single}_0}] }{ {\rm M}[\Psi^{{\rm single}_0}] },&\\
	&{\rm (Entropic~risk~reduction)}:=\frac{{\rm Ent}[\Psi^{{\rm single}_0}]-{\rm Ent}[\Psi^{{\rm model}_\gamma}] }{{\rm Ent}[\Psi^{{\rm single}_0}]},&
\end{flalign*}
where ${\rm model}\in\{{\rm single},{\rm multi}\}$.
$\Psi^{{\rm model}_\gamma}$ denotes the random total cost under the optimal design obtained from the corresponding
${\rm model}$ formulation with risk-aversion parameter $\gamma$, and $\Psi_m^{{\rm model}_\gamma}$ is its realization in the
$m$-th Monte Carlo simulation ($m=1,\ldots,M$).
We compute the sample mean and the entropic-risk metric as
${\rm M}[\Psi^{{\rm model}_\gamma}]:=\frac{1}{M}\sum_{m=1}^M \Psi_m^{{\rm model}_\gamma}$ and
${\rm Ent}[\Psi^{{\rm model}_\gamma}]
:=\frac{1}{\bar{\gamma}}\log\!\left(\frac{1}{M}\sum_{m=1}^M\exp\!\left(\bar{\gamma}\,\Psi_m^{{\rm model}_\gamma}\right)\right),
{\rm where~} \bar{\gamma}=10^{-4}$.
Expected cost premium measures the relative change in the averaged total cost with respect to the risk-neutral and single-stage
baseline; negative values indicate a reduction in expected cost.
Entropic risk reduction measures the relative reduction in right-tail costs; larger values indicate better hedging against severe cases.

Figure \ref{fig:Performance} depicts Entropic risk reduction versus Expected cost premium for different uncertainty levels $(\theta,\sigma)$.
The vertical and horizontal axes represent Entropic risk reduction and Expected cost premium, respectively.
Blue and red curves correspond to the multi-stage and single-stage network designs.
Markers indicate the risk-aversion parameter: circles ($\gamma=0$) represent a risk-neutral case, while squares
($\gamma=10^{-6}$), triangles ($\gamma=10^{-5}$), and diamonds ($\gamma=10^{-4}$) represent increasingly risk-averse cases.
By definition, the risk-neutral and single-stage baseline is located at the origin $(0,0)$.
The points above the horizontal axis achieve risk reduction, and the points left of the vertical axis achieve lower expected cost than the baseline.

From Figure \ref{fig:Performance}, we can observe that both the Entropic risk reduction and Expected cost premium increase as the risk-aversion parameter $\gamma$ increases.
This captures a typical characteristic of risk-aversion models, which is to mitigate the right-tail costs associated with severe scenarios at the expense of average performance.
Moreover, 
} 
across all combinations of uncertainty scales $(\theta,\sigma)$ and risk-aversion parameter $\gamma$,
{
the multi-stage designs lie above and to the left of the corresponding single-stage designs, indicating that}
multi-stage design simultaneously reduces average cost and right-tail cost compared to single-stage design.
{ This advantage becomes more pronounced under high strategic uncertainty $\theta$, as shown in Figures \ref{fig:Performance4}--\ref{fig:Performance6}.
Figures \ref{fig:Performance1}, \ref{fig:Performance4}, and \ref{fig:Performance6} further show that under conditions of low operational uncertainty $\sigma$,
multi-stage design achieves Entropic risk reduction with a relatively modest increase in Expected cost premium compared to single-stage design.
}
These results confirm that incorporating both flexibility (through multi-stage planning) and reliability (through risk-aversion) enables complementary benefits, hedging against extreme outcomes without unduly sacrificing average performance.

%% file: ch6.tex
\section{Conclusions}

This study proposed a general framework for flexible and reliable network design problems (FR-NDPs) tailored to emerging transportation services. By jointly incorporating flexibility---the ability to adjust, postpone, or abandon infrastructure investments---and reliability---the ability to maintain a desired performance level under uncertainty---the framework enables planners to hedge against severe scenarios while avoiding unnecessary overinvestment. For emerging transportation services such as shared autonomous vehicle (SAV) systems or integrated multimodal operations, where historical data are scarce and technology trajectories uncertain, these properties are particularly valuable. They allow infrastructure deployment to be synchronized with realized demand growth and technological developments, ultimately improving both service level and cost effectiveness.

The FR-NDPs are formulated as risk-averse multi-stage stochastic programs (MSSPs) and solved using stochastic dual dynamic programming (SDDP). This approach decomposes the problem into tractable stagewise subproblems, overcomes the curse-of-horizon, and guarantees global optimality under mild convexity conditions, while preserving time-consistent risk preferences.
Beyond simply leveraging SDDP’s ability, this study systematically established sufficient conditions for convergence of the lower bound, for computation of statistical upper bounds, and for applying Dual SDDP to obtain deterministic upper bounds.
By explicitly linking these properties to transportation system models, the proposed framework provides a unified and extensible modeling structure that allows for the incorporation of diverse strategic and operational models, while also achieving both computability and provable solution quality.

The proposed framework was applied to two distinct cases: (i) capacity expansion planning for SAV systems, and (ii) sequential route design for integrated SAV-BRT (Bus Rapid Transit) systems. In both cases, existing strategic decision models are incorporated directly into the FR-NDP formulation with minimal modification. In the SAV application, the mathematical tractability of the MSSP formulation enabled the derivation of an economic property---an infrastructure cost recovery principle---linking optimal congestion tolls to infrastructure expansion and maintenance costs. These applications illustrate the framework’s ability to accommodate varied system structures and decision processes while delivering analytical insights.

Our numerical experiments on the Midtown Manhattan network examined how incorporating flexibility and reliability affects both strategic and operational decisions, as well as overall system performance.
In the SAV application, multi-stage planning produced not only different temporal investment sequences but also distinct spatial allocation patterns compared with single-stage designs.
Moreover, the impact of risk aversion differed by design setting: in single-stage NDPs, risk-averse preferences led to higher infrastructure investments to hedge against operational risks, whereas in multi-stage NDPs, early investments were restrained to preserve the ability to adapt later, reflecting a concern over locking in excessive capacity too soon.
In the SAV-BRT application, the lower operational flexibility of BRT services was reflected in the strategic decisions, with risk-averse planning shortening BRT routes to limit exposure to uncertain demand. 
These findings highlight the importance of jointly considering flexibility, reliability, and the interaction between operational and strategic decisions. 
From a performance perspective, multi-stage designs consistently outperformed their single-stage counterparts, and the advantage became more pronounced as demand uncertainty increased. 
Flexibility and reliability also exhibited a complementary effect, enabling strategies that hedge against severe scenarios while mitigating the loss of expected performance.
In contrast, single-stage designs struggled to hedge against risk without incurring substantial reductions in expected performance.
These results highlight the value of jointly incorporating flexibility and reliability to achieve both robust and high-performing long-term network planning.

{
A significant limitation concerns computational scalability.
While the convergence guarantees are independent of problem scale,
they do not bound the number of iterations or the computation time required for convergence.
Addressing this limitation is an important direction for future work.}
{
One promising approach is Neural SDDP \citep{dai2021neural}, which could handle high-dimensional action spaces while maintaining convergence guarantees based on optimality cuts.
}

An important next step {in network design modeling} is to extend the framework to capture endogenous demand responses to sequential investment and operational decisions.
In emerging transportation systems, infrastructure provision can itself influence travel demand, mode choice, and temporal patterns of usage, creating feedback loops that may amplify or dampen the benefits of flexibility and reliability.
Embedding such demand-side dynamics in the risk-averse multi-stage network design would improve its ability to inform policy decisions that balance user behavior, infrastructure efficiency, and long-term resilience.

%% file: chx.tex
{
\section{Previous studies on Stochastic Dual Dynamic Programming}
\label{app:sddp0}

In this Appendix,
we review the previous studies on Stochastic Dual Dynamic Programming (SDDP),
which forms the foundation of the proof of Propositions \ref{prop:suff}--\ref{prop:suff3}.
}

\subsection{Description of the standard SDDP algorithm}
\label{app:sddp}

SDDP has gained enormous interest from a theoretical and an application perspective, starting with the original work \citep{pereira1991multi}.
{\
This section summarizes the standard algorithm and convergence properties of SDDP to solve multi-stage stochastic problems (MSSPs) of the form
\begin{subequations}
	\label{eq:originalMSSP}
	\begin{flalign}
		&\min_{\bm{x}_1,\bm{u}_1}C_1(\bm{x}_1,\bm{u}_1,\bm{\xi}_1)+\mathbb{F}_{\bm{\xi}_1}\left[\min_{\bm{x}_2, \bm{u}_2}C_2(\bm{x}_2,\bm{u}_2,\bm{\xi}_2)+\mathbb{F}_{\bm{\xi}_{[2]}}\left[\cdots+\mathbb{F}_{\bm{\xi}_{[S-1]}}\left[\min_{\bm{x}_S,\bm{u}_S}C_S(\bm{x}_S,\bm{u}_S,\bm{\xi}_S)\right]\right]\right],&
	\end{flalign}
	\vspace{-2ex}
	\begin{flalign}
		&{\rm~s.t.~~}\bm{x}_s=T_s(\bm{x}_{s-1}, \bm{u}_s, \bm{\xi}_s)&\forall s\in\mathcal{S}:=\{1,...,S\},\\
		&~~~~~~~~\bm{u}_s\in U_s(\bm{x}_s, \bm{\xi}_s)&\forall s\in\mathcal{S}:=\{1,...,S\},
	\end{flalign}
\end{subequations}
where $\bm{x}_s$ is a state variable vector, $\bm{u}_s$ is a control variable vector, $\bm{\xi}_s$ a random variable vector, 
and $\bm{x}_0$ and $\bm{\xi}_1$ are given.
$\mathbb{F}_{\bm{\xi}_{[t]}}$ is a conditional risk measure on history of the random process $\bm{\xi}_{[t]}=\{\bm{\xi}_1,...,\bm{\xi}_t\}$.
$C_s$, $T_s$, and $U_s$ represent cost functions, transition functions, and action spaces, respectively.
} The standard assumptions commonly used are as follows:
\begin{ass} (Linearity)
	\label{app:assumption_linearity}
	Given fixed $\bm{\xi}_s$, the cost function $C_s$ is linear, the transition function $T_s$ is linear, and the action space $U_s$ is a non-empty compact polyhedral set for all stages $s\in\mathcal{S}$,
\end{ass}
\begin{ass} (Risk-neutral policy)
	\label{app:assumption_risk_neutral}
	$\mathbb{F}^{}_{\bm{\xi}_{s}}[\cdot]$ is $\mathbb{E}_{\bm{\xi}_s}[\cdot]$ for all stages $s\in\mathcal{S}$, 
\end{ass}
\begin{ass} (Continuous variables)
	\label{app:assumption_continuous}
	All the variables are continuous for all stages $s\in\mathcal{S}$, 
\end{ass}
\begin{ass} (Finite planning horizon)
	\label{app:assumption_finite_horizon}
 	The length of the planning horizon $S$ is finite and given,
\end{ass}
\begin{ass} (Finite randomness)
	\label{app:assumption_finite_randomness}
	The number of realizations of $\bm{\xi}_{s}$, $N_s$, is finite for all stages $s\in\mathcal{S}$,
\end{ass}
\begin{ass} (Stagewise independence)
	\label{app:assumption_independence}
	Realizations of $\bm{\xi}_{s}$ are independent of other stages for all stages $s\in\mathcal{S}$, and
\end{ass}
\begin{ass} (RCR: Relatively complete recourse)
	\label{app:assumption_RCR}
	For all feasible $\bm{x}_{s-1}$ and almost all realizations of $\bm{\xi}_{s}$, $\mathcal{X}_{s}(\bm{x}_{s-1}, \bm{\xi}_s):=\{(\bm{x}_s,\bm{u}_s)|\bm{x}_s=T_s(\bm{x}_{s-1}, \bm{u}_s, \bm{\xi}_s),~\bm{u}_s\in U_s(\bm{x}_s, \bm{\xi}_s)\}$ is non-empty for all stages $s\in\mathcal{S}$.
\end{ass}

{
For clarity of exposition,
we consider the following explicit formulation of MSSPs satisfying Assumptions \ref{app:assumption_linearity}--\ref{app:assumption_RCR}:}
\begin{subequations}
	\label{eq:MSSP}
	\begin{flalign}
		&\min_{\bm{x}_1\ge\bm{0}}\bm{c}_1^{\top}\bm{x}_1+\mathbb{E}_{\bm{\xi}_1}\left[\min_{\bm{x}_2\ge\bm{0}}\bm{c}_2^{\top}\bm{x}_2+\mathbb{E}_{\bm{\xi}_2}\left[\cdots+\mathbb{E}_{\bm{\xi}_{S-1}}\left[\min_{\bm{x}_S\ge\bm{0}}\bm{c}_S^{\top}\bm{x}_S\right]\right]\right],&\label{eq:MSSP_objective}\\
		&\bm{A}_s\bm{x}_s+\bm{B}_s\bm{x}_{s-1}=\bm{b}_s,~{\rm where~}\bm{x}_0 {\rm~and~} \bm{\xi}_1{\rm~are~given,}~~~~~~~~~~~~~~~~~~~~~~~~~~~~~~~~~~~~~~~~~~~~~~~~~~~~~~~~~~~~~~~~~~~~~~~~~&\forall s, \label{eq:MSSP_transition}
	\end{flalign}
\end{subequations}
where the vectors $\bm{b}_s,\bm{c}_s$ and the matrices $\bm{A}_s,\bm{B}_s$ are random variables forming the stochastic process $\bm{\xi}_s=(\bm{b}_s,\bm{c}_s,\bm{A}_s,\bm{B}_s)$.
$\mathbb{E}_{\bm{\xi}_{s-1}}[\cdot]$ denotes the conditional expectation on $\bm{\xi}_{s}$.
Eqs. (\ref{eq:MSSP}) can be decomposed into the following subproblems based on Bellman's principle:
\begin{subequations}
	\label{eq:MSSP_DP}
	\begin{flalign}
		&V_{s}(\bm{x}_{s-1},{\bm{\xi}}_s)=\min_{\bm{x}_s\in\bm{\mathcal{X}}_s(\bm{x}_{s-1},{\bm{\xi}}_s)}\left[\bm{c}_s^{\top}\bm{x}_s+\mathcal{V}_{s+1}(\bm{x}_s)\right]&\forall s,\label{eq:MSSP_DP_a}\\
		&\mathcal{V}_{s+1}(\bm{x}_{s})=\mathbb{E}_{{\bm{\xi}}_{s}}\left[V_{s+1}(\bm{x}_{s},{\bm{\xi}}_{s+1})\right] {\rm~where~}\mathcal{V}_{S+1}(\cdot)=0&\forall s,\label{eq:MSSP_DP_b}
	\end{flalign}
\end{subequations}
where $V_{s}(\bm{x}_{s-1},\bm{{\xi}}_{s})$ is the value function at stage $s$ and $\mathcal{V}_{s}(\bm{x}_{s-1})$ is the expected value function.
A dynamic programming equation (\ref{eq:MSSP_DP}) is solved when the value function$V_{s}(\bm{x}_{s-1},\bm{{\xi}}_{s})$ is identified.
Under Assumption \ref{app:assumption_RCR}, ${V}_{s}(\bm{x}_{s-1})$ is well defined and finite valued.
Furthermore, under Assumption \ref{app:assumption_finite_randomness},
the computation of Eq. (\ref{eq:MSSP_DP_b}) is achieved by the weighted summation based on probability.
Denoting the set of realizations of $\bm{\xi}_{s}$ as $\{\bm{\xi}_s^1,.... \bm{\xi}_s^{N_s}\}$ and the corresponding probability as $\{{p}_s^1,.... {p}_s^{N_s}\}$, Eq. (\ref{eq:MSSP_DP}) can be rewritten as follows:
\begin{subequations}
	\label{eq:MSSP_DP_SSA}
	\begin{flalign}
		&{V}_{s}(\bm{x}_{s-1},\bm{\xi}^n_{s}) =\min_{\bm{x}_s\in\bm{\mathcal{X}}_s(\bm{x}_{s-1},{\bm{\xi}}^n_s)}\left[(\bm{c}_s^{n})^{\top}\bm{x}_s+\mathscr{V}_{s+1}(\bm{x}_s)\right]&\label{eq:MSSP_DP_SSA_a}\forall s,n=1,...,N_s,\\
		&\mathscr{V}_{s}(\bm{x}_{s-1})=\sum_{n=1}^{N_s}{p}_s^n{V}_{s}(\bm{x}_{s-1},\bm{\xi}_{s}^n),{\rm~where~}\mathscr{V}_{S+1}(\bm{x}^{}_{S})=0&\forall s. \label{eq:MSSP_DP_SSA_b}
	\end{flalign}
\end{subequations}

SDDP forms a lower approximation of the value function using a set of piecewise linear functions, referred to as cuts.
The SDDP algorithm consists of a forward step to evaluate policies and a backward step to improve the policy.
The backward step updates information about the expected value function by going back from the final step to the initial step and solving a subproblem at each stage.
This information is obtained as a piecewise linear approximation of the expected value function in the form of a Benders cut.
Under loose conditions (e.g., convexity), the piecewise linear function defined by a set of Benders cuts provides an outer approximation (i.e., lower bound) of the expected value function.
The forward step generates realizations (sample paths) of the stochastic process and computes the optimal solution corresponding to the realizations from the initial to the final step.
The {set of Benders cuts} obtained in the backward step is used as an approximation of the expected value function at each stage.
The calculation procedure is shown in Algorithm \ref{alg:overview}.

\begin{figure}[!t]
	\begin{algorithm}[H]
		\caption{An overview of the SDDP algorithm}
		\label{alg:overview}
		\begin{algorithmic}[]
			\setlength{\itemsep}{0pt} \setlength{\parskip}{0pt}
			\REQUIRE $\varepsilon > 0~({\rm convergence~condition}),~\{{\bm{\xi}}^1_s,...,{\bm{\xi}}^N_s\}_{s=2,...,S}$
			\STATE Initialize: iteration$\leftarrow0$, $\underline{\Psi}=-\infty~({\rm lower~bound}),$ $ \overline{\Psi}=\infty~({\rm upper~bound}),$  and $\{{\bm{\mathcal{K}}}_s=\emptyset\}_{s=1,...,S-1}~({\rm cut~index~sets})$
			\WHILE{$\overline{\Psi} - \underline{\Psi}>\varepsilon$}
			\STATE Perform the forward step: Update the upper bound $\overline{\Psi}$
			\STATE Perform the backward step: Update the lower bound $\underline{\Psi}$
			\STATE Update the iteration number: iteration$\leftarrow$iteration$+1$
			\ENDWHILE
		\end{algorithmic}
	\end{algorithm}
\end{figure}

\subsubsection{Backward Step}
\label{backward}

The backward step solves the subproblem backward using the $M$ trial solutions $\{\overline{\bm{x}}_s^m\}_{s=1,...,S-1,m=1,...,M}$ computed in the forward step.
In the following, we outline the computational procedure for the $m$-th trial solution to avoid notational complexity.

Let us consider a subproblem at the final stage $s=S$.
Given a trial solution $\overline{\bm{x}}_{S-1}$, Eq. (\ref{eq:MSSP_DP_SSA}) can be rewritten as follows:
\begin{subequations}
	\label{eq:backwardT}
	\begin{flalign}
		&{V}_{S}(\overline{\bm{x}}_{S-1},\bm{\xi}^n_{S}) = \min_{\bm{x}_S\ge0, \bm{A}^n_S\bm{x}_S=-\bm{B}^n_S\overline{\bm{x}}_{S-1}+\bm{b}^n_S}(\bm{c}_S^{n})^{\top}\bm{x}_S &\forall n=1,...,N_S,\label{eq:HJB_backwardT}\\
		&\mathscr{V}_{S}(\overline{\bm{x}}_{S-1})=\sum_{n=1}^{N_S}p_S^n{V}_{S}(\overline{\bm{x}}_{S-1},\bm{\xi}_{S}^n).\label{eq:ExpQ_backwardT}&
	\end{flalign}
\end{subequations}
The outer approximation of the expected value function is formed by a set of the Benders cuts.
Under Assumption \ref{app:assumption_linearity} and \ref{app:assumption_continuous}, since $\mathscr{V}_{S}(\overline{\bm{x}}_{S-1})$ is a piecewise linear function, the Benders cut satisfies the following inequality:
\begin{flalign}
	\label{eq:Benders_inequality_multi}
	&{V}_{S}({\bm{x}}_{S-1},\bm{\xi}^n_{S})\ge{V}_{S}(\overline{\bm{x}}_{S-1},\bm{\xi}^n_{S})+(\bm{g}_S^n)^{\top}(\bm{x}_{S-1}-\overline{\bm{x}}_{S-1})&\forall n=1,...,N_S,\\
	\label{eq:Benders_inequality_single}
	&\mathscr{V}_{S}({\bm{x}}_{S-1})\ge\mathscr{V}_{S}(\overline{\bm{x}}_{S-1})+\bm{g}_S^{\top}(\bm{x}_{S-1}-\overline{\bm{x}}_{S-1}).&
\end{flalign}
where $\bm{g}_S$ and $\bm{g}_S^n$ are subgradients of $\mathscr{V}_{S}({\bm{x}}_{S-1})$ and ${V}_{S}({\bm{x}}_{S-1},\bm{\xi}^n_{S})$ at $\overline{\bm{x}}_{S-1}$, respectively,
using the dual variable vector $\bm{\pi}_S^n$ corresponding to the transition equation (\ref{eq:MSSP_transition}), given by
\begin{flalign}
	&\bm{g}_S^n=-(\bm{B}_S^{n})^{\top}\bm{\pi}_S^n&\forall n=1,...,N_S,\\
	&\bm{g}_S=\sum_{n=1}^{N_S}p_s^n\bm{g}_S^n.&
\end{flalign}
From Eqs. (\ref{eq:Benders_inequality_multi}) and (\ref{eq:Benders_inequality_single}), the Benders cut $\ell({\bm{x}}_{S-1})$ is given by
\begin{flalign}
	&\ell({\bm{x}}_{S-1}):=\bm{g}_S^{\top}\bm{x}_{S-1}+\beta_{S},&\\
	&\beta_{S}=\mathscr{V}_{S}(\overline{\bm{x}}_{S-1})-\bm{g}_S^{\top}\overline{\bm{x}}_{S-1}.&
\end{flalign}
Under Assumption \ref{app:assumption_independence}, this Benders cut can be shared by all problems in $(S-1)$ stage.

Then, consider a subproblem at the $s=S-1$ stage.
The index set of Benders cuts formed by the above procedure is denoted by $\bm{\mathcal{K}}_{S-1}$.
The outer approximation of the expected value function can then be defined as
\begin{flalign}
	&\mathfrak {V}_S:=\max_{k\in\bm{\mathcal{K}}^{\rm }_{S-1}}\{\bm{g}_{S,k}^{\top}{\bm{x}}_{S-1}+\beta_{S,k}\}.&
\end{flalign}
Given a trial solution $\overline{\bm{x}}_{S-2}$, a subproblem at the $s=S-1$ stage can be rewritten as follows:
\begin{subequations}
	\begin{flalign}
		\label{eq:HJB_backwardT-1}
		&\min_{\bm{x}_{S-1}\ge0,\mathfrak{V}_{S}\in\mathbb{R}, \bm{A}^n_{S-1}\bm{x}_{S-1}=-\bm{B}^n_{S-1}\overline{\bm{x}}_{S-2}+\bm{b}^n_{S-1}}(\bm{c}_{S-1}^{n })^{\top}\bm{x}_{S-1} + \mathfrak{V}_{S}&\forall n=1,...,N_{S-1},\\
		\label{eq:BendersCutting_T-1}
		&\mathfrak{V}_{S}\ge(\bm{x}_{S-1})\bm{g}_{S,k}^{\top}+\beta_{S,k}&\forall k\in\bm{\mathcal{K}}_{S-1}.
	\end{flalign}
\end{subequations}
The Benders cut can be computed using the same procedure as in $s=S$ subproblem.
Applying the same computational procedure to $s=S-2,...,1$, we obtain an outer approximation of the expected value function at each stage. 

The backward step finally solves the subproblem at the initial stage $s=1$:
\begin{subequations}
	\label{eq:single_backward0}
	\begin{flalign}
		\label{eq:single_HJB_backward0}
		&\underline{\Psi}=\min_{\bm{x}_{1}\ge0, \mathfrak{V}_{2}\in\mathbb{R}, \bm{A}_{1}\bm{x}_{1}=-\bm{B}_{1}\bm{x}_{0}+\bm{b}_{1}}\bm{c}_{1}^{\top}\bm{x}_{1} + \mathfrak{V}_{2},&\\
		\label{eq:single_BendersCutting0}
		&\mathfrak{V}_{2}\ge\bm{x}_{1}\bm{g}_{2,k}^{\top}+\beta_{2,k}&\forall k\in\bm{\mathcal{K}}_1.
	\end{flalign}
\end{subequations}
Eq. (\ref{eq:single_backward0}) provides a lower bound of the optimal value.
Since each iteration of the backward step adds a new Benders cut to Eq. (\ref{eq:single_BendersCutting0}), we can guarantee that the optimal value of Eq. (\ref{eq:single_backward0}) is non-decreasing.
The pseudo-code for the backward step is shown in Algorithm \ref{alg:single_Backward}.

\begin{figure}[!t]
	\begin{algorithm}[H]
		\caption{Backward pass of the SDDP algorithm for $M$ trial solutions}
		\label{alg:single_Backward}
		\begin{algorithmic}[]
			\setlength{\itemsep}{0pt} \setlength{\parskip}{0pt}
			\REQUIRE $\{\overline{\bm{x}}_s^m\}_{s=0,...,S-1, m=1,...,M}$
			\FOR{$s=S\rightarrow2$}
			\FOR{$m=1\rightarrow M$}
			\FOR{$n=1\rightarrow N_s$}
			\STATE Solve $s$-th stage problem $(\mathfrak{V}_{S+1}=0, \underline{V}_S(\overline{\bm{x}}_{S-1}^m,\bm{\xi}^n_S)={V}_S(\overline{\bm{x}}_{S-1}^m,\bm{\xi}^n_S),~ \underline{\mathscr{V}}_S(\overline{\bm{x}}_{S-1}^m)=\mathscr{V}_S(\overline{\bm{x}}_{S-1}^m)):$
			\STATE $\left(\underline{V}_s(\overline{\bm{x}}_{s-1}^m,\bm{\xi}^n_s), \bm{\pi}_s^{n,m}\right)\leftarrow$
			\STATE $\min_{\bm{x}_s\ge0, \mathfrak{V}_{s+1}\in\mathbb{R}}\left\{\bm{c}_s^{n\top}\bm{x}_s+\mathfrak{V}_{s+1}|\bm{A}^n_s\bm{x}_s=-\bm{B}^n_t\overline{\bm{x}}_{s-1}^m+\bm{b}_s^n,~\mathfrak{V}_{s+1}\ge (\bm{g}_{s+1,k}^m)^{\top}\bm{x}_t+\beta^m_{s+1,k},~\forall k\in\bm{\mathcal{K}}_s^{\rm } \right\}$
			\STATE Compute $\bm{g}_s^{n,m}$: $\bm{g}_s^{n,m}=-(\bm{B}^{n}_s)^{\top}\bm{\pi}_s^{n,m}$
			\ENDFOR
			\STATE Compute $\bm{g}_s^m,~\mathscr{V}_s(\overline{\bm{x}}_{s-1}^m),~{\rm and}~\beta_s^m:$
			\STATE $\bm{g}_s^m=\sum_{n=1}^{N_s}p_s^n\bm{g}_s^{n,m},$ 
			\STATE $\mathscr{V}_s(\overline{\bm{x}}_{s-1}^m)=\sum_{n=1}^{N_s}p_s^n\underline{V}_s(\overline{\bm{x}}_{s-1}^m,\bm{\xi}^n_s),$
			\STATE $\beta_s^m=\mathscr{V}_s(\overline{\bm{x}}_{s-1}^m)-(\bm{g}_s^m)^{\top}\overline{\bm{x}}_{s-1}^{m}.$
			\STATE Add the new cut $\mathfrak{V}_s\ge (\bm{g}_s^m)^{\top}\bm{x}_{s-1}+\beta^m_{s}$ to all $(s-1)$-th stage problems
			\ENDFOR
			\ENDFOR
			\STATE Update the lower bound: Solve 1st stage problem
			\STATE $\underline{\Psi}\leftarrow\min_{\bm{x}_1\ge0, \mathfrak{V}_{2}\in\mathbb{R}}\left\{\bm{c}_1^{\top}\bm{x}_1+\mathfrak{V}_{2}|\bm{A}_1\bm{x}_1=-\bm{B}_{1}\bm{x}_{0}+\bm{b}_1,~\mathfrak{V}_{2}\ge \bm{x}^{\top}_1\bm{g}_{2,k}+\beta_{2,k},~\forall k\in\bm{\mathcal{K}}^{\rm }_1 \right\}$
		\end{algorithmic}
	\end{algorithm}
\end{figure}

\subsubsection{Forward Step}
\label{app:forward}
The forward step provides a trial solution to the backward step and evaluates the upper bound of the optimal value.
The forward step solves the subproblem forward from the initial step after sampling $M$ realizations of the stochastic process, $\{\bm{\xi}^m_2,...\bm{\xi}^m_S\}_{m=1}^M$.
Using the {set of Benders cuts} computed in the backward step, the subproblem at stage $s$ is rewritten as
\begin{subequations}
	\label{eq:forward}
	\begin{flalign}
		\label{eq:HJB_Forward}
		&\overline{\bm{x}}_s^m=\argmin_{\bm{x}_s\ge0, \mathfrak{V}_{s+1}\in\mathbb{R}, \bm{A}^m_s\bm{x}_s=-\bm{B}^m_s\overline{\bm{x}}_{s-1}^m+\bm{b}^m_s}(\bm{c}^{m}_s)^{\top}\bm{x}_s + \mathfrak{V}_{s+1}&\forall s,m=1,...,M,\\
		\label{eq:BendersCutting0_Forward}
		&\mathfrak{V}_{s+1}\ge\bm{g}_{s+1,k}^{\top}\bm{x}_{s}+\beta_{s+1,k}{\rm~where~}\mathfrak{V}_{S+1}=0&\forall s,k\in\bm{\mathcal{K}}_s.
	\end{flalign}
\end{subequations}
A sequence $\{\overline{\bm{x}}_s^m\}_{s=1,...,S}$ obtained from Eq. (\ref{eq:forward}) is used as trial solutions by the backward step stated in Section \ref{backward}.
From the tower property and strong monotonicity of expectation, Eq. (\ref{eq:MSSP_objective}) is equivalent to the end-of-horizon formulation:
\begin{flalign}
	\label{eq:end-of-horizon_expectation}
	&\min_{\{\bm{x}_1,...,\bm{x}_S\}}\mathbb{E}\left[\bm{c}_1^{\top}\bm{x}_1+\bm{c}_2^{\top}\bm{x}_2+\cdots+\bm{c}_S^{\top}\bm{x}_S\right].&
\end{flalign}
Thus, an unbiased estimator of an upper bound for the optimal value of Eq. (\ref{eq:MSSP}) and its variance are given by
\begin{flalign}
	&\hat{\Psi}_{\mu} = \frac{1}{M}\sum_{m=1}^M\Psi^m,&\\
	&\hat{\Psi}^2_{\sigma} = \frac{1}{M-1}\sum_{m=1}^M(\Psi^m-\hat{\Psi}_{\mu})^2,&\\
	&\Psi^m=\bm{c}_1^{\top}\bm{x}_1+\sum_{s=2}^S(\bm{c}_s^m)^{\top}\overline{\bm{x}}^m_s&\forall m.
\end{flalign}
Based on the Central Limit Theorem, $\hat{\Psi}_{\mu}$ has approximately a normal distribution for a large $M$.
The pseudo-code for the forward step is shown in Algorithm \ref{alg:Forward}.

\begin{figure}[!t]
	\begin{algorithm}[H]
		\caption{Forward pass of the SDDP algorithm for $M$ sample paths}
		\label{alg:Forward}
		\begin{algorithmic}[]
			\setlength{\itemsep}{0pt} \setlength{\parskip}{0pt}
			\REQUIRE Finite lower bounds on $\{\mathfrak{V}_s\}_{s=2,...,S}$
			\FOR{$m=1\rightarrow M$}
			\STATE Solve 1st stage problem:
			\STATE $\overline{\bm{x}}_1=\argmin_{\bm{x}_1,\mathfrak{V}_2\in\mathbb{R}}\{\bm{c}_1^{\top}\bm{x}_1+\mathfrak{V}_{2}|\bm{A}_1\bm{x}_1=-\bm{B}_{1}\bm{x}_{0}+\bm{b}_1,~\mathfrak{V}_{2}\ge \bm{g}_{2,k}^{\top}\bm{x}_1+\beta_{2,k},~\forall k\in\bm{\mathcal{K}}^{}_1\}$
			\STATE Generate sample path $m$: $\{\bm{\xi}_2^m,...,\bm{\xi}_S^{m}\}$
			\FOR{$s=2\rightarrow S$}
			\STATE Solve $s$-th stage problem ($\mathfrak{V}_{S+1}=0$):
			\STATE $\overline{\bm{x}}_s^m=\argmin_{\bm{x}_s,\mathfrak{V}_{s+1}\in\mathbb{R}}\{\bm{c}_s^{\top}\bm{x}_s+\mathfrak{V}_{s+1}|\bm{A}^m_s\bm{x}_s=-\bm{B}^m_s\overline{\bm{x}}_{s-1}+\bm{b}^m_s,~\mathfrak{V}_{s+1}\ge \bm{g}_{s+1,k}^{\top}\bm{x}_s+\beta_{s+1,k},~\forall k\in\bm{\mathcal{K}}^{\rm }_s\}$
			\ENDFOR
			\STATE Compute the objective value:
			\STATE ${\Psi}^m\leftarrow \bm{c}_1^{\top}\overline{\bm{x}}_1+\sum_{s=2}^S(\bm{c}_s^m)^{\top}\overline{\bm{x}}_s^m$
			\ENDFOR
			\STATE Compute the sample average $\hat{\Psi}_{\mu}$ and SD $\hat{\Psi}_{\sigma}$:
			\STATE $\hat{\Psi}_{\mu}=\frac{1}{M}\sum_{m=1}^M{\Psi}^m$ and $\hat{\Psi}^2_{\sigma}=\frac{1}{M-1}\sum_{m=1}^M({\Psi}^m-\mu)^2$
			\STATE Update the upper bound: $\overline{\Psi}\leftarrow\hat{\Psi}_{\mu}+z_{\alpha}\hat{\Psi}_{\sigma}/\sqrt{M}$, where $z_{\alpha}$ is $(1-\alpha)$-quantile of the standard normal distribution for $\alpha\in(0,1)$.
		\end{algorithmic}
	\end{algorithm}
\end{figure}

%% file: chxx.tex
\subsection{Extensions of SDDP}
\label{app:sddp_extensions}

{This subsection reviews several extensions of SDDP that relax the standard assumptions stated in Appendix \ref{app:sddp}.
	While many algorithmic variants have been proposed to address modeling and computational challenges, we focus on convergence analyses for MSSPs that relax (i) Assumption \ref{app:assumption_linearity} (Linearity), (ii) Assumption \ref{app:assumption_risk_neutral} (Risk-neutral policy), and (iii) Assumption \ref{app:assumption_continuous} (Continuous variables).
	In particular, we summarize the convergence results in \citet{girardeau2015convergence,guigues2016convergence,zou2019stochastic,dowson2025incorporating}.
	Other extensions are beyond our scope; we refer the reader to the comprehensive review by \citet{fullner2025stochastic}.}

{
\citet{girardeau2015convergence} provide an early convergence analysis of SDDP for multi-stage stochastic convex problems, thereby relaxing Assumption \ref{app:assumption_linearity} (Linearity).
To keep the exposition concrete, consider a convex MSSP in which stage $s$ decisions $(\bm{x}_s,\bm{u}_s)$ are chosen after observing $\bm{\xi}_s$, with convex stage costs and convex constraints:
\begin{subequations}
	\label{eq:MSCP}
	\begin{flalign}
		&V_{s}(\bm{x}_{s-1},{\bm{\xi}}_s)=\min_{\bm{x}_s,\bm{u}_s}\left[C_s(\bm{x}_s,\bm{u}_s,\bm{\xi}_s)+\mathcal{V}_{s+1}(\bm{x}_s)\right]&\forall s,\\
		&{\rm~s.t.~~}\bm{x}_s=T_s(\bm{x}_{s-1}, \bm{u}_s, \bm{\xi}_s)&\forall s,\\
		&~~~~~~~~\bm{g}_s(\bm{x}_s,\bm{u}_s,\bm{\xi}_s)\le\bm{0}&\forall s,
	\end{flalign}
\end{subequations}
where $C_s(\bm{x}_s,\bm{u}_s,\bm{\xi}_s)$ is convex, $T_s(\bm{x}_{s-1}, \bm{u}_s, \bm{\xi}_s)$ is linear, and $\bm{g}_s(\bm{x}_s,\bm{u}_s,\bm{\xi}_s)$ is convex. $\mathcal{V}_{s+1}(\bm{x}_{s})=\mathbb{E}_{{\bm{\xi}}_{s}}\left[V_{s+1}(\bm{x}_{s},{\bm{\xi}}_{s+1})\right]$ is the value function.
Although the value functions $\mathcal{V}_s(\cdot)$ are generally no longer polyhedral, they remain convex, and SDDP can still build valid lower approximations using supporting cuts obtained from subgradients.
A key technical requirement is that bounded subgradients exist.
To ensure this property and that dual information can be used to compute supporting cuts, the standard RCR (Assumption \ref{app:assumption_RCR}) is strengthened as follows:
\begin{ass} (ERCR: Extended relatively complete recourse)
	\label{app:assumption_ERCR}
	Let $\mathcal{L}$ denote the set of indices so that $\bm{g}_s$ is nonlinear.
	For all feasible $\bm{x}_{s-1}$ and all realizations of $\bm{\xi}_{s}$, $\hat{\mathcal{X}}_{s}(\bm{x}_{s-1}, \bm{\xi}_s):=\{(\bm{x}_s,\bm{u}_s)|\bm{x}_{s}=T_s(\bm{x}_{s-1}, \bm{u}_s, \bm{\xi}_s),~g^l_s(\bm{x}_s,\bm{u}_s,\bm{\xi}_s)<0 ~~~\forall l\in \mathcal{L},~g^l_s(\bm{x}_s,\bm{u}_s,\bm{\xi}_s)\le0~~~\forall l\notin \mathcal{L}\}$ is non-empty for all stages $s\in\mathcal{S}$.
\end{ass}
Intuitively, this strengthened recourse assumption can be interpreted as a Slater-type strict feasibility requirement for the nonlinear constraints.
Under this additional assumption, the authors show that an $\varepsilon$-optimal policy can be obtained almost surely after a finite number of iterations for some predefined $\varepsilon>0$; see Theorem 3.1 in \citet{girardeau2015convergence}.
}

{
\citet{guigues2016convergence} generalizes the convergence analysis of \citet{girardeau2015convergence} to risk-averse settings, thereby relaxing Assumption \ref{app:assumption_risk_neutral} (Risk-neutral policy).
Specifically, the author analyzes the convergence of SDDP for convex MSSPs with (conditional) coherent risk measures, a special case of convex risk measures.
The central technique enabling cut construction is the dual representation of coherent risk measures on finite supports given by
\begin{flalign}
	\nonumber
	&\mathbb{F}_s[V_s(\bm{x}_{s-1},\bm{\xi}_s)] = \sup_{\bm{q}_s\in\mathcal{M}(\bm{p}_s)\subseteq\mathcal{P}} \sum_{n=1}^{N_s} q_s^n V_s(\bm{x}_{s-1},\bm{\xi}^n_s),~~~{\rm where~~}\mathcal{P}:=\left\{\bm{q}\in\mathbb{R}^{N_s}\Bigl|\sum_{n=1}^{N_s} q_s^n = 1, {\rm and~} q_s^n\ge0\right\},&
\end{flalign}
where $\mathcal{M}(\bm{p}_s)$ is a convex risk set, and $\bm{q}_s$ denotes risk-adjusted probabilities.
This representation implies that, instead of weighting scenario-wise subgradients by the nominal probabilities $p_s^n$, one can generate cuts for the risk-adjusted value function by weighting them with the optimal risk-adjusted probabilities $\bm{q}_s$. 
Moreover, because coherent risk measures satisfy monotonicity and convexity (Axioms 1 and 3 in Definition \ref{def:convex_risk}), the risk-adjusted value function $\mathbb{F}_s[V_s(\bm{x}_{s-1},\bm{\xi}_s)]$ remains convex.
Following a similar procedure as in \citet{girardeau2015convergence}, the author proves almost sure convergence of SDDP-type algorithms; see Theorem 4.1 in \citet{guigues2016convergence}.
}

{
\citet{zou2019stochastic} provide the first convergence analysis of an SDDP-type algorithm for MSSPs with integer state variables, thereby relaxing Assumption \ref{app:assumption_continuous} (Continuous variables).
In general, introducing integrality destroys convexity of the value functions, which prevents the direct use of supporting cuts.
The authors address this difficulty by focusing on binary state spaces and proving that the value function with respect to binary states can be represented exactly by a convex polyhedral function on the continuous relaxation domain (Theorem~1 in \citealp{zou2019stochastic}).
Since bounded integer variables can be equivalently expressed by a set of binary variables, their framework extends to general bounded integer states.
Building on this property and the strong duality of Lagrangian relaxation
(Theorem 1, \citealp{geoffrion1974lagrangean}),
\citet{zou2019stochastic} propose an SDDP-type algorithm, called SDDiP, that generates cuts from Lagrangian relaxations.
Furthermore, the almost sure finite convergence of the algorithm was proved (see Theorem 3 in \citealp{zou2019stochastic}) under the strengthened RCR assumption given by
\begin{ass} (CCR: Complete continuous recourse)
	\label{app:assumption_CCR}
	Let $\bm{u}_s^{\rm I}$ denote integer decision variables and $\bm{u}_s^{\rm C}$ denote continuous decision variables.
	For all feasible $(\bm{x}_{s},\bm{x}_{s-1},\bm{u}_s^{\rm I})$, and all realizations of $\bm{\xi}_{s}$, ${U}_s^{\rm C}(\bm{x}_{s},\bm{x}_{s-1},\bm{u}_s^{\rm I},\bm{\xi}_s):=\{\bm{u}_s^{\rm C}|(\bm{x}_{s},\bm{u}_s^{\rm I},\bm{u}_s^{\rm C})\in\mathcal{X}_s(\bm{x}_{s-1},\bm{\xi}_s)\}$ is non-empty for all stages $s\in\mathcal{S}$.
\end{ass}
Informally, complete continuous recourse requires that, for all feasible integer state and decision variables, one can always select continuous decisions to satisfy the stage constraints.
\citet{zou2019stochastic} further remark that the convergence results can be extended to settings with nonlinear convex objective functions and constraint sets.
}

{
\citet{dowson2025incorporating} further extend \citet{guigues2016convergence} from coherent risk measures to the broader class of convex risk measures.
A convex conditional risk measure admits a dual representation of the form \citep{follmer2002convex}
\begin{flalign}
	\nonumber
	&\mathbb{F}_s[V_s(\bm{x}_{s-1},\bm{\xi}_s)] = \sup_{\bm{q}_s\in\mathcal{M}(\bm{p}_s)\subseteq\mathcal{P}} \left[\sum_{n=1}^{N_s} q_s^n V_s(\bm{x}_{s-1},\bm{\xi}^n_s)-\alpha_s(\bm{q}_s)\right],&
\end{flalign}
where $\alpha_s(\cdot)$ is a lower semicontinuous convex penalty function.
The authors note that the convergence arguments used in the coherent case do not rely on the more restrictive axioms of coherent risk measures, relative to those of convex risk measures.
Accordingly, convergence results established for coherent risk measures such as in \citet{guigues2016convergence} can be extended to convex risk measures; see Remark 10 in \citet{dowson2025incorporating}.
Furthermore, the authors remark that risk-averse SDDP can be extended to the MSSPs with discrete variables by replacing gradient-based cut computations with the Lagrangian cuts in \citet{zou2019stochastic}; see Remark 11 in \citet{dowson2025incorporating}. 
}

{The above literature indicates that, by strengthening the standard RCR assumption  (Assumption \ref{app:assumption_RCR}), 
	the standard assumptions stated in Appendix \ref{app:sddp} can be relaxed as follows:
}
\begin{ass} (Convexity)
	\label{app:assumption_convexity}
	Given fixed $\bm{\xi}_s$, {the cost function} $C_s$ is { lower semicontinuous, proper, and} convex, {the transition function} $T_s$ is {linear}, and {the action space} $U_s$ is a {non-empty compact mixed integer} convex set for all stages $s\in\mathcal{S}$,
\end{ass}
\begin{ass} (Risk-averse policy)
	\label{app:assumption_convex}
	The conditional risk measure $\mathbb{F}^{}_{\bm{\xi}_{s}}[\cdot]$ is convex for all stages $s\in\mathcal{S}$,
\end{ass}
\begin{ass} (Integer variables)
	\label{app:assumption_integer}
	When integer variables are included, all the $\bm{x}_s$ are integer variables for all stages $s\in\mathcal{S}$.
\end{ass}

%% file: chxxx.tex
\subsection{Description of the Dual SDDP algorithm}
\label{app:dual_sddp}

This subsection reviews the Dual SDDP proposed by \citet{guigues2023duality}.
Dual SDDP applies an SDDP-type iteration scheme to the dynamic programming equations of the dual of a multi-stage stochastic linear problem.
Suppose Assumptions \ref{app:assumption_linearity}--\ref{app:assumption_RCR} hold,
we consider the end-of-horizon formulation of Eq. (\ref{eq:MSSP}) given by
\begin{subequations}
	\label{eq:MSSP_end-of_horizon}
	\begin{flalign}
		&\min_{\{\bm{x}_1,...,\bm{x}_S\}}\mathbb{E}\left[\bm{c}_1^{\top}\bm{x}_1+\bm{c}_2^{\top}\bm{x}_2+\cdots+\bm{c}_S^{\top}\bm{x}_S\right],&\label{eq:MSSP_objective_end-of_horizon}\\
		&{\rm s.t.~~~}\bm{A}_s\bm{x}_s+\bm{B}_s\bm{x}_{s-1}=\bm{b}_s,~{\rm where~}\bm{x}_0 {\rm~and~} \bm{\xi}_1{\rm~is~given}&\forall s. \label{eq:MSSP_transition_end-of_horizon}
	\end{flalign}
\end{subequations}
Dualization of Eq. (\ref{eq:MSSP_transition_end-of_horizon}) leads to the following dual of problem (see \citet{shapiro2009lectures}, Section 3.2.3):
\begin{subequations}
	\label{eq:Dual_MSSP_end-of_horizon}
	\begin{flalign}
		&\max_{\{\bm{\pi}_1,...,\bm{\pi}_S\}}\mathbb{E}\left[\left(\bm{b}_1-\bm{B}_1\bm{x}_0\right)^{\top}\bm{\pi}_1+(\bm{b}_2)^{\top}\bm{\pi}_2+\cdots+\bm{b}_S^{\top}\bm{\pi}_S\right],&\label{eq:Dual_MSSP_objective_end-of_horizon}\\
	&{\rm s.t.~~~}\bm{A}^{\top}_{S}\bm{\pi}_{S} \le \bm{c}_{S},& \\
	&~~~~~~~~\bm{A}^{\top}_{s-1}\bm{\pi}_{s-1} + \mathbb{E}_{\bm{\xi_{s-1}}}[\bm{B}_s^{\top}\bm{\pi}_s]\le \bm{c}_{s-1},~{\rm where~}\bm{\xi}_1{\rm~is~given}&\forall s\in\{2,...,S\}, \label{eq:Dual_MSSP_transition_end-of_horizon}		
	\end{flalign}    
\end{subequations}
{where $\{\bm{\pi}_s\}_{s=1}^{S}$ is a dual variable vector corresponding to Eq.~\eqref{eq:MSSP_transition_end-of_horizon}.
$\bm{\pi}_s$ is a state variable vector connecting two consecutive stages $s-1$ and $s$.}
Using Assumptions \ref{app:assumption_finite_randomness} and \ref{app:assumption_independence}, we can write the dynamic programming equations of Eq. (\ref{eq:Dual_MSSP_end-of_horizon}) as follows:
\begin{subequations}
	\label{eq:Dual_DP}
	\begin{flalign} 
		\nonumber
		&\underline{1{\rm st~stage~subproblem}}&\\
		\label{eq:Dual_DP3}
		&\max_{\bm{\pi}_1}\left(\bm{b}_1-\bm{B}_1\bm{x}_0\right)^{\top}\bm{\pi}_1 + V_{2}(\bm{\pi}_1),&	
	\end{flalign}  
	\vspace{-4ex}
	\begin{flalign}
		\nonumber
		&\underline{s\mathchar`-{\rm th~stage~subproblem}}&\\
		\label{eq:Dual_DP2}
		&V_{s}(\bm{\pi}_{s-1}^{n'})=\max_{\bm{\pi}^1_s,...,\bm{\pi}^{N_s}_s}\sum_{n=1}^{N_s}p_s^n\left[(\bm{b}^n_s)^{\top}\bm{\pi}^n_s + V_{s+1}(\bm{\pi}_s^n)\right]&\forall n'\in\{1,...,N_{s-1}\},\\
		&(\bm{A}^{n'}_{s-1})^{\top}\bm{\pi}^{n'}_{s-1} + \sum_{n=1}^{N_s}p_s^n[(\bm{B}^n_s)^{\top}\bm{\pi}^n_s]\le \bm{c}^{n'}_{s-1},~{\rm where~}\bm{\xi}_1{\rm~is~given}&\forall n'\in\{1,...,N_{s-1}\}, s\in\{2,...,S-1\},
	\end{flalign}
    \vspace{-4ex}
    \begin{flalign}
		\nonumber
		&\underline{S\mathchar`-{\rm th~stage~subproblem}}&\\
		\label{eq:Dual_DP1}
		&V_{S}(\bm{\pi}_{S-1}^{n'})=\max_{\bm{\pi}^1_S,...,\bm{\pi}^{N_S}_S}\sum_{n=1}^{N_S}p_S^n\left[(\bm{b}^n_S)^{\top}\bm{\pi}^n_S\right]&\forall n'\in\{1,...,N_{S-1}\},\\
		&(\bm{A}^n_{S})^{\top}\bm{\pi}^n_{S} \le \bm{c}^n_{S}&\forall n\in\{1,...,N_S\},\\
		&(\bm{A}^{n'}_{S-1})^{\top}\bm{\pi}^{n'}_{S-1} + \sum_{n=1}^{N_S}p_S^n[(\bm{B}^n_S)^{\top}\bm{\pi}^n_S]\le \bm{c}^{n'}_{S-1}&\forall n'\in\{1,...,N_{S-1}\},
	\end{flalign} 
\end{subequations}
{where $\bm{\pi}_s^n$ represent a dual variable vector corresponding to $\bm{\xi}_s^n$ and
$V_s(\bm{\pi}^n_{s-1})$ denotes the value function with respect to state $\bm{\pi}^n_{s-1}$.}

Since each subproblem is linear, the value function is guaranteed to be piecewise concave.
Therefore, as in the primal problem, the upper bound of the value function $V_{s}$ can be tightened by iteratively adding Benders cuts.
As the set of Benders cuts becomes richer with each iteration, the upper bound is a non-increasing function of the iteration count {(Proposition 4.2 in \citet{guigues2023duality})}.

By comparing with Eq. (\ref{eq:MSSP_DP_SSA}), three main observations can be made.
First, the 1-st stage subproblem (\ref{eq:Dual_DP3}) is not bounded.
\citet{guigues2023duality} show that, if the primal problem \eqref{eq:MSSP_end-of_horizon} has a finite optimal value, then there exist finite vectors $\underline{\bm{\pi}}_s,\overline{\bm{\pi}}_s$ such that adding box constraints $\underline{\bm{\pi}}_s\le \bm{\pi}_s^n\le \overline{\bm{\pi}}_s$ does not change the dual optimal value 
{
(Proposition 4.1 in \citet{guigues2023duality}).
Second, even when the problem \eqref{eq:MSSP_end-of_horizon} satisfies the RCR assumption (Assumption \ref{app:assumption_RCR}), the dual DP \eqref{eq:Dual_MSSP_end-of_horizon} may violate the RCR assumption.
To overcome this violation, nonnegative slack variables are introduced in Eq. \eqref{eq:Dual_MSSP_transition_end-of_horizon} 
and penalized in the objective.
Following Theorem 4.1 in \citet{guigues2023duality}), by increasing the penalty coefficient of the objective function with iteration, the deterministic upper bounds obtained from the Dual SDDP algorithm almost surely converge to the optimal value.}
Third, because the subproblems (\ref{eq:Dual_DP2}) and (\ref{eq:Dual_DP1}) do not decompose into separate problems with respect to each $\bm{\pi}^n_s$, the forward pass differs from Algorithm~\ref{alg:Forward}.
At each stage $s$, Dual SDDP solves one linear program in variables $(\bm{\pi}_s^1,\ldots,\bm{\pi}_s^{N_s})$,
and then selects the component corresponding to the sampled realization to define the next state.
Algorithm~\ref{alg:dual_Forward} gives the forward pass in the Dual SDDP algorithm.

\begin{figure}[!t]
	\begin{algorithm}[H]
		\caption{Forward pass of the Dual SDDP algorithm}
		\label{alg:dual_Forward}
		\begin{algorithmic}[]
			\setlength{\itemsep}{0pt} \setlength{\parskip}{0pt}
			\REQUIRE Finite upper bounds on $\{\mathfrak{V}_s\}_{s=2,...,S}$
			\FOR{$m=1\rightarrow M$}
			\STATE Solve 1st stage problem:
			\STATE $\overline{\bm{\pi}}_1=\argmax_{\bm{\pi}_1\in[\underline{\pi},\overline{\pi}],\mathfrak{V}_2\in\mathbb{R}}\{\bm{b}_1^{\top}\bm{\pi}_1+\mathfrak{V}_{2}|\mathfrak{V}_{2}\le \bm{g}_{2,k}^{\top}\bm{\pi}_1+\beta_{2,k},~\forall k\in\bm{\mathcal{K}}^{}_1\}$
			\FOR{$s=2\rightarrow S$}
			\STATE Solve $s$-th stage problem ($\mathfrak{V}_{S+1}=0$):
			\STATE $\{\overline{\bm{\pi}}_s^1,...,\overline{\bm{\pi}}_s^{N_s}\}=\argmax_{\{\bm{\pi}^1_s,...,\bm{\pi}^{N_s}_s\},\{\mathfrak{V}^1_{s+1},...,\mathfrak{V}^{N_s}_{s+1}\}\in\mathbb{R}^{N_s}}\{\sum_{n=1}^{N_s}p^n_s\left[(\bm{b}^n_s)^{\top}\bm{\pi}^n_s+\mathfrak{V}^n_{s+1}\right]$
			\STATE $|{\bm{A}^{n'}_{s-1}}^{\top}\overline{\bm{\pi}}^{n'}_{s-1}+\sum_{n=1}^{N_s}p_s^n\left[(\bm{B}^n_s)^{\top}{\bm{\pi}}^n_{s}\right]\le\bm{c}^{n'}_{s-1},~\mathfrak{V}^n_{s+1}\le \bm{g}_{s+1,k}^{\top}\bm{\pi}^n_s+\beta_{s+1,k},~\forall n,k\in\bm{\mathcal{K}}^{\rm }_s\}$
			\STATE $m$-th sampling $\overline{\pi}^{n'}_s $from $\{\overline{\bm{\pi}}_s^1,...,\overline{\bm{\pi}}_s^{N_s}\}$
			\ENDFOR
			\ENDFOR
		\end{algorithmic}
	\end{algorithm}
\end{figure}

%% file: chxxxx.tex
{
\section{The Proof of Propositions \ref{prop:suff}--\ref{prop:suff3}}
\label{app:prop}

\subsection{The Proof of Proposition \ref{prop:suff}: Almost sure convergence of lower bounds}
\label{app:prop1}

Standard convergence analyses of the SDDP-type algorithms are stated under the RCR-type assumption (Assumption \ref{app:assumption_RCR}--\ref{app:assumption_CCR}), as seen in Appendix \ref{app:sddp0}.
In this paper, we relax Assumption \ref{app:assumption_RCR}--\ref{app:assumption_CCR} by constructing an exact-penalty reformulation, where the problematic {constraints} that may violate the RCR-type infeasibility are {relaxed} and penalized in the objective.
{
Specifically, we prove that the SDDP-type algorithm provides a lower bound to almost surely converge, even when the feasible region for continuous decision variables is empty for all feasible integer variables.}
We then show that, under Conditions 1--6 in Proposition \ref{prop:suff}, a finite penalty coefficient exists so that the penalized problem has the same optimal value and the set of optimal solutions as the original one.
This implies that even without imposing the RCR-type assumption on the original problem, the SDDP-type algorithm can provide a feasible policy that yields an equivalent optimal value via the penalized problem.
}

{
{To keep the exposition concrete, we consider a multi-stage stochastic problem of the form
\begin{subequations}
	\label{eq:MSCiP}
	\begin{flalign}
		&\min_{\bm{x}_1,\bm{u}^{\rm C}_1,\bm{u}^{\rm I}_1}C_1(\bm{x}_1,\bm{u}^{\rm C}_1,\bm{u}^{\rm I}_1,\bm{\xi}_1)+\mathbb{F}_{\bm{\xi}_1}\left[\min_{\bm{x}_2, \bm{u}^{\rm C}_2,\bm{u}^{\rm I}_2}C_2(\bm{x}_2,\bm{u}^{\rm C}_2,\bm{u}^{\rm I}_2,\bm{\xi}_2)+\mathbb{F}_{\bm{\xi}_{2}}\left[\cdots+\mathbb{F}_{\bm{\xi}_{S-1}}\left[\min_{\bm{x}_S,\bm{u}^{\rm C}_S,\bm{u}^{\rm I}_S}C_S(\bm{x}_S,\bm{u}^{\rm C}_S,\bm{u}^{\rm I}_S,\bm{\xi}_S)\right]\right]\right],&
	\end{flalign}
	\vspace{-7ex}
	\begin{flalign}
		&{\rm~s.t.~~}\bm{x}_s=T_s(\bm{x}_{s-1}, \bm{u}^{\rm C}_s,\bm{u}^{\rm I}_s, \bm{\xi}_s)~~~~\label{eq:MSCiP_eq}&\forall s,\\
		&~~~~~~~~\bm{g}_s(\bm{x}_s,\bm{u}^{\rm C}_s,\bm{u}^{\rm I}_s,\bm{\xi}_s)\le\bm{0}~~~~~~~~~~\label{eq:MSCiP_ineq}&\forall s,
	\end{flalign}
\end{subequations}
where $C_s(\bm{x}_s,\bm{u}^{\rm C}_s,\bm{u}^{\rm I}_s,\bm{\xi}_s)$ is lower semicontinuous, proper, convex cost function,
$T_s(\bm{x}_{s-1}, \bm{u}^{\rm C}_s,\bm{u}^{\rm I}_s, \bm{\xi}_s)$ is linear transition function,
and $\bm{g}_s(\bm{x}_s,\bm{u}^{\rm C}_s,\bm{u}^{\rm I}_s,\bm{\xi}_s)$ is convex lower semicontinuous function. 
$\bm{u}^{\rm C}_s $ and $\bm{u}^{\rm I}_s$ are continuous and discrete decision variables, respectively.
To avoid notational complexity, we sometimes write decision variables as $\bm{u}_s=(\bm{u}^{\rm C}_s,\bm{u}^{\rm I}_s)$. 
}

As a first step, we prove a risk-neutral version of Proposition \ref{prop:suff}, that is
\begin{prop}
	\label{prop:suff0}
	The sufficient conditions for the lower bound of the value function {$V_1$} in Eqs. (\ref{eq:initial_condition})--(\ref{eq:terminal_condition}) to converge almost surely and non-decreasingly to the global optimum using the SDDP algorithm in a finite number of iterations are as follows:
	\begin{enumerate}
		\setcounter{enumi}{9}
		\item The conditional risk measure $\mathbb{F}^{}_{\bm{\xi}_{s}}$ is $\mathbb{E}^{}_{\bm{\xi}_{s}}$ for all stages $s\in\mathcal{S}$, and 
	\end{enumerate}
	Conditions 1, 3--6 in Proposition \ref{prop:suff}, where superscripts are dropped as the same conditions are required for strategic and operational subproblems.
\end{prop}

\begin{proof} [Proof of Proposition \ref{prop:suff0}]
{
    Given Conditions 1--6, 
    we consider the penalized problem of a risk-neutral version of Eq.~\eqref{eq:MSCiP} as follows:
    \begin{subequations}
    	\label{eq:MSCiP_slack}
    	\begin{flalign}
        \nonumber
    		&\min_{\bm{x}_1,\bm{u}^{\rm C}_1,\bm{u}^{\rm I}_1,\bm{\zeta}^{\pm}_1}C^\lambda_1(\bm{x}_1,\bm{u}^{\rm C}_1,\bm{u}^{\rm I}_1,\bm{\xi}_1)+\mathbb{E}_{\bm{\xi}_1}\Biggl[\min_{\bm{x}_2, \bm{u}^{\rm C}_2,\bm{u}^{\rm I}_2,\bm{\zeta}^{\pm}_2}C^\lambda_2(\bm{x}_2,\bm{u}^{\rm C}_2,\bm{u}^{\rm I}_2,\bm{\xi}_2)&\\
            &~~~~~~~~~~~~~~~~~~~~~~~~~~~~~~~~~~~~~~~~~~~~~~~~~~~+\mathbb{E}_{\bm{\xi}_{2}}\left[\cdots+\mathbb{E}_{\bm{\xi}_{S-1}}\left[\min_{\bm{x}_S,\bm{u}^{\rm C}_S,\bm{u}^{\rm I}_S,\bm{\zeta}^{\pm}_S}C^\lambda_S(\bm{x}_S,\bm{u}^{\rm C}_S,\bm{u}^{\rm I}_S,\bm{\xi}_S)\right]\right]\Biggr],&
        \end{flalign}
		\vspace{-2ex}
		\begin{flalign}
            &{\rm~s.t.~~}C^\lambda_s(\bm{x}_s,\bm{u}^{\rm C}_s,\bm{u}^{\rm I}_s,\bm{\xi}_s):=C_s(\bm{x}_s,\bm{u}^{\rm C}_s,\bm{u}^{\rm I}_s,\bm{\xi}_s)+\bm{\lambda}^{\top}(\bm{\zeta}_s^{+} + \bm{\zeta}_s^{-})~~~~~~~~~~~~~~~~~~~~~~~~~~~~~~~~~~~~~~~~~~~~~~~~~~~~\label{eq:MSCiP_slack_cost}&\forall s,\\
    		&~~~~~~~~\bm{x}_s=T_s(\bm{x}_{s-1}, \bm{u}^{\rm C}_s,\bm{u}^{\rm I}_s, \bm{\xi}_s)+\bm{\zeta}_{s,{\rm x}}^{+}-\bm{\zeta}_{s,{\rm x}}^{-}~~~~~~~~~~~~~~~~~~~~~~~~~~~~~~~~~~~~~~~~~~~~~~~~~~~~~~~~~~~~~~~~~~~~~~~~~~~~~~~~\label{eq:MSCiP_slack_eq}&\forall s,\\
    		&~~~~~~~~\bm{g}_s(\bm{x}_s,\bm{u}^{\rm C}_s,\bm{u}^{\rm I}_s,\bm{\xi}_s)\le\bm{\zeta}_{s,{\rm u}}^{+}~~~~~~~~~~~~~~~~~~~~~~~~~~~~~~~~~~~~~~~~~~~~~~~~~~~~~~~~~~~~~~~~~~~~~~~~~~~~~~~~~~~~~~~~~~~~~~~~~~~~~~\label{eq:MSCiP_slack_ineq}&\forall s,
    	\end{flalign}
    \end{subequations}
    where $\bm{\lambda}$ is a predefined penalty coefficient vector.
    $\bm{\zeta}_{s,\rm x}^{\pm}$ and $\bm{\zeta}_{s,\rm u}^{+}$ are non-negative slack variable vectors for Eqs.~\eqref{eq:MSCiP_eq} and \eqref{eq:MSCiP_ineq}.
    To avoid notational complexity, we sometimes write slack variables as $\bm{\zeta}^{\pm}_s=(\bm{\zeta}^{\pm}_{s,{\rm x}},\bm{\zeta}^{+}_{s,{\rm u}})$. 
    Since all the constraints are relaxed, 
    the penalized problem \eqref{eq:MSCiP_slack} is feasible and thus satisfies the RCR-type infeasibility.
    The penalized problem \eqref{eq:MSCiP_slack} consists of a stage cost given by $C_s^\lambda$, a transition function given by $T_s+\bm{\zeta}_{s,{\rm x}}^{+}-\bm{\zeta}_{s,{\rm x}}^{-}$, a nonlinear function given by $\bm{g}_s-\bm{\zeta}_{s,{\rm u}}^{+}$, and additional continuous decision variables $\bm{\zeta}_{s}^{\pm}$.
    Therefore, as long as $\bm{\zeta}_{s}^{\pm}$ is defined on a compact space\footnote{
        The compact action space and the RCR-type assumption are both satisfied when each $\bm{\zeta}_{s}^{\pm}$ is defined on $[0, \overline{M}]$ using a sufficiently large constant $\overline{M}$.
    }, Conditions 1 and 3 remain satisfied; consequently, 
    the SDDP-type algorithm provides a lower bound for the optimal value of the penalized problem \eqref{eq:MSCiP_slack} that converges almost surely.
    In the following, we prove that the optimal value and optimal policy of the penalized problem \eqref{eq:MSCiP_slack} coincide with those of the original problem \eqref{eq:MSCiP}.   }
      
	Since the realizations of $\bm{\xi}_s$ are finite and stagewise independent (Conditions 5 and 6) and
	the length of the planning horizon $S$ is finite (Condition 4),  
	by representing the sequence of possible realizations in the form of scenario tree,
	we can rewrite the original problem \eqref{eq:MSCiP} and penalized problem \eqref{eq:MSCiP_slack} as equivalent mixed integer convex problems, called deterministic equivalent\footnote{	The extensive form possesses an exorbitant number of variables and constraints. 
		This transformation is intended to theoretically prove the equivalence of the problems.
		SDDP does not solve the extensive form directly.}.
	Denoting $J$ as the number of scenarios, each scenario, numbered $j\in\{1,...,J\}$, is associated with a data sequence $\{\bm{\xi}_1^j,...,\bm{\xi}_S^j\}$ and its probability ${p}^j$ and the corresponding sequence of states $\bm{x}^j=\{\bm{x}^j_1,...,\bm{x}^j_S\}$, decisions $\bm{u}^j=\{\bm{u}^j_1,...,\bm{u}^j_S\}$, {and slacks $\bm{\zeta}^{j,\pm}=\{\bm{\zeta}^{j,\pm}_1,...,\bm{\zeta}^{j,\pm}_S\}$}.
	The deterministic equivalent of a risk-neutral version of Eq. \eqref{eq:MSCiP} can be written as follows:
    \begin{subequations}
		\label{eq:DE_MSCiP}
		\begin{flalign}
			&\min_{\{\bm{x}^j\}_{j=1}^J, \{\bm{u}^j\}_{j=1}^J}\sum_{j=1}^Jp^j\left[\sum_{s=1}^S C_s(\bm{x}^j_s, \bm{u}^j_s, \bm{\xi}^j_s)\right]&\\
			&{\rm~s.t.~~}\bm{x}^j_s=T_s(\bm{x}^{j}_{s-1}, \bm{u}^j_s, \bm{\xi}^j_s)~~~~~~~~~~~~~~~~~~~~\label{eq:DE_MSCiP_eq}&\forall s,j,\\
			&~~~~~~~~\bm{g}_s(\bm{x}^j_s,\bm{u}^j_s,\bm{\xi}^j_s)\le\bm{0}~~~~~~~~\label{eq:DE_MSCiP_ineq}&\forall s,j,\\
			&~~~~~~~~\bm{x}_s^j=\bm{x}_s^{j'}~~~~~~~~~~~~\label{eq:DE_MSCiP_ant_x}&\forall s,j,j' {\rm~such~that~} \bm{\xi}^j_{[s]}=\bm{\xi}^{j'}_{[s]},\\
			&~~~~~~~~\bm{u}_s^j=\bm{u}_s^{j'}~~~~~~~\label{eq:DE_MSCiP_ant_u}&\forall s,j,j' {\rm~such~that~} \bm{\xi}^j_{[s]}=\bm{\xi}^{j'}_{[s]},
		\end{flalign}
	\end{subequations}
    where Eqs.~\eqref{eq:DE_MSCiP_ant_x} and \eqref{eq:DE_MSCiP_ant_u} are called the nonanticipativity constraints.
    Next, the deterministic equivalent of the penalized problem \eqref{eq:MSCiP_slack} is given by
    \begin{subequations}
    \label{eq:DE_MSCiP_LR}
    	\begin{flalign}
            \label{eq:DE_MSCiP_eq_slack_objective}
    		&{\min_{\{\bm{x}^j\}_{j=1}^J, \{\bm{u}^j\}_{j=1}^J, \{\bm{\zeta}^{j,\pm}\}_{j=1}^J}\sum_{j=1}^Jp^j\left[\sum_{s=1}^S \left[C^\lambda_s(\bm{x}^j_s, \bm{u}^j_s, \bm{\xi}^j_s)\right]\right]} &\\
            \label{eq:DE_MSCiP_eq_slack}
            &{{\rm~s.t.~~}C^{\lambda}_s(\bm{x}^j_s, \bm{u}^{j}_s, \bm{\xi}^j_s):=C_s(\bm{x}^j_s, \bm{u}^j_s, \bm{\xi}^j_s)+\bm{\lambda}^{\top}(\bm{\zeta}_s^{j,+} + \bm{\zeta}_s^{j,-})~}~~~~~~~~~~~~~&\forall s,j,\\
            &{~~~~~~~~\bm{x}^j_s=T_s(\bm{x}^{j}_{s-1}, \bm{u}^j_s, \bm{\xi}^j_s)+\bm{\zeta}_{s,{\rm x}}^{j,+}-\bm{\zeta}_{s,{\rm x}}^{j,-}~~~~~~~~~~~~~~}\label{eq:DE_MSCiP_eq}&\forall s,j,\\
			&{~~~~~~~~\bm{g}_s(\bm{x}^j_s,\bm{u}^j_s,\bm{\xi}^j_s)\le\bm{\zeta}_{s,{\rm u}}^{j,+}~~~~~~~~~~~~~~~~~~~~}\label{eq:DE_MSCiP_ineq}&\forall s,j,
		\end{flalign}
		\vspace{-4ex}
		\begin{flalign}
            \nonumber
			&~~~~~~~~~~~~~~~~~~~~~~~~~~~~~~~~~~~~~~~~~~~~~~~~~~~~~~~~~~~~~~~~~~~~~~~~~~~~~~~~~~~~~~~~~~~~~~~~~~~~~~~~~~~~~~~~~~~~~~~~~~~~~~~~~~~~~~~~~~~~~~~{\rm Eqs.~\eqref{eq:DE_MSCiP_ant_x}~and~\eqref{eq:DE_MSCiP_ant_u}.}&
		\end{flalign}
    \end{subequations}
    
    According to Theorems 10 and 14 in \citet{lefebvre2024exact}\footnote{
        Theorems 10 and 14 in \citet{lefebvre2024exact} present more general conditions under which the optimal value and the set of optimal solutions for the mixed-integer nonlinear optimization problem coincide with those of the penalized problem.
        The conditions under which Eq.~\eqref{eq:DE_MSCiP} is equivalent to the penalized problem by Lagrangian relaxation using the penalty function $\Psi$ with a finite penalty coefficient are the three conditions already mentioned, in addition to $\psi(\cdot)=\|\cdot\|$ for some norm $\|\cdot\|$ (see Definition 2 and Assumptions 3--4 in \citet{lefebvre2024exact}).
        Eq.~\eqref{eq:DE_MSCiP_LR} is a penalized problem that adopts the $L_1$-norm as the penalty function (i.e., $\psi(\cdot)=\|\cdot\|_1$).
    },
    there exists a finite penalty coefficient vector $\bm{\lambda}$
    such that the optimal value and the set of optimal solutions of Eq.~\eqref{eq:DE_MSCiP_LR} coincide with those of Eq.~\eqref{eq:DE_MSCiP}
    when the following conditions hold:
    \begin{itemize}
        \setlength{\itemsep}{0pt} \setlength{\parskip}{0pt}
        \item [i)]{ The feasible region is non-empty and compact,}
        \item [ii)]{The objective function is convex, and }
        \item [iii)] {The following Slater condition holds.
            Let $\mathcal{L}$ denote the set of indices so that $\bm{g}_s$ is nonlinear.
    		For all feasible $(\{\bm{x}^j\}_{j=1}^J,\{\bm{u}^{{\rm I},j}\}_{j=1}^J)$, one of the two following conditions is met:
    		\begin{itemize}
    			\setlength{\itemsep}{0pt} \setlength{\parskip}{0pt}
    			\item [(A)] $U^{\rm C}:=\{\{\bm{u}^{{\rm C},j}\}_{j=1}^J|\bm{x}^j_{s}=T_s(\bm{x}^j_{s-1}, \bm{u}^j_s, \bm{\xi}^j_s)~\forall s,j,~ g^l_s(\bm{x}^j_s,\bm{u}^j_s,\bm{\xi}^j_s)<0~\forall s,j,l\in\mathcal{L},~{\rm and}~g^l_s(\bm{x}^j_s,\bm{u}^j_s,\bm{\xi}^j_s)\le0~\forall s,j,l\notin\mathcal{L}\}$ is non-empty, or
    			\item [(B)] $U^{\rm C}:=\{\{\bm{u}^{{\rm C},j}\}_{j=1}^J|\bm{x}^j_{s}=T_s(\bm{x}^j_{s-1}, \bm{u}^j_s, \bm{\xi}^j_s)~~\forall s,j,~ g_s(\bm{x}^j_s,\bm{u}^j_s,\bm{\xi}^j_s)\le0~~\forall s,j\}$ is empty.
    		\end{itemize}}
    \end{itemize}
    
    These three conditions correspond, respectively, to Assumptions 1, 2, and 5 in \citet{lefebvre2024exact}.
    Under Condition 1, the feasible region of the original problem \eqref{eq:MSCiP} is compact and its objective function is convex.
    Since the linear constraints \eqref{eq:DE_MSCiP_ant_x} and \eqref{eq:DE_MSCiP_ant_u} added to the deterministic equivalent preserve the compactness and convexity,
    the first two conditions are satisfied.
    For the last condition, (A) 
    is a strict feasibility requirement for the continuous restriction obtained by fixing the integer variables and guarantees the strong duality of the problem.
    Condition (B), by contrast, permits the continuous feasible set $U^{\rm C}$ to be empty for some feasible integer variables\footnote{
        RCR-type assumption would require the continuous feasible set $U^{\rm C}$ to be non-empty for every integer state and decision variables.
        The Slater-type condition is strictly weaker than RCR while
        still constituting an appropriate constraint qualification for the exact-penalty
        reformulation.
    }.
    Therefore, under the Slater condition, which constitutes an appropriate constraint qualification,
    the optimal values and optimal policies of Eqs. \eqref{eq:DE_MSCiP} and \eqref{eq:DE_MSCiP_LR} coincide.
    Finally, since Eqs. \eqref{eq:DE_MSCiP} and \eqref{eq:DE_MSCiP_LR} are deterministic equivalent to Eqs. \eqref{eq:MSCiP} and \eqref{eq:MSCiP_slack}, respectively,
    the optimal values and optimal policies of Eqs. \eqref{eq:MSCiP} and \eqref{eq:MSCiP_slack} also coincide\footnote{
        This equivalence means that the value function $V_1$ of the penalized MSSP provided by SDDP converges almost surely to the optimal value function of the original problem.
        Although the value functions $V_s$ for the dynamic programming equations of the penalized MSSP for other stages $s \ge 2$ are valid with respect to those of the original problem, they are not tight.
    }.
\end{proof} }

{
Let us move on to the proof of Proposition \ref{prop:suff}.
To avoid relying on the RCR-type assumptions \ref{app:assumption_RCR}--\ref{app:assumption_CCR}, we strengthen Assumption \ref{app:assumption_convex} as follows:
\begin{ass} (Consistent risk-averse policy)
	\label{app:assumption_consistency}
	$\mathbb{F}^{}_{\bm{\xi}_{s}}[\cdot]$ is conditionally consistent and convex for all stages $s\in\mathcal{S}$.
\end{ass}

A canonical example of a convex, conditionally consistent risk measure is the entropic risk measure, see Definition \ref{def:entropic_risk}.
Within the class of risk measures commonly adopted in the optimization setting for time-consistent decision-making (e.g., law invariance; see \citealp{homem2016risk}), the combination of convexity and conditional consistency is highly restrictive.
\citet{kupper2009representation} show that, in Theorem 1.10, these conditions force the risk measure to belong to the entropic class.
Expected conditional risk measures ({ECRMs}) can be alternatives that satisfy both convexity and conditional consistency, depending on the choice of conditional risk measures $\mathbb{F}_{\bm{\xi}_s}$.
However, explicit dynamic programming reformulations of MSSPs with ECRMs rely on adding continuous state variables; thus, they fail to satisfy Condition 3.
When integer variables are absent, MSSPs with ECRMs can be solved as if the problem were risk-neutral\footnote{\citet{khatami2024risk} defines the property of certain ECRMs that can be solved as if they were risk-neutral problems as risk-neutral convertible properties.}, as seen in Appendix \ref{app:prop3}.
Consequently, for the ECRM settings, the convergence property of the lower bound obtained through SDDP reduces to Proposition \ref{prop:suff0}.
This subsection proves Proposition \ref{prop:suff} for MSSPs equipped with entropic risk measures without significantly compromising generality for time-consistent optimization frameworks.

\begin{proof} [Proof of Proposition {\ref{prop:suff}}]
	Assume the conditional risk measure $\mathbb{F}_{\bm{\xi}_s}$ to be the entropic risk measure $\mathbb{ENT}_{\gamma}$.
    { Given Conditions 1--6, 
    we consider the penalized problem of the MSSP \eqref{eq:MSCiP} as follows:
    \begin{flalign}
    \nonumber
        &\min_{\bm{x}_1,\bm{u}_1,\bm{\zeta}^{\pm}_1}C^\lambda_1(\bm{x}_1,\bm{u}_1,\bm{\xi}_1)+\mathbb{ENT}_{\gamma}\Biggl[\min_{\bm{x}_2, \bm{u}_2,\bm{\zeta}^{\pm}_2}C^\lambda_2(\bm{x}_2,\bm{u}_2,\bm{\xi}_2)+\mathbb{ENT}_{\gamma}\left[\cdots+\mathbb{ENT}_{\gamma}\left[\min_{\bm{x}_S,\bm{u}_S,\bm{\zeta}^{\pm}_S}C^\lambda_S(\bm{x}_S,\bm{u}_S,\bm{\xi}_S)\right]\right]\Biggr],&\\
        \label{eq:MSCiP_ENT_slack}
        &{\rm~s.t.~~Eqs.~\eqref{eq:MSCiP_slack_cost}\mathchar`-\mathchar`-\eqref{eq:MSCiP_slack_ineq}}.&
    \end{flalign}
    As in the risk-neutral case, Conditions 1 and 3 hold as long as $\bm{\zeta}_{s}^{\pm}$ is defined on a compact space.
    The SDDP-type algorithm therefore provides a lower bound for the optimal value of the penalized problem \eqref{eq:MSCiP_ENT_slack} that converges almost surely.
    
    Following a procedure similar to that in the risk-neutral case,
    we prove that the optimal value and optimal policies of the deterministic equivalent of the penalized problem \eqref{eq:MSCiP_ENT_slack} coincide with those of the deterministic equivalent of the original problem \eqref{eq:MSCiP}.}
    From the conditional consistency of $\mathbb{ENT}_{\gamma}$, we can write the end-of-horizon formulation\footnote{For readability, detailed descriptions of measure theory are omitted.
    In the scenario tree representation adopted in this paper, the required measurability conditions are equivalently enforced by the nonanticipativity constraints in the deterministic equivalent.}
    of the MSSP \eqref{eq:MSCiP} as follows:
	\begin{flalign}
		\label{eq:MSCiP_ENT_EOH}
		&\min_{\{\bm{x}_s\}_{s=1}^S,\{\bm{u}_s\}_{s=1}^S}\mathbb{ENT}_{\gamma}\left[C_1(\bm{x}_1,\bm{u}_1,\bm{\xi}_1)+\cdots C_S(\bm{x}_S,\bm{u}_S,\bm{\xi}_S)\right],&\\
		\nonumber
		&{\rm~s.t.~~Eqs.~\eqref{eq:MSCiP_eq}\mathchar`-\mathchar`-\eqref{eq:MSCiP_ineq}}.
	\end{flalign}
    Similarly, the end-of-horizon formulation of the penalized problem \eqref{eq:MSCiP_ENT_slack} is given by
	\begin{flalign}
		\label{eq:MSCiP_ENT_EOH_slack}
		&\min_{\{\bm{x}_s\}_{s=1}^S,\{\bm{u}_s\}_{s=1}^S}\mathbb{ENT}_{\gamma}\left[C^\lambda_1(\bm{x}_1,\bm{u}_1,\bm{\xi}_1)+\cdots C^\lambda_S(\bm{x}_S,\bm{u}_S,\bm{\xi}_S)\right],&\\
		\nonumber
		&{\rm~s.t.~~Eqs.~\eqref{eq:MSCiP_slack_cost}\mathchar`-\mathchar`-\eqref{eq:MSCiP_slack_ineq}}.&
	\end{flalign}
    
	Under Conditions 4--6, using the definition of $\mathbb{ENT}_{\gamma}$, seen in Eq.~\eqref{eq:ENT}, 
	we can write the deterministic equivalent of Eq. \eqref{eq:MSCiP_ENT_EOH} as follows:
	\begin{flalign}
		\label{eq:DE_MSCiP_ENT}
		&\min_{\{\bm{x}^j\}_{j=1}^J, \{\bm{u}^j\}_{j=1}^J}\frac{1}{\gamma}\log\left[\sum_{j=1}^Jp^j\exp\left(\gamma\sum_{s=1}^SC_s(\bm{x}^j_s, \bm{u}^j_s, \bm{\xi}^j_s)\right) \right],&\\
		\nonumber
		&{\rm~s.t.~~Eqs.~\eqref{eq:DE_MSCiP_eq}\mathchar`-\mathchar`-\eqref{eq:DE_MSCiP_ant_u}}.
	\end{flalign}
    {Similarly, the deterministic equivalence of the penalized problem \eqref{eq:MSCiP_ENT_slack} is given by
    \begin{flalign}
		\label{eq:DE_MSCiP_ENT_slack}
		&\min_{\{\bm{x}^j\}_{j=1}^J, \{\bm{u}^j\}_{j=1}^J,\{\bm{\zeta}^{j,\pm}\}_{j=1}^J}\frac{1}{\gamma}\log\left[\sum_{j=1}^Jp^j\exp\left(\gamma\sum_{s=1}^S C^\lambda_s(\bm{x}^j_s, \bm{u}^j_s, \bm{\xi}^j_s)\right) \right],&\\
		\nonumber
		&{\rm~s.t.~~Eqs.~\eqref{eq:DE_MSCiP_ant_x},~\eqref{eq:DE_MSCiP_ant_u},~and~\eqref{eq:DE_MSCiP_eq_slack}\mathchar`-\mathchar`-\eqref{eq:DE_MSCiP_ineq}}.
	\end{flalign}
    
    The deterministic equivalents \eqref{eq:DE_MSCiP_ENT} and \eqref{eq:DE_MSCiP_ENT_slack} are formed by only replacing the objective functions of those for the risk-neutral case, Eqs.~\eqref{eq:DE_MSCiP} and \eqref{eq:DE_MSCiP_LR}, with the log-sum-exp form, which is convex.
    Since the three conditions required by Theorems 10 and 14 in \cite{lefebvre2024exact} are satisfied, under the Slater condition, the optimal values and optimal policies of Eqs. \eqref{eq:DE_MSCiP_ENT} and \eqref{eq:DE_MSCiP_ENT_slack} coincide. 
    Finally, since Eqs. \eqref{eq:DE_MSCiP_ENT} and \eqref{eq:DE_MSCiP_ENT_slack} are deterministic equivalents of Eqs. \eqref{eq:MSCiP} and \eqref{eq:MSCiP_ENT_slack}, respectively, the optimal values and optimal policies of Eqs. \eqref{eq:MSCiP} and \eqref{eq:MSCiP_ENT_slack} also coincide.}
\end{proof}
}

\subsection{The proof of Proposition \ref{prop:suff2}: Statistical upper bounds}
\label{app:prop2}

{
\begin{proof}
	We first prove that the risk-adjusted total cost in the end-of-horizon view when any policy is implemented serves as an upper bound for the optimal value.
	To ensure the feasibility of the policy induced by using the approximated value function at any iteration,
	we consider the penalized version of Eq.~\eqref{eq:MSCiP} with coefficient $\lambda$.
	Let $V^*$ and $V_\lambda^*$ denote the optimal values of Eq.~\eqref{eq:MSCiP} and its penalized version with $\lambda$, respectively.
	For any feasible policy $\Pi_\lambda$ of the penalized problem, we define the stagewise random cost $\{C_{s}(\Pi_\lambda)\}_{s=1}^S$, ,where $C_{s}(\Pi_\lambda):=C_s^\lambda$ and 
	the risk-adjusted total cost in terms of the nested and end-of-horizon view, respectively, as follows:
	\begin{flalign*}
		&\mathbb{F}^{\rm nest}_\lambda[C_{1}(\Pi_\lambda),...,C_{S}(\Pi_\lambda)]:=C_{1}(\Pi_\lambda)+\mathbb{F}_{\bm{\xi}_1}\left[C_{2}(\Pi_\lambda)+\mathbb{F}_{\bm{\xi}_{2}}\left[\cdots+\mathbb{F}_{\bm{\xi}_{S-1}}\left[C_{S}(\Pi_\lambda)\right]\right]\right]&{(\rm nested~risk)},\\
		&\mathbb{F}^{\rm EOH}_\lambda[C_{1}(\Pi_\lambda),...,,C_{S}(\Pi_\lambda)]:=\mathbb{F}\left[C_{1}(\Pi_\lambda)+\cdots+C_{S}(\Pi_\lambda)\right]&{(\rm end\mathchar`-of\mathchar`-horizon~risk)}.
	\end{flalign*}
	Since $V_\lambda^*$ is the minimizer over all feasible policies, we have $V_\lambda^*\le\mathbb{F}^{\rm nest}_{\lambda}[C_1(\Pi_\lambda),...,C_S(\Pi_\lambda)].$
	As seen in the proof of Proposition \ref{prop:suff} in Appendix \ref{app:prop1}, 
	under Conditions 1--6, there exists a finite penalty coefficient $\lambda<\infty$ such that the optimal value of the original problem coincides with that of penalized problem, that is, $V_\lambda^*=V^*.$
	Then, applying conditional consistency to the sequence of $\{C_{s}(\Pi_\lambda)\}_{s=1}^S$ yields $\mathbb{F}^{\rm nest}_\lambda[C_{1}(\Pi_\lambda),...,C_{S}(\Pi_\lambda)]=\mathbb{F}^{\rm EOH}_\lambda[C_{1}(\Pi_\lambda),...,,C_{S}(\Pi_\lambda)]$;
	accordingly, we have 
	\begin{flalign}
		\label{eq:upper_bounds}
		&V^*=V_\lambda^*\le\mathbb{F}^{\rm nest}_\lambda[C_{1}(\Pi_\lambda),...,C_{S}(\Pi_\lambda)]=\mathbb{F}^{\rm EOH}_\lambda[C_{1}(\Pi_\lambda),...,,C_{S}(\Pi_\lambda)].&
	\end{flalign}
	
	We next derive the upper confidence bound for $\mathbb{F}^{\rm EOH}_\lambda[C_{1}(\Pi_\lambda),...,,C_{S}(\Pi_\lambda)]$.
	The total random cost incurred when implementing the policy derived from the current approximation of the value function is denoted as $\Psi=C_{1}(\Pi_\lambda)+\cdots+C_{S}(\Pi_\lambda)$.
	Assume that $\mathbb{F}[\Psi] = h\left(\mathbb{E}[g(\Psi)]\right)$,
	where $h(\cdot)$ is nondecreasing, and that $W:=g(\Psi)$ has finite variance.
	Let $\{\Psi^m\}_{m=1}^M$ be i.i.d.\ samples of $\Psi$ obtained from $M$
	independent Monte Carlo simulations of the SDDP forward pass, and define $W^m := g(\Psi^m)$.
	Let $\mu_W := \frac{1}{M}\sum_{m=1}^M W^m$ and $\sigma_W^2 := \frac{1}{M-1}\sum_{m=1}^M (W^m-\mu_W)^2$.
	For a large $M$, based on the Central Limit Theorem,
	$\mu_W$ has approximately a normal distribution, and $\left[\mu_W-z_{\alpha}\frac{\sigma_W}{\sqrt{M}}, \mu_W+z_{\alpha}\frac{\sigma_W}{\sqrt{M}}\right]$ 
	gives an approximate $100(1-\alpha)$\% confidence interval of $\mathbb{E}\left[W\right]$.
	Finally, based on the nondecreasing property of $h(\cdot)$, 
	we conclude that $h(\mu_W+z_{\alpha}\frac{\sigma_W}{\sqrt{M}})$
	is an asymptotic $100(1-\alpha)\%$ upper confidence bound for $F_\lambda^{\rm EOH}[C_1(\Pi_\lambda),...,C_S(\Pi_\lambda)]$.
	Since $V^* \le F_\lambda^{\rm EOH}[C_1(\Pi_\lambda),...,C_S(\Pi_\lambda)]$ from Eq.~\eqref{eq:upper_bounds}, the same expression is also an asymptotic
	upper confidence bound for the original optimal value $V^*$.
\end{proof}
}

\subsection{The proof of Proposition \ref{prop:suff3}: Almost sure convergence of upper bounds}
\label{app:prop3}

{
\begin{proof}
	We prove Proposition \ref{prop:suff3} by showing that the dynamic programming equations of the dual problem
	associated with $\mathbb{ECRM}$ objective with mean-$\mathbb{CV@R}_\alpha$ {satisfy the three structural requirements
    under which \citet{guigues2023duality} prove that Dual SDDP returns a non-increasing deterministic upper bound (Appendix~\ref{app:dual_sddp}).}	
	Consider a MSSP \eqref{eq:MSSP} with $\mathbb{ECRM}$ objective with conditional mean-$\mathbb{CV@R}_\alpha$ given by
	\begin{subequations}
		\label{eq:MSSP_ECRM}
		\begin{flalign}
			\nonumber	&\min_{\{\bm{x}_1,...,\bm{x}_S\}}\bm{c}_1^{\top}\bm{x}_1+(1-\beta_2)\mathbb{E}_{\bm{\xi}_1}\left[\bm{c}_2^{\top}\bm{x}_2\right]+\beta_2\mathbb{CV@R}_{\alpha,\bm{\xi}_1}\left[\bm{c}_2^{\top}\bm{x}_2\right]&\\
			\nonumber
			&~~~~~~~~~~~~~+\mathbb{E}_{\bm{\xi}_1}\Biggl[(1-\beta_3)\mathbb{E}_{\bm{\xi}_2}\left[\bm{c}_3^{\top}\bm{x}_3\right]+\beta_3\mathbb{CV@R}_{\alpha,\bm{\xi}_2}\left[\bm{c}_3^{\top}\bm{x}_3\right]+\mathbb{E}_{\bm{\xi}_2}\Biggl[\cdots&\\ &~~~~~~~~~~~~~+\mathbb{E}_{\bm{\xi}_{S-2}}\Biggl[(1-\beta_S)\mathbb{E}_{\bm{\xi}_{S-1}}\left[\bm{c}_S^{\top}\bm{x}_S\right]+\beta_S\mathbb{CV@R}_{\alpha,\bm{\xi}_{S-1}}\left[\bm{c}_S^{\top}\bm{x}_S\right]\Biggr]\Biggr]\Biggr],&\\
			&{\rm s.t.~~~}\bm{A}_s\bm{x}_s+\bm{B}_s\bm{x}_{s-1}=\bm{b}_s,~{\rm where~}\bm{x}_0 {\rm~and~} \bm{\xi}_1{\rm~are~given}~~~\label{eq:MSSP_ECRM_Transition}&\forall s.
		\end{flalign}
	\end{subequations}

    Following the definition of $\mathbb{CV@R}$ in Eq.~\eqref{eq:CVaR} and linearizing the max operator in $\mathbb{CV@R}$, the optimization problem formulation \eqref{eq:MSSP_ECRM} is equivalent to the following formulation: 
	\begin{subequations}
		\label{eq:MSSP_ECRM_Epi}
		\begin{flalign}
			\nonumber	&\min_{\{\bm{x}_1,...,\bm{x}_S\},\{\eta_1,...,\eta_{S-1}\},\{v_2,...,v_S\}}\bm{c}_1^{\top}\bm{x}_1 + \beta_2\eta_1
			+\mathbb{E}_{\bm{\xi}_1}\left[(1-\beta_2)\bm{c}_2^{\top}\bm{x}_2 + \beta_3\eta_2 + \frac{\beta_2}{1-\alpha}v_2\right]+&\\
			\nonumber
			&~~~~~~~~~~~~~~~~~~~~~~~~~~~~~~~~~~~~~~~~~~~~~~~~~~~~~~~~~~~~~~~~~~~~~~~~~~+\mathbb{E}_{\bm{\xi}_2}\Biggl[(1-\beta_3)\bm{c}_3^{\top}\bm{x}_3 + \beta_4\eta_3 + \frac{\beta_3}{1-\alpha}v_3+\mathbb{E}_{\bm{\xi}_3}\Biggl[\cdots&\\ &~~~~~~~~~~~~~~~~~~~~~~~~~~~~~~~~~~~~~~~~~~~~~~~~~~~~~~~~~~~~~~~~~~~~~~~~~~+\mathbb{E}_{\bm{\xi}_{S-1}}\Biggl[(1-\beta_S)\bm{c}_S^{\top}\bm{x}_S + \frac{\beta_S}{1-\alpha}v_S \Biggr]\Biggr]\Biggr],&
		\end{flalign}
		\vspace{-4ex}
		\begin{flalign}
			&{\rm s.t.~~~} v_s\ge \bm{c}_s^{\top}\bm{x}_s-\eta_{s-1}~~~~~~~~~~~~~\label{eq:MSSP_ECRM_Epi_ineq}&\forall s\in\{2,...,S\},\\
            &{ {~~~~~~~~} v_s\ge0~~~~~~~~ }\label{eq:MSSP_ECRM_nonnegative}&\forall s\in\{2,...,S\},
		\end{flalign}
		\vspace{-6ex}
		\begin{flalign}
			\nonumber
			&~~~~~~~~~~~~~~~~~~~~~~~~~~~~~~~~~~~~~~~~~~~~~~~~~~~~~~~~~~~~~~~~~~~~~~~~~~~~~~~~~~~~~~~~~~~~~~~~~~~~~~~~~~~~~~~~~~~~~~~~~~~~~~~~~~~~~~~~~~~~~~~~~~~~~~~~~{\rm and~Eq.~\eqref{eq:MSSP_ECRM_Transition}},&
		\end{flalign}
	\end{subequations}
    {where an optimal solution of $\{\eta^*_s\}_{s=1}^{S-1}$ and $\{v^*_s\}_{s=2}^{S}$, respectively, represent the $\alpha$-quantile of the stage cost $\bm{c}^{\top}_s \bm{x}^*_s$ and its excess (i.e., $v^*_s=\max\{\bm{c}_s^\top \bm{x}_s-\eta^*_{s-1},0\}$).
    Problem~\eqref{eq:MSSP_ECRM_Epi} is a risk-neutral MSSP in the state $(x_s,\eta_s)$ augmented with the continuous stage variables $v_s$ and the additional linear constraints~\eqref{eq:MSSP_ECRM_nonnegative} and \eqref{eq:MSSP_ECRM_Transition}.}   
    
	Applying the tower property and strong monotonicity of the expectation operator to Eq. \eqref{eq:MSSP_ECRM_Epi}, we obtain
	\begin{flalign}
		\nonumber	&\min_{\{\bm{x}_1,...,\bm{x}_S\},\{\eta_1,...,\eta_{S-1}\},\{v_2,...,v_S\}}\mathbb{E}\Biggl[\bm{c}_1^{\top}\bm{x}_1 + \beta_2\eta_1
		+\left\{(1-\beta_2)\bm{c}_2^{\top}\bm{x}_2 + \beta_3\eta_2 + \frac{\beta_2}{1-\alpha}v_2\right\}+&\\
		\nonumber
		&~~~~~~~~~~~~~~~~~~~~~~~~~~~~~~~~~~~~~~~~~~~~~~~~~~~~~~~~~~~~~~~~~~~~~~~~~~~~~~~+\left\{(1-\beta_3)\bm{c}_3^{\top}\bm{x}_3 + \beta_4\eta_3 + \frac{\beta_3}{1-\alpha}v_3\right\}+\cdots&\\
        \label{eq:MSSP_ECRM_Epi_EOH_cost}
        &~~~~~~~~~~~~~~~~~~~~~~~~~~~~~~~~~~~~~~~~~~~~~~~~~~~~~~~~~~~~~~~~~~~~~~~~~~~~~~~+\left\{(1-\beta_S)\bm{c}_S^{\top}\bm{x}_S + \frac{\beta_S}{1-\alpha}v_S\right\} \Biggr],&\\
        \nonumber
		&{\rm s.t.~~~} {\rm Eqs.~\eqref{eq:MSSP_ECRM_Transition},~\eqref{eq:MSSP_ECRM_Epi_ineq},~and~\eqref{eq:MSSP_ECRM_nonnegative}}.&
	\end{flalign}
    Linearization of the max operator in $\mathbb{CV@R}$ adds only the linear objective terms $\beta_s\eta_{s-1}$, $\frac{\beta_s}{1-\alpha}v_s$ and the linear constraints~\eqref{eq:MSSP_ECRM_Epi_ineq} and \eqref{eq:MSSP_ECRM_nonnegative}.
    Since the stage costs shown in Eq.~\eqref{eq:MSSP_ECRM_Epi_EOH_cost}, 
    $\bm{c}^{\top}_1 \bm{x}^*_1 + \beta_2\eta_1$,~
    $(1-\beta_2)\bm{c}^{\top}_2 \bm{x}^*_2 + \beta_3\eta_2+\frac{\beta_2}{1-\alpha}v_2$,...,
    $(1-\beta_S)\bm{c}^{\top}_S \bm{x}^*_S+\frac{\beta_S}{1-\alpha}v_S$,
    are linear and the newly added constraints \eqref{eq:MSSP_ECRM_Epi_ineq} and \eqref{eq:MSSP_ECRM_nonnegative} are both linear inequalities,
    the strong duality theorem holds for Eqs.~\eqref{eq:MSSP_ECRM_Epi_EOH_cost}.
    Dualization of Eqs.~\eqref{eq:MSSP_ECRM_Transition},~\eqref{eq:MSSP_ECRM_Epi_ineq},~and \eqref{eq:MSSP_ECRM_nonnegative} lead to the following dual problem
	\begin{subequations}
		\label{eq:Dual_MSSP_ECRM}
		\begin{flalign}
			&\max_{\{\bm{\pi}_1,...,\bm{\pi}_S\},\{{\upsilon}_2,...,{\upsilon}_S\}}\mathbb{E}\left[\left(\bm{b}_1-\bm{B}_1\bm{x}_0\right)^{\top}\bm{\pi}_1+(\bm{b}_2)^{\top}\bm{\pi}_2+\cdots+\bm{b}_S^{\top}\bm{\pi}_S\right],&\\
            \label{eq:Dual_MSSP_ECRM_ineq_S}
			&{\rm s.t.~~~}\bm{A}^{\top}_{S}\bm{\pi}_{S} \le (1-\beta_S + \upsilon_S)\bm{c}_{S},& \\
            \label{eq:Dual_MSSP_ECRM_ineq_1}
			&~~~~~~~~\bm{A}^{\top}_{1}\bm{\pi}_{1} + \mathbb{E}_{\bm{\xi_{1}}}[\bm{B}_2^{\top}\bm{\pi}_2]\le \bm{c}_{1},~{\rm where~}\bm{\xi}_1{\rm~is~given},&\\
            \label{eq:Dual_MSSP_ECRM_ineq_s}
			&~~~~~~~~\bm{A}^{\top}_{s-1}\bm{\pi}_{s-1} + \mathbb{E}_{\bm{\xi_{s-1}}}[\bm{B}_s^{\top}\bm{\pi}_s]\le (1-\beta_{s-1} + \upsilon_{s-1})\bm{c}_{s-1},&\forall s\in\{3,...,S\},\\
            \label{eq:Dual_MSSP_ECRM_eq}
			&~~~~~~~~\beta_s - \mathbb{E}_{\bm{\xi}_{s-1}}[\upsilon_s] = 0&\forall s\in\{2,...,S\},\\
            \label{eq:Dual_MSSP_ECRM_Epi}
			&~~~~~~~~0\le\upsilon_s\le\frac{\beta_s}{1-\alpha}&\forall s\in\{2,...,S\},	
		\end{flalign}
	\end{subequations}
    where $\{\bm{\pi}_s\}_{s=1}^{S}$ and $\{{\upsilon}_s\}_{s=2}^{S}$
    are dual variables corresponding to Eqs.~\eqref{eq:MSSP_ECRM_Transition} and \eqref{eq:MSSP_ECRM_Epi_ineq}, respectively.
    
	Using Assumption \ref{app:assumption_finite_randomness} and \ref{app:assumption_independence}, we can write the dynamic programming equations of Eq. (\ref{eq:Dual_MSSP_ECRM}) as follows:
	\begin{subequations}
		\label{eq:Dual_DP_ECRM}
		\begin{flalign} 
			\nonumber
			&\underline{1{\rm st~stage~subproblem}}&\\
			\label{eq:Dual_DP3_ECRM}
			&\max_{\bm{\pi}_1}\left(\bm{b}_1-\bm{B}_1\bm{x}_0\right)^{\top}\bm{\pi}_1 + V_{2}(\bm{\pi}_1),&	
		\end{flalign} 
		\vspace{-4ex}
		\begin{flalign}
			\nonumber
			&\underline{2{\rm nd~stage~subproblem}}&\\
			\label{eq:Dual_DP2_5_ECRM}
			&V_{2}(\bm{\pi}_{1})=\max_{\{\bm{\pi}^1_2,...,\bm{\pi}^{N_2}_s\},\{\upsilon_{2}^{1},...,\upsilon_{2}^{N_2}\}}\sum_{n=1}^{N_2}p_2^n\left[(\bm{b}^n_2)^{\top}\bm{\pi}^n_2 + V_{3}(\bm{\pi}_2^n,\upsilon_2^n)\right],&\\
			&(\bm{A}_{1})^{\top}\bm{\pi}_{1} + \sum_{n=1}^{N_2}p_2^n[(\bm{B}^n_2)^{\top}\bm{\pi}^n_2]\le \bm{c}_{1},~{\rm where~}\bm{\xi}_1{\rm~is~given},&\\
			&\beta_2 - \sum_{n=1}^{N_2}p_2^n\upsilon^n_2 = 0,&\\
			&0\le\upsilon^n_2\le\frac{\beta_2}{1-\alpha}&\forall n\in\{1,...,N_{2}\},
		\end{flalign}
		\vspace{-2ex}
		\begin{flalign}
			\nonumber
			&\underline{s\mathchar`-{\rm th~stage~subproblem}}&\\
			\nonumber
			&{V_{s}(\bm{\pi}_{s-1}^{n'},\upsilon_{s-1}^{n'})=\max_{\{\bm{\pi}^1_s,...,\bm{\pi}^{N_s}_s\},\{\upsilon_{s}^{1},...,\upsilon_{s}^{N_s}\}}\sum_{n=1}^{N_s}p_s^n\left[(\bm{b}^n_s)^{\top}\bm{\pi}^n_s + V_{s+1}(\bm{\pi}_s^n,\upsilon_s^n)\right]}&
		\end{flalign} 
		\vspace{-2ex}
		\begin{flalign}
			\label{eq:Dual_DP2_ECRM}
			&&\forall n'\in\{1,...,N_{s-1}\}, s\in\{3,...,S-1\},
		\end{flalign} 
		\vspace{-4ex}
		\begin{flalign}
			&(\bm{A}^{n'}_{s-1})^{\top}\bm{\pi}^{n'}_{s-1} + \sum_{n=1}^{N_s}p_s^n[(\bm{B}^n_s)^{\top}\bm{\pi}^n_s]\le (1-\beta_{s-1}+\upsilon^{n'}_{s-1})\bm{c}^{n'}_{s-1}~~~&\forall n'\in\{1,...,N_{s-1}\}, s\in\{3,...,S-1\},\\
			&\beta_s - \sum_{n=1}^{N_s}p_s^n\upsilon^n_s = 0~~~~~~&\forall s\in\{3,...,S-1\},\\
			&0\le\upsilon^n_s\le\frac{\beta_s}{1-\alpha}~~~~~~~~~	&\forall n\in\{1,...,N_{s}\}, s\in\{3,...,S-1\},
		\end{flalign}
		\vspace{-4ex}
		\begin{flalign}
			\nonumber
			&\underline{S\mathchar`-{\rm th~stage~subproblem}}&\\
			\label{eq:Dual_DP1_ECRM}
			&{V_{S}(\bm{\pi}_{S-1}^{n'},\upsilon_{S-1}^{n'})}=\max_{\{\bm{\pi}^1_S,...,\bm{\pi}^{N_S}_S\},\{\upsilon_{S}^{1},...,\upsilon_{S}^{N_S}\}}\sum_{n=1}^{N_S}p_S^n\left[(\bm{b}^n_S)^{\top}\bm{\pi}^n_S\right]&\forall n'\in\{1,...,N_{S-1}\},\\
			&(\bm{A}^n_{S})^{\top}\bm{\pi}^n_{S} \le (1-\beta_S+\upsilon^n_S)\bm{c}^n_{S}&\forall n\in\{1,...,N_S\},\\
			&(\bm{A}^{n'}_{S-1})^{\top}\bm{\pi}^{n'}_{S-1} + \sum_{n=1}^{N_S}p_S^n[(\bm{B}^n_S)^{\top}\bm{\pi}^n_S]\le (1-\beta_{S-1}+\upsilon^{n'}_{S-1})\bm{c}^{n'}_{S-1}&\forall n'\in\{1,...,N_{S-1}\},\\
			&\beta_S - \sum_{n=1}^{N_S}p_S^n\upsilon^n_S = 0,&\\
			&0\le\upsilon^n_S\le\frac{\beta_S}{1-\alpha}&\forall n\in\{1,...,N_{S}\},
		\end{flalign} 
	\end{subequations}
    {where $\bm{\pi}_s^n$ and ${\upsilon}_s^n$ represent dual variables corresponding to $\bm{\xi}_s^n$ and
    $V_s(\bm{\pi}^n_{s-1},{\upsilon}^n_{s-1})$ denotes the value function with respect to the pair of state variables $(\bm{\pi}^n_{s-1},{\upsilon}^n_{s-1})$.
    For the stage $s=2$, since state variables include only $\bm{\pi}^n_{s-1}$ and $\xi_1$ is given, the value function is denoted as $V_2(\bm{\pi}_{1})$.

    We now verify that the augmented dual DP in Eq.~\eqref{eq:Dual_DP_ECRM} satisfies the requirements for applying Dual SDDP reviewed in Appendix~A.3.
    The dual DP with $\mathbb{ECRM}$ objective with mean-$\mathbb{CV@R}_\alpha$ modifies the objective function and introduces additional dual variables and constraints.
    The argument below shows that these changes do not violate the structural properties used by \citet{guigues2023duality}.
    
    First, all subproblems in Eq.~\eqref{eq:Dual_DP_ECRM} remain linear programs.
    The additional dual variables $\upsilon_s^n$ introduced by the linearization of the max operator in $\mathbb{CV@R}$ are continuous, and they enter the dual DP only through linear constraints~\eqref{eq:Dual_MSSP_ECRM_ineq_S},\eqref{eq:Dual_MSSP_ECRM_ineq_s},\eqref{eq:Dual_MSSP_ECRM_eq}, and \eqref{eq:Dual_MSSP_ECRM_Epi}.
    Therefore, the value function $V_{s}$ is piecewise concave with respect to the augmented dual state. 
    For stages $s\geq 3$, this state is the pair $(\bm{\pi}_{s-1}^{n'},\upsilon_{s-1}^{n'})$, and the Benders cuts are constructed in this augmented state space.
    As long as $\upsilon_s^n$ are defined in a compact space, the usual cut-generation argument applies to the augmented value function $V_s(\bm{\pi}_{s-1}^{n'},\upsilon_{s-1}^{n'})$. 
    
    Second, the dual DP~\eqref{eq:Dual_DP_ECRM} does not introduce an additional boundedness issue in the first-stage subproblem.
    As reviewed in Appendix~A.3, $\bm{\pi}_{1}$ is not bounded in the risk-neutral dual DP.
    \citet{guigues2023duality} show that, if the primal problem has a finite optimal value, finite box constraints can be imposed on $\bm{\pi}_{s}$ without changing the dual optimal value.
    The newly introduced variables, $\upsilon_s^n$, require no analogous bounding argument.
    They are already restricted by Eq.~\eqref{eq:Dual_MSSP_ECRM_Epi}, $0 \leq \upsilon_s^n \leq \frac{\beta_s}{1-\alpha}$, where $\alpha\in[0,1)$ and $\beta_s\in[0,1]$.
    Hence, their feasible domain is compact. 
    No incoming $\upsilon_s^n$ appears in the first-stage subproblem, and $\upsilon_s^n$ generated in later stages are bounded by the above constraints.
    
    Third, as in the risk-neutral case, the augmented dual DP may violate relatively complete recourse. 
    We handle this in the same way as \citet{guigues2023duality}: nonnegative slack variables are introduced in the dual inequality constraints \eqref{eq:Dual_MSSP_ECRM_ineq_S}--\eqref{eq:Dual_MSSP_ECRM_ineq_s} and penalized in the objective.
    The only additional structure introduced by the augmented dual DP is the stagewise constraints \eqref{eq:Dual_MSSP_ECRM_eq} and \eqref{eq:Dual_MSSP_ECRM_Epi}.
    These constraints do not obstruct the slack approach for two reasons.
    First, Eqs.~\eqref{eq:Dual_MSSP_ECRM_eq} and \eqref{eq:Dual_MSSP_ECRM_Epi} involve only the stage $s$ variables $\{\upsilon_s^n\}_{n=1}^{N_s}$ and contain no incoming state $\upsilon_{s-1}^{n'}$.
    Second, $\upsilon_s^n=\beta_s$ for all $n$ always satisfies both additional constraints.
    Therefore, Eqs.~\eqref{eq:Dual_MSSP_ECRM_eq} and \eqref{eq:Dual_MSSP_ECRM_Epi} do not create an additional source of infeasibility.
    Appending slack variables to the inequality constraints \eqref{eq:Dual_MSSP_ECRM_ineq_S}--\eqref{eq:Dual_MSSP_ECRM_ineq_s} alone restores relatively complete recourse
    while leaving additional constraints \eqref{eq:Dual_MSSP_ECRM_eq} and \eqref{eq:Dual_MSSP_ECRM_Epi} to be satisfied,
    and Theorem~4.1 of \citet{guigues2023duality} yields almost sure convergence of the deterministic upper bounds to the optimum in finitely many iterations.}
\end{proof}
}

%% file: chxxxxx.tex
\section{Original formulation of the SAV and SAV-BRT Systems}
\label{app:OriginalFormulation}

\subsection{SAV formulation \citep{seo2022multi}}
\label{app:SAV}

\citet{seo2022multi} proposed a multi-objective linear optimization problem for strategic infrastructure planning of SAV systems that explicitly and endogenously considers vehicle operations.
The SAV and traveler flows in the model are described by a macroscopic dynamic traffic assignment (DTA) framework.
The decisions include SAVs' routing (including empty vehicles and occupied vehicles), vehicle-traveler assignment, total number of SAVs, link capacity, and storage capacity.
The multiple objective functions include total travel time of travelers, total distance traveled by SAVs, total number of SAVs, and total infrastructure construction cost.
The multi-objective linear optimization problem is reduced to the corresponding single-objective linear optimization using the weighted sum method.
The optimization problem below extends the original model by including travelers’ schedule costs in its objective function:
\begin{subequations}
	\begin{flalign}
		&\min T + S + D + N + C,&\\
		&{\rm s.t.~~} T=\sum_t\bm{f}^{\top}\bm{y}^t,&\\
		&~~~~~~~ S=\epsilon\sum_{t< l}\bm{1}^{\top}\bm{q}^t+\lambda\sum_{t>l}\bm{1}^{\top}\bm{q}^t,&\\
		&~~~~~~~ D=\sum_t\bm{d}^{\top}\bm{z}^t,&\\
		&~~~~~~~ N=d\bm{1}^{\top}\hat{\bm{z}},&\\
		&~~~~~~~ C=\bm{c}^{\top}\bm{\kappa},&\\
		\label{app:SAV_dynamics}
		&~~~~~~~\sum_j z_{ji}^{t-t_{ji}}=\sum_j z_{ij}^{t}- \delta_{t1} \hat{z}_{i}& \forall t,i, \\
		\label{app:SAV_passenger_flow_dynamics}
		&~~~~~~~\sum_j y_{d,ji}^{l,t-t_{ji}}=\sum_j y_{d,ij}^{l,t} + \delta_{id} q_{d}^{l,t}& \forall t,d,l,i, \\
		\label{app:SAV_passenger_arivals}
		&~~~~~~~\sum_t q_{d}^{l,t} = \sum_o \overline{Q}_{od}^{l}&\forall d,l,\\
		\label{app:SAV_vehicle_capacity}
		&~~~~~~~\sum_{d,l}{y}_{d,ij}^{l,t}\le \rho{z}_{ij}^t &\forall t,ij,,\\
		\label{app:SAV_network_capacity}
		&~~~~~~~{z}_{ij}^t\le {\kappa}_{ij} &\forall t,ij,
	\end{flalign}
\end{subequations}
where all the variables are nonnegative and continuous.
The notations in the scalar form are listed in Tables \ref{tab:variable} and \ref{tab:parameter}.

%% file: chxxxxxx.tex
\subsection{SAV-BRT formulation \citep{maruyama2023integrated}}
\label{app:SAV_BRT}

\citet{maruyama2023integrated} proposed a multi-objective optimization framework for SAV-BRT systems that leverages the flexibility of SAVs and the mass transit capability of BRT.
Their model builds upon the SAV model developed by \citet{seo2022multi}, with additional components for BRT route design and scheduling.
In the model, the continuous flows of SAVs and travelers are represented by the DTA framework, while BRT decisions are modeled as integer variables.
The optimization problem below extends the original model by including travelers' schedule costs in its objective function:
\begin{subequations}
	\begin{flalign}
		&\min T + S +  D + N  + C + {G}&\\
		&{\rm s.t.~~} T=\sum_t\bm{f}^{\top}\bm{y}^t+\sum_t\tilde{\bm{f}}^{\top}\tilde{\bm{y}}^t,&\\
        &~~~~~~~ S=\epsilon\sum_{t< l}\bm{1}^{\top}{\bm{q}}^t+\lambda\sum_{t>l}\bm{1}^{\top}{\bm{q}}^t+\epsilon\sum_{t< l}\bm{1}^{\top}\tilde{\bm{q}}^t+\lambda\sum_{t>l}\bm{1}^{\top}\tilde{\bm{q}}^t,&\\
		&~~~~~~~ D=\sum_t\bm{d}^{\top}\bm{z}^t+\sum_t\tilde{\bm{d}}^{\top}\sum_m\bm{w}^{m,t},&
	\end{flalign}
	\vspace{-4ex}
	\begin{flalign}
		&~~~~~~~ N=d\bm{1}^{\top}\hat{\bm{z}}+\tilde{d}\bm{1}^{\top}\sum_{t,m}\hat{\bm{w}^{m,t}},&\\
		&~~~~~~~ C=\bm{c}^{\top}\bm{\kappa},&\\
		&~~~~~~~ G=(\bm{c}^{\rm BRT})^{\top}\sum_m\bm{g}^m,&
	\end{flalign}
	\vspace{-4ex}
	\begin{flalign}
		&~~~~~~~ \sum_i{a}_{s}^{m}= 1& \forall m, \\
		&~~~~~~~\sum_i{b}_{s}^{m}= 1& \forall m, \\
		\label{app:BRT_SAV_vehicle_flow_dynamics}
		&~~~~~~~ \sum_j z_{ji}^{t-t_{ji}}=\sum_j z_{ij}^{t} - \delta_{t1} \hat{z}_{i}& \forall t,i, \\
		\label{app:BRT_BRT_vehicle_flow_dynamics}
		&~~~~~~~ \sum_j w_{ji}^{m,t-\tilde{t}_{ji}}=\sum_j w_{ij}^{m,t} - \hat{w}_{0i}^{m,t} + \hat{w}_{i0}^{m,t}& \forall m,t,i, \\
		\label{app:BRT_SAV_passenger_flow_dynamics}
		&~~~~~~~ v_{d,i}^{l,t-\tilde{t}_{ii}} + \sum_j y_{d,ji}^{l,t-t_{ji}}=\hat{v}_{d,i}^{l,t} + \sum_j y_{d,ij}^{l,t} + \delta_{id} q_{d}^{l,t}& \forall t,d,l,i,, \\
		\label{app:BRT_BRT_passenger_flow_dynamics}
		&~~~~~~~ \hat{v}_{d,i}^{l,t-\tilde{t}_{ii}} + \sum_j \tilde{y}_{d,ji}^{l,t-\tilde{t}_{ji}}={v}_{d,i}^{l,t} + \sum_j \tilde{y}_{d,ij}^{l,t} + \delta_{id} \tilde{q}_{d}^{l,t}& \forall t,d,l,i, \\
		\label{app:BRT_SAV_passenger_arivals}
		&~~~~~~~ \sum_t q_{d}^{l,t} + \sum_t \tilde{q}_{d}^{l,t} = \sum_o Q_{od}^{l} &\forall d,l,\\
		\label{app:BRT_SAV_vehicle_capacity}
		&~~~~~~~ \sum_{d,l}{y}_{d,ij}^{l,t}\le \rho{z}_{ij}^t &\forall t,ij,\\
		\label{app:BRT_BRT_vehicle_capacity}
		&~~~~~~~ \sum_{d,l}{\tilde{y}}_{d,ij}^{l,t}\le \sum_m\tilde{\rho}{w}_{ij}^{t,m} &\forall t,ij,\\
		\label{app:BRT_SAV_network_capacity}
		&~~~~~~~ {z}_{ij}^t\le {\kappa}_{ij} - \tilde{\kappa}_{ij}g_{ij}^m &\forall m,t,ij,\\
		\label{app:BRT_BRT_route}
		&~~~~~~~ 0\le{w}_{ij}^{m,t}\le {g}_{ij}^m &\forall m,t,ij,\\
		\label{app:BRT_BRT_start}
		&~~~~~~~ 0\le{w}_{0i}^{m,t}\le {a}_{i}^m &\forall m,t,i,\\
		\label{app:BRT_BRT_end}
		&~~~~~~~ 0\le{w}_{i0}^{m,t}\le {b}_{i}^m &\forall m,t,i,\\
		\label{app:BRT_subtour}
		&~~~~~~~ {w}_{ij}^{m,1}=0 &\forall m,ij,\\
		\label{app:BRT_binary}
		&~~~~~~~\bm{a}\in\{\bm{0},\bm{1}\},\bm{b}\in\{\bm{0},\bm{1}\},\bm{g}\in\{\bm{0},\bm{1}\},\bm{w}_s\in\{\bm{0},\bm{1}\},\bm{\hat{w}}\in\{\bm{0},\bm{1}\},&
	\end{flalign}
\end{subequations}
where all the variables are nonnegative.
The notations in scalar form are listed in Tables \ref{tab:variable2} and \ref{tab:parameter2}.

\section{The Proof of Proposition 4}
\label{app:prop4}

\begin{proof}
	Under Conditions 4--9 in Proposition \ref{prop:suff3}, the duality of risk-neutral MSSPs results in the general duality theorem of LP \citep{shapiro2009lectures}.
	A part of the optimality conditions of Eq. (\ref{eq:SAV-NR-NDP}) are given by
	\begin{flalign}
		\label{eq:strategic_kap_complementary}
		&\begin{cases}
			-\theta_{s,ij} + \hat{c}_{s,ij} \ge 0 & {\rm if~~} \hat{\kappa}_{s,ij}=0 \\
			-\theta_{s,ij} + \hat{c}_{s,ij} = 0 & {\rm if~~} \hat{\kappa}_{s,ij}\ge0
		\end{cases}&\forall ij,s\in\mathcal{S}^{\rm F},\\
		&\begin{cases}
			\theta_{s-1,ij} - \mathbb{E}_{\bm{\xi}^{\rm F}}[\theta_{s,ij}] \ge 0 & {\rm if~~} {\kappa}_{s,ij}=0 \\
			\theta_{s-1,ij} - \mathbb{E}_{\bm{\xi}^{\rm F}}[\theta_{s,ij}] = 0 & {\rm if~~} {\kappa}_{s,ij}\ge0
		\end{cases}&\forall ij,s\in\mathcal{S}^{\rm F},\\
		&\begin{cases}
			\theta_{s-1,ij} + {c}_{s-1,ij} - \mathbb{E}_{\bm{\xi}^{\rm G}}[\theta_{s,ij} + \sum_tp^{t}_{s,ij}] \ge 0 & {\rm if~~} {\kappa}_{s,ij}=0 \\
			\theta_{s-1,ij} + {c}_{s-1,ij} - \mathbb{E}_{\bm{\xi}^{\rm G}}[\theta_{s,ij} + \sum_tp^{t}_{s,ij}] = 0 & {\rm if~~} {\kappa}_{s,ij}\ge0
		\end{cases}&\forall ij,s\in\mathcal{S}^{\rm G},\\
		\label{eq:capacity_complementary}
		&\begin{cases}
			\kappa_{s,ij} - z^t_{s,ij} \ge 0 & {\rm if~~} {p}^t_{s,ij}=0 \\
			\kappa_{s,ij} - z^t_{s,ij} = 0 & {\rm if~~} {p}^t_{s,ij}\ge0
		\end{cases}&\forall ij,s\in\mathcal{S}^{\rm G},\\
		\label{eq:strategic_transition_kap}
		&{\kappa}_{s,ij} = {\kappa}_{s-1,ij} + \hat{\kappa}_{s,ij} &\forall ij,s\in\mathcal{S}^{\rm F},\\
		\label{eq:operational_transition_kap}
		&{\kappa}_{s,ij} = {\kappa}_{s-1,ij} &\forall ij,s\in\mathcal{S}^{\rm G}.
	\end{flalign}
	From Eqs. (\ref{eq:strategic_kap_complementary})--(\ref{eq:capacity_complementary}), for $\kappa^*_{s,ij}>0$, we obtain
	\begin{flalign}
		\nonumber
		&(c^{}_{1,ij}+\hat{c}^{}_{1,ij}+\hat{c}^{}_{3,ij})\kappa^*_{1,ij}+\mathbb{E}_{\xi_2^{\rm G}}\left[(c^{}_{3,ij}+\hat{c}_{3,ij}+\hat{c}_{5,ij})\kappa^*_{3,ij}+\mathbb{E}_{\xi_4^{\rm G}}\left[\cdots+\mathbb{E}_{\xi_{S-2}^{\rm G}}\left[c^{}_{S-1,ij}\kappa^*_{S-1,ij}\right]\right]\right]&
	\end{flalign}
	\vspace{-3.2ex}
	\begin{flalign}
		\label{eq:self_financing0}
		&=\mathbb{E}_{\xi_1^{\rm F}}\left[\sum_t p^{t,*}_{2,ij}z^{t,*}_{2,ij}+\mathbb{E}_{\xi_3^{\rm F}}\left[\cdots+\mathbb{E}_{\xi_{S-1}^{\rm F}}\left[\sum_t p^{t,*}_{S,ij}z^{t,*}_{S,ij}\right]\right]\right]~~~~~~~~~~~~~~~~~~~~~~~~~~~~~~~~~~~~~~~~~~~~~~~~~~~~~~~~~~~&\forall ij.	
	\end{flalign}
	Substituting Eqs. (\ref{eq:strategic_transition_kap}) and (\ref{eq:operational_transition_kap}) into Eq. (\ref{eq:self_financing0}) yields Eq. (\ref{eq:self_financing}).
\end{proof}

\section{Evaluation of Optimal Policies for the SAV and SAV-BRT Systems}
\label{app:TotalCost}

\begin{figure}[t]
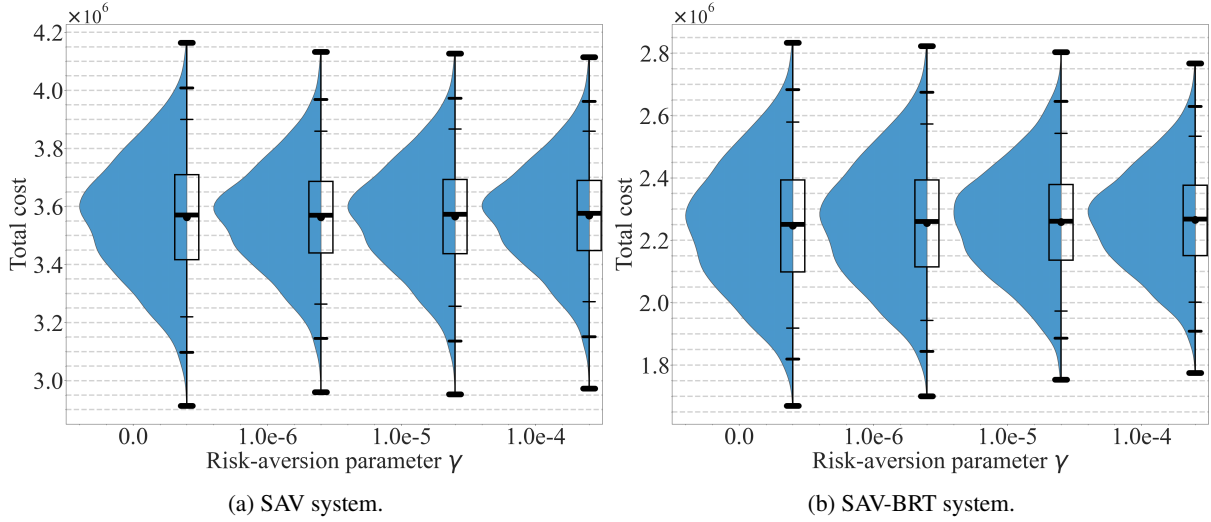

	\centering
	\begin{subfigure}{0.48\textwidth}
		\centering
		\includegraphics[width=1.0\textwidth]{TotalCostSAV.pdf}
		\caption{SAV system.}
		\label{fig:TotalCostSAV}  
	\end{subfigure} 
	\begin{subfigure}{0.48\textwidth}
		\centering
		\includegraphics[width=1.0\textwidth]{TotalCostBRT.pdf}
		\caption{SAV-BRT system.}
		\label{fig:TotalCostBRT}  
	\end{subfigure} 
	\caption{Probability distributions of the total cost over the planning horizon for $\gamma\in\{0.0, 10^{-6}, 10^{-5}, 10^{-4}\}$.}
	\label{fig:TotalCost}  
\end{figure}

{Figure \ref{fig:TotalCost} illustrates the total cost over the planning horizon for risk-aversion parameters $\gamma\in\{0.0, 10^{-6}, 10^{-5}, 10^{-4}\}$ in the form of probability distributions derived from kernel density estimators, with overlaid box plots.
The horizontal axis denotes the risk-aversion parameter $\gamma$, and the vertical axis denotes the total cost.
The left and right panels of Figure \ref{fig:TotalCost} show the probability distributions of the total costs for the SAV system and the SAV-BRT system, respectively.
The whiskers of the box plots represent, from top to bottom, the upper 100th, 99th, and 95th percentiles of the sampled total cost.

Figure \ref{fig:TotalCost} shows that the risk-averse preference clearly influences the evaluation of the optimal policy.
For both systems, as $\gamma$ increases,
the expected total cost rises while the upper tail of the cost distribution is suppressed.
This represents the trade-off of risk-averse
policies such as hedging against severe performance degradation in adverse scenarios at the expense of average performance.

Furthermore, comparison with Table~\ref{tab:gap} indicates how strongly this risk aversion is expressed. 
For the SAV system, the lower bound of the optimal value at $\gamma=10^{-4}$, $4.105\times10^{6}$, exceeds the $99$th
percentile of the corresponding total-cost distribution. 
A similar relationship holds for the SAV-BRT system, whose lower bound at $\gamma=10^{-4}$, $2.716\times10^{6}$, likewise lies above the $99$th percentile of its distribution.
These results indicate that the chosen range
of $\gamma$ is sufficient to span risk-averse preferences. }